\documentclass[10pt]{amsart}
\usepackage{amsmath,amssymb,graphicx}
\usepackage{gafmath}
\usepackage{amsthm,stackrel,bm,centernot}
\usepackage{subcaption}
\usepackage{setspace}
\usepackage[alphabetic,initials]{amsrefs}
\usepackage[hypertexnames=true,pdfborderstyle={/S/U/W 1}]{hyperref}

\numberwithin{theorem}{section}
\numberwithin{equation}{section}
\numberwithin{figure}{section}
\hypersetup{ pdfauthor={G. Austin Ford, Andrew Hassell, and Luc Hillairet},
  pdftitle={} }
\newcommand{\boxop}{\ensuremath{\text{\qedsymbol}}} 

\newcommand{\geom}{\ensuremath{{\bm{\mathsf{G}}}}} 
\newcommand{\diff}{\ensuremath{{\bm{\mathsf{D}}}}} 
\newcommand\ZZ{\mathbb{Z}}

\DeclareMathOperator{\arccosh}{arccosh} 
\DeclareMathOperator{\Dom}{Dom} 
\DeclareMathOperator{\Diff}{Diff} 
\DeclareMathOperator{\Tr}{Tr} 

\newcommand{\bfW}{\ensuremath{\mathbf{W}}} 

\newcommand{\calD}{\ensuremath{\mathcal{D}}} 
\newcommand{\calU}{\ensuremath{\mathcal{U}}} 

\newcommand{\upc}{\ensuremath{\mathrm{c}}} 
\newcommand{\bmp}{\ensuremath{\bm{p}}} 

\newcommand{\bmq}{\ensuremath{\bm{q}}} 

\newcommand{\bmA}{\ensuremath{\bm{A}}} 
\newcommand{\bmE}{\ensuremath{\bm{E}}} 
\newcommand{\bmU}{\ensuremath{\bm{U}}} 

\newcommand{\bmLambda}{{\ensuremath{\bm{\Lambda}}}} 

\newcommand{\wbar}{{\ensuremath{\overline{w}}}} 

\newcommand\RR{\mathbb{R}}

\title[Wave propagation on an ESCS. I]{Wave propagation on Euclidean
  surfaces with conical singularities. I: Geometric diffraction.}

\author[G.A.~Ford]{G.~Austin Ford} \address{Department of Mathematics, Stanford
  University, Stanford, California 94305, USA}
\email{austin.ford@math.stanford.edu}

\author[A.~Hassell]{Andrew Hassell} \address{Mathematical Sciences Institute,
  Australian National University, Canberra 0200 ACT Australia}
\email{hassell@maths.anu.edu.au}

\author[L.~Hillairet]{Luc Hillairet} \address{Universit\'e d'Orl\'eans, UFR
  Sciences, B\^atiment de Math\'ematiques, 45067 Orl\'eans, France}
\email{luc.hillairet@univ-orleans.fr}

\thanks{This research was supported by National Science Foundation Postdoctoral Fellowship DMS-1204304 (GAF);
 Australian Research Council Discovery grants FT0990895, DP1095448 and DP120102019 (AH); and L'Agence Nationale de la Recherche program ANR-13-BS01-0007-01 GERASIC (LH).  GAF also thanks the Mathematical Sciences Institute of the Australian National University for its hospitality during part of this work.}



\begin{document}




\begin{abstract}
  We investigate the singularities of the trace of the half-wave group, $\Tr e^{-it\sqrt\Delta}$, on Euclidean surfaces with conical singularities $(X,g)$.  We compute the leading-order singularity associated to periodic orbits with successive degenerate diffractions.  This result extends the previous work of the third author \cite{Hil} and the two-dimensional case of the work of the first author and Wunsch \cite{ForWun} as well as the seminal result of Duistermaat and Guillemin \cite{DuiGui} in the smooth setting.  As an intermediate step, we identify the wave propagators on $X$ as singular Fourier integral operators associated to intersecting Lagrangian submanifolds, originally developed by Melrose and Uhlmann \cite{MelUhl}.
\end{abstract}

\maketitle

\tableofcontents


\setcounter{section}{-1}


\section{Introduction}
\label{sec:introduction}

In this article, we investigate the spectral geometry of Euclidean surfaces with conical singularities $(X,g)$. We determine the precise microlocal structure of the half-wave propagator, $e^{-it\sqrt{\Delta}}$, near a ray that undergoes one or two degenerate diffractions. Using this, we compute the leading-order singularity of the trace of the half-wave group, $\Tr e^{- it \sqrt\Delta}$, associated to an isolated periodic orbit undergoing two degenerate diffractions through cone points.  For example, if the periodic orbit has length $L$ and undergoes degenerate diffractions through two cone points at a distance $b$ apart, we show that the associated wave trace singularity is
\begin{equation}
  \frac{1}{4i\pi^2} \cdot \sqrt{b (L - b)} \cdot (t - L - i0)^{-1} .
  \label{2dd}
\end{equation}

\subsection{Background}
Spectral geometry typically aims at understanding the relations between the
spectrum of the Laplace operator on a Riemannian manifold and the geometry of
the associated geodesic flow. These relations may be revealed by the study of
wave propagation. For instance, the Poisson relation states that the trace of
the wave propagator is smooth except possibly at the lengths of periodic
orbits. Moreover, in a generic and smooth situation, the singularity that is
brought to the wave trace by a particular periodic orbit can be fully understood
and leads to the definition of the so-called {\em wave-invariants} (see
\cite{DuiGui}). These wave-invariants may then be used for instance in inverse
spectral problems. They also serve as a particular motivation to study
wave propagation on different kind of singular surfaces. We will focus on
Euclidean surfaces with conical singularities since this general setting
includes polygonal billiards and translation surfaces, both of which are very
interesting and natural.

The basic new feature of wave propagation on singular manifolds is the dichotomy
between waves that hit the singularity---that are then diffracted in all
possible directions---and waves that miss the singularity and propagate
according to the usual laws for smooth manifolds. This fact leads to the
definition of the so-called {\em geometric (or direct) front} that consists of
rays that miss the vertex and the {\em diffracted front} that consists of rays
that hit the vertex and are reemitted in all possible directions. On a two dimensional  cone,
these two fronts share two rays in common that correspond to the limit of rays
that nearly miss the cone point from one side or the other. In the literature,
these two rays are called either ``geometrically diffractive'' \cite{MelWun} or
``singular diffractive'' \cite{Hil}. We will use here the former terminology.
On a compact surface with conical singularities the situation becomes quickly
complicated for a diffractive ray may hit successive conical points and
experience new diffractions that may be singular and so on. These diffractive
phenomena are established in the quite abundant literature on wave propagation
on singular manifolds starting with Sommerfeld's result for Euclidean sectors or
cones \cite{Som}. Among the important milestones of this story are the studies
by Cheeger and Taylor for cones of exact product-type \cites{CheTay1,CheTay2}
and by Melrose and Wunsch in the general case \cite{MelWun}.

Over the years, there has been investigation of the impact of diffraction on the
wave-trace.  For instance, Wunsch showed in \cite{Wun} that singularities may
appear at length of periodic diffractive orbits.  For some periodic diffractive
geodesics, the leading singularity is then computed in \cite{Hil} in the
Euclidean case and in \cite{ForWun} in a more general case (see also
\cite{BogoPavSch} for related results from a physics perspective). Both these
results are built upon a precise description of the wave propagator that is
microlocalized in the vicinity of given periodic (possibly diffractive)
geodesic. However, none of these studies attempted to determine the precise microlocal nature of the propagator near a geometrically diffractive ray: in  \cite{ForWun}, it is assumed that no geometric diffraction occurs (with a
non-focusing assumption that would be automatically satisfied in our case), while in
\cite{Hil}, it is assumed that the periodic geodesic has at most one geometric
diffraction.
The main purpose of the present paper is to fill this gap, i.e., to give a precise microlocal description of the wave
propagator near the geometric diffractive rays, on an ESCS. More precisely, we
will identify the microlocalized propagator near a ray that undergoes one or two geometric diffractions as
an element of the Melrose-Uhlmann class of singular Fourier Integral Operators
(\cite{MelUhl}), associated to either two, or four, Lagrangian submanifolds.
One advantage of this identification is the ease of computing wave trace singularities,
such as \eqref{2dd}, using standard methods such as stationary phase.


This is the first article in a planned series of three.  In the second paper, we will show how to compute wave traces for any closed orbit on an ESCS (with any number of geometric diffractions). In the third paper, we will apply our results to inverse spectral results, specifically isospectral compactness in the class of ESCSs. To keep the length of the present paper within reasonable bounds, we restrict our attention here to at most two geometric diffractions.


\subsection{Cones and ESCSs}
The Euclidean cone of cone angle $\alpha > 0$ is the product manifold
$C_\alpha \defeq \left (\bbR_+ \right)_r \times \left( \bbR / \alpha \bbZ
\right)_\theta$ equipped with the exact warped product metric
\begin{equation*}
  ds^2 \defeq dr^2 + r^2 \, d\theta^2 .
\end{equation*}
The vertex of the cone $\bmp$ is the point where all $(0,\theta)$ are
identified, and we will denote by
$C_\alpha^\circ \defeq C_\alpha \setminus \{ \bmp \}$ the cone without its
vertex. Let us recall that the Euclidean distance on $C_\alpha$ between two
points $\bmq_1=(r_1,\theta_1)$ and $\bmq_2=(r_2,\theta_2)$ in polar coordinates
is:
\begin{equation}
  \begin{cases}
    \mbox{dist}(\bmp,\bmq_1) = r_1, & \\
    \mbox{dist}(\bmq_1,\bmq_2) = r_1+r_2, & |\theta_2 - \theta_2| > \pi \\
    \mbox{dist}(\bmq_1,\bmq_2) =
    \sqrt{r_1^2+r_2^2-2r_1r_2\cos(\theta_2-\theta_1)}, &
    |\theta_2-\theta_1| \leqslant \pi.\\
  \end{cases}
\end{equation}

A Euclidean surface $X$ with conical singularities (denoted by ESCS in the sequel)
 is a singular Riemannian surface such that any point has a neighbourhood
that is isometric either to a Euclidean ball in $\bbR^2$ or to a ball centered
at the vertex of some Euclidean cone $C_\alpha$.
\begin{example}
  From any polygonal domain $\Omega$ in the plane we may generate an ESCS by
  taking two copies of the polygon, reflecting one of these copies across the
  $y$-axis, and identifying the corresponding sides.  Starting from a square, we
  build in this way a surface that is topologically a sphere that is flat with
  four singularities of angle $\pi.$
\end{example}
\begin{example}
  More generally, a surface that is obtained by gluing Euclidean polygons along
  their sides also is Euclidean with conical singularities. The surface of a
  cube is a ESCS that is topologically a sphere with $8$ singularities of angle
  $\frac{3\pi}{2}.$
\end{example}

Let $X$ be a Euclidean surface with conical singularities, and let $\vec{P}$ be
the set of its conical points.  Define $X^\circ \defeq X \setminus \vec{P}$.
Let $u$ be a smooth function that vanishes near the conical points.  Using the
Euclidean metric, one defines the gradient of $u$, $\nabla u$, and
the action of the Laplacian on $u$, $\Delta u$, as usual.  The Laplace operator
thus defined is not essentially self-adjoint. Among the possible
self-adjoint extensions, the most natural one is the Friedrichs extension that
is associated with the Dirichlet energy quadratic form
\begin{equation*}
  Q(u) \defeq \int_{X} \left | \nabla u\right |^2 dS, \quad u \in C_c^\infty(X^\circ),
\end{equation*}
where $dS$ is the Euclidean area measure.  Throughout the paper $\Delta$ will
always define the Friedrichs extension of the Euclidean Laplace operator. By
choice it is a non-negative operator.

Writing $\boxop = D_t^2 - \Delta$ with $D_t = \frac{1}{i} \, \del_t,$ the
associated wave operator is then defined as
\begin{equation}
  \label{eq:wave-eqn}
  \left\{
    \begin{aligned}
      &\boxop_g u(t,x) = 0 \\
      &u(0,x) = u_0(x) \\
      &\del_tu(0,x) = \dot{u}_0(x)
    \end{aligned}
  \right.
\end{equation}
We will always take $t \geqslant 0$.  The wave propagators that are associated
with this wave equation may be defined through functional calculus and we denote
them by:
\begin{equation}
  \mathbf{W}(t) \defeq \frac{\sin\!\left( t \sqrt{ \smash[b]{\Delta} }
    \right)}{ \sqrt{\smash[b]{\Delta}} } \qquad \text{and} \qquad
  \dot{\mathbf{W}}(t) \defeq \cos\!\left( t \sqrt{ \smash[b]{\Delta} } \right).
\end{equation}
We will also use the half-wave propagator
$\calU(t) \defeq \exp(-it\sqrt{\Delta}).$

Since singularities of solutions to the wave equation propagate with finite
speed, the propagator $\mathbf{W}(t)$ can be understood by patching together
local propagators that are defined either on the plane or on $C_\alpha.$ As a
first step it is therefore crucial to understand wave propagation on the flat
cone $C_\alpha$.

\subsection{The wave kernel on cones}

It turns out that the wave kernel on $C_\alpha$ is explicitly known (see
\cites{Som, CheTay1, CheTay2, Fri} for different ways of constructing this
kernel --- we describe these briefly  at the beginning of Sections~\ref{sec:micro-structure} and~\ref{sec:micro-structure-alpha}).  Propagation of singularities for the wave equation on $C_\alpha$ is
then described as follows. Using polar coordinates, we define on
$(0,\infty)\times T^*C_\alpha^\circ \times T^*C_\alpha^\circ$ two Lagrangian submanifolds $\Lambda^\geom$ and $\Lambda^\diff$. For $\alpha > \pi$, these can be defined as follows.
\begin{subequations}
  The geometric (or ``main'') Lagrangian is
  \begin{equation}
    \label{introeq:geom-Lagrangian-relation}
    \Lambda^\geom \defeq N^* \! \left\{ t^2 = r_1^2 + r_2^2 - 2 r_1 r_2
      \cos(\theta_1 - \theta_2) \text{ and } |\theta_1 - \theta_2| \leqslant \pi
    \right\},
  \end{equation}
  the diffractive Lagrangian is
  \begin{equation}
    \label{introeq:diff-Lagrangian-relation}
    \Lambda^\diff \defeq N^* \! \left\{ t^2 = \left(r_1 + r_2\right)^2 \right\} ,
  \end{equation}
  and their intersection is the singular set
  \begin{equation}
    \label{introeq:singular-set}
    \Sigma \defeq \Lambda^\geom \cap \Lambda^\diff = \Lambda^\diff \cap \{ |\theta_1 - \theta_2| = \pi \}.
  \end{equation}
\end{subequations}
In the case $\alpha \leq \pi$, we choose an integer $N$ such that $\alpha N > \pi$. Then we consider the $N$-fold covering map from $C^\circ_{N\alpha}$ to $C^\circ_\alpha$ induced by the natural map $\RR /N\alpha \ZZ \to \RR /\alpha \ZZ$.
As this is a local isometry, this induces a covering map $T^* C^\circ_{N\alpha} \to T^* C^\circ_\alpha$. We define $\Lambda^\geom_\alpha$ to be the image of $\Lambda^\geom_{N\alpha}$ under this covering map.

The terminology indicates that $\Lambda^\geom$ corresponds to geometric, or non-diffractive
geodesics (i.e., geodesics on $C_\alpha$ that avoid $\bmp$) which carry the main
singularity whereas $\Lambda^\diff$ corresponds to diffractive geodesics (i.e.,
concatenation of two rays emanating from $\bmp$.) The singular set thus
corresponds to diffractive geodesics that are limits of non-diffractive ones. We
will refer to these as {\em geometrically diffractive}.  We will denote by
$\Lambda^{\geom/\diff}_\pm$ the Lagrangian submanifolds obtained by restricting
$\Lambda^{\geom/\diff}$ to $\mp\tau>0$ where $\tau$ is the dual variable to $t.$

The explicit expression of the propagator implies, first, that singularities
propagate according to $\Lambda^\geom \cup \Lambda^\diff$, and second, that away
from the intersection $\Sigma$ the propagator is a classical Fourier Integral
Operator (FIO). Away from the intersection $\Sigma$, the kernel of the half-wave
propagator $e^{-it\sqrt{\Delta}}$ is given by the so-called {\em Geometric
  Theory of Diffraction} (see Appendix \ref{sec:GTD}).

 \subsection{Main results}

Our first result is a precise description of the kernel of the wave propagator
on the cone $C_\alpha$ near the singular set $\Sigma.$ It is actually a bit
simpler to describe the result for the half-wave propagator
$e^{-it\sqrt{\Delta}}$, whose Schwartz kernel we denote by $\bmU_\alpha.$

We observe that $(t^*,\bmq_1^*,\bmq_2^*)$ is in the projection of $\Sigma$ on
$(0,\infty)\times C_\alpha^\circ \times C_\alpha^\circ$ if and only if, in polar
coordinates, we have $r_1^*+r_2^*=t^*$ and
$\theta_1^*-\theta_2^* = \epsilon\pi,\,\epsilon = \pm 1.$ Let $\gamma$ be the
parametrization by arclength of the geometrically diffractive geodesic that
joins $\bmq_1^*$ to $\bmq_2^*$ normalized in such a way that
$\gamma(-r_2^*)=\bmq_2,~\gamma(0)=\bmp,~\gamma(r_1^*)=\bmq_1.$ Since the cone is
flat, $\gamma$ can be extended to a local isometry $\mathcal{I}_\epsilon$ that is defined
on $\bbR^2\setminus \{(0,\epsilon y),~ y>0,\,\epsilon = \pm 1\, \}.$ Using
$\mathcal{I}_\epsilon$ we can thus parametrize a neighbourhood of $(\bmq_1^*,\bmq_2^*)$ in
$C_\alpha^\circ\times C_\alpha^\circ$ by the product of two Euclidean balls in
$\bbR^2$ the first one centered at $(r_1^*,0)$ and the second one at
$(-r_2^*,0)$ (in Euclidean coordinates).

\begin{theorem}\label{thm:intro-kernel}
  Let $\bmq_1^*$ and $\bmq_2^*$ be the extremities of a geometrically
  diffractive geodesic of length $t^*$ and diffraction angle $\epsilon \pi$
  ($\epsilon = \pm 1).$ Locally near $(t^*,\bmq_1^*,\bmq_2^*)$ in
  $(0,\infty)\times C_\alpha^\circ\times C_\alpha^\circ$ the kernel
  $\bmU_\alpha$ can be written as the following oscillatory integral:
  \begin{equation}
  \bmU_\alpha(t,\bmq_1,\bmq_2)= (2\pi)^{-2} \int_{s\geq 0} \int_{\omega>0} %
  e^{i\phi_{\epsilon}(t,\bmq_1,\bmq_2,s,\omega)}
  a_{\alpha,\epsilon}(t,\bmq_1,\bmq_2,s,\omega)\, d\omega ds
  \label{Uexpr}\end{equation}
  where (using $\mathcal{I}_\epsilon$ for parametrization---i.e.,
  $g_\epsilon(x_1,y_1)=\bmq_1,~g_\epsilon(x_2,y_2)=\bmq_2$)
  \begin{enumerate}
  \item the phase $\phi_\epsilon$ is defined by
    \[
    \phi_{\epsilon}(t,\bmq_1,\bmq_2,s,\omega)=\omega\left [
      \sqrt{x_1^2+(y_1+s\epsilon)^2}+\sqrt{{x_2}^2+({y_2}+s\epsilon)^2}-t\right],
    \]
  \item the amplitude $a_{\alpha,\epsilon}$ is a classical symbol that is smooth
    in $(t,\bmq_1,\bmq_2,s)$ and of order $1$ in $\omega$ so that we have
    \[
    a_{\alpha,\epsilon}\sim \sum_{k\geq 0}
    a_{\alpha,\epsilon,1-k}(t,\bmq_1,\bmq_2,s) \, \omega^{1-k}.
    \]
  \item In polar coordinates, we have at leading order
    \[
    a_{\alpha,\epsilon}(\bmq_1,\,\bmq_2\,,\,s=0,\omega)
    =-2\pi i\epsilon \cdot \frac{S_\alpha(\theta_1-\theta_2)}{
      (r_1r_2)^{\frac{1}{2}}}\cdot \Big [\sin \theta_1+\sin \theta_2\Big
    ]\cdot \omega.
    \]
    where $S_\alpha$ is the (absolute) scattering matrix for the cone $C_\alpha$. An explicit expression for $S_\alpha$ is given by  \eqref{Salpha-expr}.
  \end{enumerate}
\end{theorem}

From this expression we deduce the following corollary.
\begin{theorem}
  The half-wave propagator $\calU_\alpha(t)$ on the Euclidean
  cone $C_\alpha$ is in the Melrose-Uhlmann class
  $I^{m}(\Lambda^\diff_+,\Lambda^\geom_+)$ of singular Fourier Integral
  operators.  The order $m$ is equal to $0$ if $t$ is regarded as a parameter, or $-1/4$ if $t$ is regarded
  as a `spatial' variable.
   Similarly, the sine propagator $\bfW(t)$ on
  the Euclidean cone $C_\alpha$ is in the Melrose-Uhlmann class
  $I^{m-1}(\Lambda^\diff,\Lambda^\geom)$ of singular Fourier Integral
  operators.
\end{theorem}

It can be noted that elements of this class are standard FIOs away from the
intersection $\Sigma$ so that this theorem doesn't say anything new away from
$\Sigma$. On the other hand, although the explicit expression of the propagator
was already known near $\Sigma$, the fact that it belonged to the
Melrose-Uhlmann class was not. It is also worth remarking that it may be
possible to obtain the latter theorem by some brute computations starting from
the explicit expression of the propagator. We propose a different method, the `moving conical point'
method, that
exploits geometric features of wave propagation on cones. It has the advantage that the parameter
$s$ in \eqref{Uexpr} then has geometric significance: it is the distance by which the conical point is shifted.


\begin{remark}
  It is actually convenient to use the Riemannian metric to identify functions
  and half-densities. This amounts to multiply the oscillatory integral
  representation by the half-density $|d\bmq' d\bmq |^{\frac{1}{2}}$ or
  $|dtd\bmq' d\bmq |^{\frac{1}{2}}$.
\end{remark}

\begin{remark}
Recall (or see Section~\ref{sec:int-Lagn-distns}) that in the Melrose-Uhlmann calculus, the order
of the distribution on the first Lagrangian $\Lambda^\diff$ is $\frac{1}{2}$-order
less than on the second, $\Lambda^\geom$. This  allows to recover the
 fact that the diffracted wave is $\frac{1}{2}$-order smoother (in a Sobolev
 sense) than the direct wave (in two dimensions).
\end{remark}

The oscillatory integral representation of the preceding theorem has several
interesting applications mainly because it allows one to compute simply the wave
propagator on an ESCS when microlocalized near a geodesic with several geometric
diffractions.  We will illustrate this by obtaining, for a geodesic with two
geometric diffractions in a row an oscillatory integral representation that fits
into the class of singular FIO that is constructed in \cite[Sections 7--10]{MelUhl} and
associated with a system of \emph{four} intersecting Lagrangians. More precisely,
consider a geodesic of length $t$ between $\bmq$ and $\bmq'$ with two geometric
diffractions at $\bmp_1$ and $\bmp_2$.  There are four types of nearby
geodesics:
\begin{enumerate}
\item non-diffractive geodesics;
\item geodesics that are diffractive at $\bmp_1$ but not at $\bmp_2$;
\item non-diffractive geodesics at $\bmp_1$ that diffract at $\bmp_2$; and
\item geodesics that diffracts at $\bmp_1$ and $\bmp_2.$
\end{enumerate}
Each type corresponds to a particular Lagrangian and these four Lagrangians form
a intersecting system in the sense of \cite{MelUhl}.

Using the preceding theorem and standard stationary phase arguments we obtain
the theorem.
\begin{theorem}
  Microlocally near a geodesic with two geometric diffractions, the half-wave
  propagator on a ESCS is in the Melrose-Uhlmann class of operators associated
  with a system of four intersecting Lagrangians.
\end{theorem}
We actually get much more accurate information since we can derive the principal symbol of the half-wave propagator on the twice diffracted front --- see \eqref{atilde4} and \eqref{2Diff-prsymb}.

Finally we will use our new expression for $\bmU_\alpha$ to compute the
contribution to the wave-trace of an isolated periodic geodesic with two
geometric diffractions.

\begin{proposition}
  On a ESCS, the leading contribution to the wave trace of an isolated periodic
  diffractive orbit with two geometric diffractions is
  \[
  -\frac{1}{4i\pi^2} \cdot \sqrt{b(L-b)} \cdot (t-L-i0)^{-1}.
  \]
\end{proposition}

This is perhaps the simplest setting for which neither \cite{ForWun} nor
\cite{Hil} applies. This proposition shows that such a geodesic creates in the
wave-trace a singularity that is comparable to the singularity that is created
in a smooth setting by an isolated periodic orbit. The singularity is $\frac1{2}$ stronger
than a diffractive geodesic with one non-geometric diffraction and $\frac1{2}$
weaker than a cylinder of periodic orbits.

With our new representation of the wave kernel, it should actually be possible
to compute the full asymptotic expansion of the contribution to the wave-trace
of any kind of periodic diffractive geodesic. This is a far-reaching
generalization of results in \cite{Hil} and it leads to the possible computation
of many wave-invariants. This opens new questions concerning inverse spectral
problems in this kind of geometric setting which, we recall, includes Euclidean
polygons. For instance it can be asked whether the full asymptotic expansion of
a particular geodesic allows to recover the full picture describing the
geodesic: that is the number of diffractions, the lengths of the legs between
two diffractions, the diffraction angles and the angles of the cone at which the
diffractions occur. We will tackle some of these questions in the second and
third parts of this series.

\subsection{Organisation of the paper}
In Section \ref{sec:int-Lagn-distns} we will recall the definition of singular
Fourier Integral Operators as defined in \cite{MelUhl}.  We will first study the
case of two intersecting Lagrangians.  We will give the oscillatory integral
representation using a phase function that depends on an extra parameter $s$. We
will then give the generalization to a system of four intersecting Lagrangians.

In Section \ref{sec:micro-structure} we will study wave propagation on a cone of
total angle $4\pi.$ The first reason why we study this particular cone is that
it is the simplest case in which we can implement our method of `moving the
conical point' that leads to our new expression for the wave propagator. The
fact that the wave propagator belongs to the Melrose-Uhlmann class can then be
directly read off from this expression. It is also worth remarking that, in this
case the extra parameter $s$ has a geometric meaning since it represents the
amount of which the conical point has moved.

The second reason why we can first study the cone of angle $4\pi$ is that the
most singular part of the wave propagator near $\Sigma$ actually does not depend
on its angle. This can be seen using the construction of the wave kernel made by
Friedlander in \cite{Fri}. We will recall this fact in Section
\ref{sec:micro-structure-alpha} and then proceed to prove Theorem
\ref{thm:intro-kernel}.

In Section \ref{sec:wave-kernel-two-diffractions} we will use Theorem
\ref{thm:intro-kernel} to compute the wave propagator when microlocalized near
some particular kind of geodesics. We will focus on the case of a geodesic with
two geometric diffractions for which a desciption of the microlocalized
propagator is not already available in the literature.

In Section \ref{sec:wave-invariants} we will end this paper by computing the
leading contribution to the wave-trace of an isolated periodic orbit with two
geometric diffractions.


\section{Intersecting Lagrangian distributions}
\label{sec:int-Lagn-distns}

The class of distributions central to our study of the wave propagators on
$C_\alpha$ is that of intersecting Lagrangian distributions, introduced by
Melrose and Uhlmann \cite{MelUhl}.  These are distributions whose singularities
(in terms of wavefront set) lie along a pair of conic Lagrangian submanifolds
$(\Lambda_0, \Lambda_1)$ of the cotangent bundle. Here, $\Lambda_1$ is a
manifold with boundary, and $\Lambda_0$ and $\Lambda_1$ intersect cleanly at
$\partial \Lambda_1$. In particular, the intersection is codimension $1$ in both
Lagrangians.  These distributions were introduced to construct fundamental
solutions to operators of real principal type. An analogous class of
distributions associated to four intersecting Lagrangian submanifolds, also
introduced in \cite{MelUhl}, will show up in our study of the wave kernel on a
ESCS after two diffractions---see Section~\ref{subsec:4il}.

\subsection{Model Lagrangian submanifolds} Let $X$ be a manifold, and let
$(\Lambda_0, \Lambda_1)$ be a pair of conic Lagrangian submanifolds of
$T^* X \setminus \{ 0 \}$ with the geometry described above: $\Lambda_1$ is a
manifold with boundary, and $\Lambda_0$ and $\Lambda_1$ intersect cleanly at
$\partial \Lambda_1$. Moreover, let $q \in \Lambda_0 \cap \Lambda_1$ be a point
in the intersection.  Melrose and Uhlmann showed that there is a normal form for
this geometry. Indeed, let $(\tilde \Lambda_0, \tilde \Lambda_1)$ be the model
Lagrangian submanifolds in $T^* \RR^n$:
\begin{equation}\begin{aligned}
    \Lambda_0 &= N^* \{ 0 \} = \{ (x, \xi): x = 0, \} \\
    \Lambda_1 &= N^* \{ x' = 0, x_1 \geq 0 \} = \{ (x, \xi): x' = 0, \xi_1 = 0,
    x_1 \geq 0 \}.
  \end{aligned}\end{equation}
Here we decompose $x = (x_1, x')$, where $x' = (x_2, \dots, x_n)$; similarly,
$\xi = (\xi_1, \xi')$. Choose any point
$\tilde q \in \tilde \Lambda_0 \cap \tilde \Lambda_1$.  Then Melrose and Uhlmann
showed that there is a homogeneous sympectic map from a conic neighbourhood of
$\tilde q$ to a conic neighbourhood of $q$, such that $\tilde \Lambda_i$ gets mapped to $\Lambda_i$. To
define intersecting Lagrangian distributions, they first defined them in the
model situation.
We recall this definition.

\begin{definition}[Melrose-Uhlmann] An intersecting Lagrangian distribution of order $m$
  associated to the model pair $(\tilde \Lambda_0, \tilde \Lambda_1)$ is a
  distributional half-density given by an oscillatory integral of the form
  \begin{equation}\label{def:modelILD}
    (2\pi)^{-n - \frac{1}{2}} \int \int_0^\infty e^{i(x \cdot \xi - s \xi_1)}
    a(x, s, \xi) \, ds \, d\xi  |dx|^{\frac{1}{2}}
  \end{equation}
  where $a$ is smooth, compactly supported in $x$ and $s$, and a symbol of order
  $m +\frac{1}{2}-\frac{n}{4}$ in $\xi$. The space of such distributions is
  denoted $I^m(X; \tilde\Lambda_0, \tilde\Lambda_1)$.
\end{definition}

It is shown in \cite{MelUhl} that elements of
$I^m(X; \tilde\Lambda_0, \tilde\Lambda_1)$ are Lagrangian distributions of order
$m$ on $\tilde \Lambda_1$ when microlocalized away from $\tilde \Lambda_0$, and
Lagrangian distributions of order $m - \frac{1}{2}$ on $\tilde \Lambda_0$ when
microlocalized away from $\tilde \Lambda_1$.  Also, they showed that the space
$I^m(X; \tilde\Lambda_0, \tilde\Lambda_1)$ is invariant under the action of
Fourier integral operators that fix $\tilde\Lambda_0$ and $\tilde\Lambda_1$.
Consequently, one can define intersecting Lagrangian distributions for a general
pair $(\Lambda_0, \Lambda_1)$ to be the image of the model space
$I^m(X; \tilde\Lambda_0, \tilde\Lambda_1)$ under an FIO mapping
$\tilde \Lambda_i$ to $\Lambda_i$. The precise definition is as follows:

 \begin{definition}\label{def:ILD}
   Let $(\Lambda_0, \Lambda_1)$ be a pair of intersecting conic Lagrangian
   distributions in $T^* X \setminus \{ 0 \}$ with the geometry described
   above. The space $I^m(X; \Lambda_0, \Lambda_1)$ consists of those
   distributional half-densities $u$ that can be written as a locally finite sum
 $$
 u = u_0 + u_1 + \sum_i F_i (v_i) + u_\infty,
 $$
 where $u_0 \in I^{m-\frac{1}{2}}(X;\Lambda_0)$,
 $u_1 \in I^m(X;\Lambda_1 \setminus \Lambda_0)$,
 $v_i \in I^m(X;\tilde \Lambda_0, \tilde \Lambda_1)$, $F_i$ are FIOs mapping
 $(\tilde\Lambda_0, \tilde\Lambda_1)$ to $(\Lambda_0, \Lambda_1)$, and
 $u_\infty$ is $\calC^\infty$.
\end{definition}

In what follows, we will often omit the space `$X$' from the notation for these
distributions, i.e., we will write $I^m(\Lambda_0,\Lambda_1)$ in the place of
$I^m(X;\Lambda_0,\Lambda_1)$.

\subsection{Parametrization of intersecting Lagrangian submanifolds}

Over the course of this paper, we will construct the fundamental solution of the
wave kernel on a two-dimensional cone directly; we will want to be able to
identify it as an intersecting Lagrangian distribution. To do this, we need a
direct definition of intersecting Lagrangian distribution in terms of a phase
function parametrizing a given pair $(\Lambda_0, \Lambda_1)$ in place of the
indirect Definition~\ref{def:ILD}.

\begin{definition}\label{def:intLeg-param} Let $(\Lambda_0, \Lambda_1)$ be a
  pair of intersecting Lagrangian submanifolds, and let
  $q \in \Lambda_0 \cap \Lambda_1$ be a point in the intersection.  A local
  parametrization of $(\Lambda_0, \Lambda_1)$ near $q$ is a function
  $\phi(x, \theta, s)$, defined in neighbourhood of
  $(x_0, \theta_0, 0) \subseteq X \times \RR^k \times \bbR_{\geqslant 0}$ such
  that
  \begin{itemize}
  \item $d_{\theta, s} \phi(x_0, \theta_0, 0) = 0$, and
    $q = (x_0, d_x \phi(x_0, \theta_0, 0))$;
  \item the differentials
    $$
    d_{x, \theta} \bigg( \frac{\partial \phi}{\partial \theta_i} \bigg) \quad
    \text{and} \quad d_{x, \theta} \bigg( \frac{\partial \phi}{\partial s}
    \bigg)
    $$
    in the $(x, \theta)$ directions are linearly independent at
    $(x_0, \theta_0, 0)$;
  \item the map
    \begin{equation}
      C_0 \defeq \{ (x, \theta) : d_\theta \phi(x, \theta, 0) = 0 \} \mapsto \{
      (x, d_x \phi(x, \theta, 0)) \} \subseteq T^* X
      \label{C0bij}\end{equation}
    is a local diffeomorphism from $C_0$ onto a neighbourhood of $q$ in
    $\Lambda_0$;
  \item the map
    \begin{equation}
      C_1 \defeq \{ (x, \theta, s) : d_{\theta,s} \phi(x, \theta, s) = 0, \ s
      \geq 0 \} \mapsto \{ (x, d_x \phi(x, \theta, s)) \} \subseteq T^* X
      \label{C1bij}\end{equation}
    is a local diffeomorphism from $C_1$ onto a neighbourhood of $q$ in
    $\Lambda_1$.
  \end{itemize}
\end{definition}

Let us make some remarks about the definition above. The second condition
ensures that the sets $C_0$ is a smooth submanifold of $X \times \RR^k$ of dimension
$n = \dim X$, and $C_1$ is a smooth submanifold of
$X \times \RR^k \times \bbR_{\geqslant 0}$ of dimension
$n$ transverse to $\{ s = 0 \}$. This makes it
possible to speak of diffeomorphisms from $C_i$ to $\Lambda_i$ as in the third
and fourth conditions.  The first condition simply says that the base point
$(x_0, \theta_0, 0)$ corresponds to the base point $q$.

\begin{proposition}\label{prop:int-pair-phasefunction}
  (i) Let $(\Lambda_0, \Lambda_1) \subseteq T^* X$ be a pair of intersecting
  Lagrangian submanifolds, and let $q$ be a point in the intersection. Then
  there exists a local parametrization of $(\Lambda_0, \Lambda_1)$ near $q$.

  (ii) Let $\phi$, defined in a neighbourhood $U$ of
  $(x_0, \theta_0, 0) \in X \times \RR^k \times R_{\geq 0}$, be a local
  parametrization of $(\Lambda_0, \Lambda_1)$ near $q$. Let $a(x, \theta, s)$ be
  a classical symbol of order $m - \frac{k}{2} + \frac{1}{2} + \frac{n}{4}$ in
  the $\theta$ variables which is compactly supported in $U$.  Then the
  oscillatory integral
  \begin{equation}
    (2\pi)^{- \frac{k}{2} - \frac{n}{4} - \frac{1}{2}} \int_{\RR^k}
    \int_0^\infty e^{i\phi(x, \theta, s)} a(x, \theta,
    s) \, ds \, d\theta \, |dx|^{\frac{1}{2}}
    \label{intlegm}\end{equation}
  is in $I^m(\Lambda_0, \Lambda_1)$.

\end{proposition}

\begin{proof}
  (i) By \cite{MelUhl}, there is a homogeneous canonical transformation $\chi$
  defined in a neighbourhood of
  $\tilde q \in \tilde \Lambda_0 \cap \tilde \Lambda_1$ mapping
  $\tilde \Lambda_0$ to $\Lambda_0$ and $\tilde \Lambda_1$ to $\Lambda_1$, and
  sending $\tilde q$ to $q$. Let $\Psi(x, y, \theta)$ be a phase function
  parametrizing the graph of this canonical transformation. Consider the sum of
  the phase functions
  $$
  \Psi(x, y, \theta) + y \cdot \eta - \eta_1 s,
  $$
  where the second phase function is the standard parametrization of the model
  Lagrangian pair.  Following \cite{HorFIO}*{p.~175}, we define a new variable
  $$
  Y = |\theta| y.
  $$
  We then write this sum of the phase functions in terms of $Y$. That is, we
  define
  $$
  \phi(x, Y, \theta, \eta, s) = \Psi\!\left(x, \frac{Y}{|\theta|}, \theta\right)
  +\frac{Y}{|\theta|} \cdot \eta - \eta_1 s .
  $$
  Notice that $\phi$ is homogeneous of degree 1 in the variables
  $(Y, \theta, \eta)$.  We claim that $\phi$ is a nondegenerate local
  parametrization of $(\Lambda_0, \Lambda_1)$ near $q$.

  Let $(y_0, \eta_0, 0)$ be the point corresponding to $\tilde q$ and
  $(x_0, y_0, \theta_0)$ be the point corresponding to $(q, \tilde q)$ in the
  graph of $\chi$. Then $d_{\theta, Y, \eta} \phi = 0$ and $s=0$ implies that
  $d_\theta \Psi(x_0, y_0, \theta_0) = 0$, $y_0 = 0$,
  $d_y \Psi(x_0, 0, \theta_0) = -\eta$ and
  $d_x\Psi(x_0, 0, \theta_0) = \chi(0, \eta) = q$, so the first condition in
  Definition~\ref{def:intLeg-param} is satisfied.

  We next check that the second condition is satisfied, i.e., that $\phi$ is
  nondegenerate. To do this, we claim that the differentials
  $$
  d_{x, \theta} \bigg( \frac{\partial \Psi}{\partial \theta_i} \bigg) \quad
  \text{and} \quad d_{x, \theta} \bigg( \frac{\partial \Psi}{\partial y_i}
  \bigg)
  $$
  are linearly independent at $(x_0, y_0, \theta_0)$. This is a consequence of
  the fact that $\Psi$ parametrizes $\Lambda_{\Psi}$, the (twisted) graph of the
  canonical transformation $\chi$, which implies that the functions $y_i$ and
  $d_{y_j} \Psi$ are coordinates on $\Lambda_{\Psi}$. Using the diffeomorphism
  between
  $$
  C_\Psi = \{ (x, y, \theta) : d_\theta \Psi = 0 \}
  $$
  and $\Lambda_\Psi$, we see that $y_i$ and $d_{y_j} \Psi$ are coordinates on
  $C _\Psi$. This implies that
  $$
  y_i, \quad \frac{\partial \Psi}{\partial y_j}, \quad \text{and} \quad
  \frac{\partial \Psi}{\partial \theta_i}
  $$
  have linearly independent differentials at $(x_0, y_0,
  \theta_0)$. Equivalently we can say that
  $$
  d_{x, \theta} \bigg( \frac{\partial \Psi}{\partial y_j} \bigg) \quad
  \text{and} \quad d_{x, \theta} \bigg( \frac{\partial \Psi}{\partial \theta_i}
  \bigg)
  $$
  are linearly independent at $(x_0, y_0, \theta_0)$. This in turn is equivalent
  to the statement that
  \begin{equation}
    \label{phi-nondeg-1}
    d_{x, \theta} \bigg( \frac{\partial \phi}{\partial Y_j} \bigg)
    \text{ and }  d_{x, \theta} \bigg( \frac{\partial \phi}{\partial \theta_i}
    \bigg) \text{ are linearly independent at } (x_0, Y_0, \theta_0),
  \end{equation}
  where $Y_0 = y_0 |\theta_0|$.  Now, from the explicit form of $\phi$ it is
  evident that
  \begin{equation}
    \label{phi-nondeg-2}
    d_{Y, \eta} \bigg( \frac{\partial \phi}{\partial \eta_i} \bigg) \text{ and }
    d_{Y, \eta} \bigg( \frac{\partial \phi}{\partial s} \bigg)  \text{ are linearly
      independent at } (x_0, Y_0, \theta_0).
  \end{equation}
  Putting \eqref{phi-nondeg-1} and \eqref{phi-nondeg-2} together we find that
  $\phi$ is a nondegenerate phase function, i.e., it satisfies the second point
  in Definition~\ref{def:intLeg-param}.

  To check the third point, consider a point $(x, Y, \theta, \eta, 0)$ where
  $d_{ Y, \theta, \eta} \phi = 0$ and $s=0$. This implies that
  \begin{equation}
    d_\theta \Psi(x, y, \theta) = 0, \quad d_\eta (y \cdot \eta) = 0, \quad
    \text{and} \quad
    d_y\Psi(x, y, \theta) + d_y (y \cdot \eta) = 0.
  \end{equation}
  Using the fact that $\Psi$ parametrizes the twisted graph of $\chi$, this
  implies that
  \begin{equation}
    y = 0, \quad d_y \Psi = -\eta, \quad \text{and} \quad (x, d_x \Psi) =
    \chi(y, -d_y \Psi).
  \end{equation}
  Thus, $d_{ Y, \theta, \eta} \phi = 0$ implies that the Lagrangian parametrized
  is
  $$
  \{ (x, d_x \phi ) \} = \{ (x, d_x \Psi) \} = \{ \chi(0, \eta) \}.
  $$
  As $(x, Y, \theta, \eta)$ range over a neighbourhood of
  $(x_0, Y_0, \theta_0, \eta_0)$, the point $(0, \eta)$ ranges over a
  neighbourhood of $\tilde q \in \tilde\Lambda_0$, and therefore $\chi(0, \eta)$
  ranges over a neighbourhood of $q \in \Lambda_0$. This verifies the third
  condition in the Definition. Exactly the same reasoning shows that the fourth
  condition in the Definition is also satisfied. This completes the proof of
  part (i) of the Lemma.

  (ii) Choose an FIO $F$ associated to the canonical relation $\chi$ as above,
  and which is microlocally invertible at $(q, \tilde q)$. Let $F^{-1}$ denote a
  microlocal inverse to $F$. Write $F^{-1}$ with respect to a phase function
  $S(y, x, \omega)$. Then the phase function
  $$
  \Phi = S(y, x, \omega) + \phi(x, \theta, s)
  $$
  parametrizes the model pair $(\tilde \Lambda_0, \tilde \Lambda_1)$ (after we
  homogenize the $x$ variable by changing to the variable $X = x |\omega|$, as
  we did in the proof of part (i)). The proof is the same as in part (i), so we
  omit it. It then suffices to show that an oscillatory integral with phase
  function $\Phi$,
  \begin{equation}
    \int \int_0^\infty e^{i\Phi(y, X, \omega, \theta, s)} a(y, X, \omega,
    \theta, s) \, ds \, dX \, d\theta \, d\omega
    \label{osc-int-S}\end{equation}
  gives an element of $I^m(\tilde \Lambda_0, \tilde \Lambda_1)$, since the
  original oscillatory integral is, modulo $\calC^\infty$ functions, the image
  of \eqref{osc-int-S} by the Fourier integral operator $F$, which by definition
  maps $I^m(\tilde \Lambda_0, \tilde \Lambda_1)$ to
  $I^m( \Lambda_0, \Lambda_1)$. Thus, we have reduced to the case that the
  intersecting pair $( \Lambda_0, \Lambda_1)$ is the model pair
  $( \tilde \Lambda_0, \tilde \Lambda_1)$.

  We now simplify our notation, and assume that $\Phi(y, \theta, s)$ is a
  nondegenerate phase function parametrizing
  $( \tilde \Lambda_0, \tilde \Lambda_1)$ locally near $\tilde q$, with
  $(y_0, \theta_0)$ corresponding to the point $\tilde q$. Here
  $\theta \in \RR^k$, with $k \geq n$. We want to show that
  \begin{equation}
    u = \int \int_0^\infty e^{i\Phi(y, \theta, s)} a(y,  \theta, s) \, ds \,
    d\theta \quad \text{for} \quad a \in S^{m - \frac{k}{2} + \frac{1}{2} +
      \frac{n}{4}}(X \times \bbR_{\geqslant 0}; \RR^k)
    \label{osc-int-Phi}\end{equation}
  is in the space $I^{m}( \tilde \Lambda_0, \tilde \Lambda_1)$. Essentially this
  proof follows that of Proposition 3.2 in \cite{MelUhl}. We first note that
  $\Phi_0(y, \theta) \defeq \Phi(y, \theta, 0)$ parametrizes $\Lambda_0$. We
  have by \cite{HorFIO}*{(3.2.12)} that the rank of
  $d^2_{\theta \theta} \Phi(y_0, \theta_0)$ is $k - n$. By rotating in the
  $\theta$ variables we can arrange that $\theta = (\theta', \theta'')$ with
  $\dim \theta' = n$, $\dim \theta'' = k-n$ and such that
  $d^2_{\theta'' \theta''} \Phi(y_0, \theta_0)$ is nondegenerate. Integrating in
  the $\theta''$ variables and applying the stationary phase expansion, as in
  \cite{HorFIO}*{p.~142}, we find that the result takes the form
  \begin{equation}
    u = \int \int_0^\infty e^{i\Phi(y, \theta', \theta''(y, \theta', s), s)}
    \tilde a(y,  \theta', s) \, ds \, d\theta'' \quad \text{for} \quad \tilde a
    \in S^{{m - \frac{k}{2} + \frac{1}{2} +\frac{n}{4}}}
    \label{osc-int-Phi2}\end{equation}
  where $\theta''(y, \theta', s)$ is the critical point, determined by the
  equation
  $$
  d_{\theta''} \phi(y, \theta', \theta'', s) = 0;
  $$
  this varies smoothly with $(y, \theta', s)$ near $(y_0, \theta_0, 0)$ thanks
  to the implicit function theorem and the nondegeneracy of
  $d^2_{\theta'' \theta''} \Phi(y_0, \theta_0)$ near $(y_0, \theta_0, 0)$. Then
  the phase function
  $\Phi'_0(y, \theta', 0) \defeq \Phi(y, \theta', \theta''(y, \theta', 0), 0)$
  parametrizes $\tilde \Lambda_0$. Moreover, it has the same number of fibre
  variables as the standard phase function $y \cdot \eta$, and its fibre Hessian,  $d_{\theta' \theta'} \Phi'_0$
  has the same signature (namely, zero) as the fibre Hessian of $y \cdot \eta$.
  By H\"ormander's equivalence of
  phase functions, \cite{HorFIO}*{(3.2.12)}, there is a coordinate
  transformation $\eta = \eta(y, \theta')$ mapping $\Phi'_0$ to $y \cdot \eta$ in a
  neighbourhood of $(y_0, \theta'_0, 0)$. Employing this change of variables, we
  are reduced to the case that $\Phi(y, \theta', s)$ has the form
  $y \cdot \theta' + O(s)$.  We can now follow the proof of Proposition 3.2 in
  \cite{MelUhl} from Equation (3.7) of \cite{MelUhl} to the conclusion, which
  completes the proof of the Lemma.
\end{proof}

We next want to identify the symbols at $\Lambda_0$ and $\Lambda_1$ directly
from the oscillatory integral expression \eqref{intlegm}. Recall that the symbol
on each $\Lambda_i$ is half-density taking values in the Maslov bundle. For our
purposes, it is enough to do this when our Lagrangians $\Lambda_0$ and
$\Lambda_1$ are conormal bundles. In this case, the Maslov bundle is canonically
trivial, which means that we may regard the symbol as being simply a
half-density. In the following theorem, we identify functions on $\Lambda_i$ and
$C_i$, where $C_i$ is given by \eqref{C0bij}, \eqref{C1bij}. We let
$\lambda = (\lambda_1, \dots, \lambda_n)$ be local coordinates on $C_1$, or
equivalently on the Lagrangian $\Lambda_1$. Similarly, we let $\tilde \lambda$
be local coordinates on $C_0$. Notice that we could choose
$\lambda, \tilde \lambda$ to be of the form $\lambda = (\lambda', s)$ and
$\tilde \lambda = (\lambda', d_s \phi)$ where $\lambda'$ are coordinates on
$C_0 \cap C_1$.

\begin{proposition}
  Suppose now that $\Lambda_0$ and $\Lambda_1$ are both the conormal bundle of
  codimension one submanifolds $M_0$ and $M_1$. Then

  (i) The symbol of \eqref{intlegm} at $\Lambda_1$ is given by
  \begin{equation}
    e^{\frac{i\pi \sigma}{4}} a(x, \theta, s) |_{C_1} \Big|
    \frac{ \partial (\lambda, \phi_\theta, \phi_s)}{\partial (x, \theta,
      s)}\Big|^{-\frac{1}{2}} \,  |d\lambda|^{\frac{1}{2}}
  \end{equation}
  where $\sigma$ is the signature of the Hessian
  $\phi''_{(\theta, s)(\theta, s)}$ in the $(\theta, s)$ variables.

  (ii) The symbol of \eqref{intlegm} at $\Lambda_0$ is given by
  \begin{equation}
    (2\pi)^{-\frac{1}{2}} e^{\frac{i\pi \sigma'}{4}} \frac{ i a(x, \theta,
      0)}{\phi_s(x, \theta, 0)} \big|_{C_0}
    \Big| \frac{ \partial (\tilde\lambda, \phi_\theta)}{\partial (x,
      \theta)}\Big|^{-\frac{1}{2}}\, |d \tilde\lambda|^{\frac{1}{2}}
    \label{diff-pr-symbol}  \end{equation}
  where $\sigma$ is the signature of the Hessian $\phi''_{\theta\theta}$ at $s=0$.
\end{proposition}

\begin{remark} We remark that $\sigma$ and $\sigma'$ are constant, as follows
  from \cite{HorFIO}*{(3.2.10)} by comparing with the standard parametrization
  of a conormal bundle with linear phase function.
\end{remark}

This proposition follows directly from \cite{HorFIO}*{Section 3}.

\subsection{Four intersecting Lagrangians}\label{subsec:4il}
The wave kernel after two diffractions is associated to four different
Lagrangian submanifolds: the direct front, one front from a diffraction with
each cone point, and a fourth front from diffractions with both cone
points. We shall show that the wave kernel in this case is contained
in the Melrose-Uhlmann calculus of distributions associated to four Lagrangian
distributions described in \cite{MelUhl}*{Sections 7--10}.  We now recall some
of this material, starting with the definition of a system of intersecting
Lagrangian submanifolds.

\begin{definition}\label{def:4ILS}
  A system of four intersecting conic Lagrangian submanifolds of $T^* X$ is a
  quadruple $\bmLambda = (\Lambda_0, \Lambda_1, \Lambda_2, \Lambda_3)$ of Lagrangian
  submanifolds, where $\Lambda_1$ and $\Lambda_2$ are manifolds with boundary
  and $\Lambda_2$ is a manifold with codimension two corner, with the following
  properties:
  \begin{itemize}
  \item $(\Lambda_0, \Lambda_1)$ and $(\Lambda_0, \Lambda_2)$ are intersecting
    pairs in the sense of the previous subsection;
  \item
    $\Lambda_1 \cap \Lambda_2 = \partial\Lambda_1 \cap \partial\Lambda_2 =
    \Lambda_0 \cap \Lambda_3 = c\Lambda_3$,
    where $c\Lambda_3$ denotes the codimension 2 corner of $\Lambda_3$;
  \item The two boundary hypersurfaces of $\Lambda_3$ are
    $\Lambda_3 \cap \Lambda_1$ and $\Lambda_3 \cap \Lambda_2$.
  \end{itemize}
\end{definition}

For example, the following is a system of intersecting Lagrangian submanifolds:

\begin{definition}
  Suppose $n \geqslant 3$.  For $j = 0,\ldots,3$, define $\tilde \bmLambda = (\tilde \Lambda_0,
  \tilde \Lambda_1, \tilde \Lambda_2, \tilde \Lambda_3)$ to be the following
  quadruple of Lagrangian submanifolds of $T^*\RR^n$:
  \begin{equation}
    \begin{aligned}
      \tilde \Lambda_0 &= \{ (x, \xi) : x = 0 \} \\
      \tilde \Lambda_1 &= \{ (x, \xi) : x_1 \geq 0, x_2 = \dots = x_n = 0,
      \xi_1 = 0 \} \\
      \tilde \Lambda_2 &= \{ (x, \xi) : x_2 \geq 0, x_1 = x_3 = \dots = x_n
      = 0, \xi_2 = 0 \} \\
      \tilde \Lambda_3 &= \{ (x, \xi) : x_1 \geq 0, x_2 \geq 0, x_3 = \dots x_n
      = 0, \xi_1 = \xi_2 = 0 \}.
    \end{aligned}
    \label{model-system}
  \end{equation}
  \end{definition}

  Locally, an intersecting system as in Definition~\ref{def:4ILS} may be realized as follows.
 Let $\Lambda_0$ be a Lagrangian submanifold, and let $p_1$, $p_2$ be two functions on $T^* X$ such
the Hamilton vector fields $H_{p_1}$, $H_{p_2}$ are linearly independent,
transverse to $\Lambda_0$, and commute with each other.  Then we define
$\Lambda_i$, $i = 1, 2$ to be the flowout from $\Lambda_0 \cap \{ p_i = 0 \}$ by
$H_{p_i}$, and $\Lambda_3$ to be the flowout from
$\Lambda_0 \cap \{ p_1 = p_2 = 0 \}$ by the flowout of both Hamilton vector
fields. For example, the model system is of this form, where $p_1 = \xi_1$ and $p_2 = \xi_2$.
It turns out that, locally,  all intersecting systems arise in this way. As a
consequence, every system of four intersecting Lagrangian submanifolds is the
image of a model system under a homogeneous canonical transformation. We now
define the model system. That is, one could alternatively define an intersecting system
by the requirement that, locally, it is the same of the model system under a
 homogeneous canonical transformation.

We next define the space of Lagrangian distributions associated to the model
intersecting system $\tilde \bmLambda$ given by \eqref{model-system}.

\begin{definition}[{\cite{MelUhl}*{Definition 8.1}}]
  \label{def:mod-Lag-dist-sys}
  The space $I^m_\upc(\RR^n; \tilde \bmLambda)$ consists of those distributional
  half-densities that can be expressed in the form
  \begin{equation}
    (2\pi)^{-n-1} \int \int_0^\infty \int_0^\infty e^{i(x \cdot \xi - s_1 \xi_1
      - s_2 \xi_2)}
    a(x, \xi, s_1, s_2) \, ds_1 \, ds_2 \, d\xi \ |dx|^{\frac{1}{2}}
  \end{equation}
  where $a$ is smooth and compactly supported in $(x, s_1, s_2)$ and is a symbol
  of order $m + 1 - \frac{n}{4}$ in the $\xi$-variables.
\end{definition}

It is not hard to check that if $u \in I^m_\upc(\RR^n; \tilde \bmLambda)$ then
the wavefront set is of $u$ is contained in
$\bigcup_{i = 0}^3 \tilde \Lambda_i$, and if $q \in \tilde \Lambda_i$ is not
contained in $\tilde \Lambda_j$ for $j \neq i$, then $u$ is a Lagrangian
distribution associated to $\tilde \Lambda_i$ microlocally near $q$, of order
$m$ if $i=2$, $m - \frac{1}{2}$ if $i = 1$ or $2$ and $m-1$ if $i=0$. We can
also observe that if $u$ is microsupported near
$\tilde\Lambda_i \cap \tilde\Lambda_j$, $i < j$, and away from the other
$\tilde\Lambda_k$, then it is an intersecting pair of order $m-\frac{1}{2}$
associated to $(\tilde\Lambda_i, \tilde\Lambda_j)$ for $(i,j) = (0,1)$ or
$(0,2)$, or of order $m$ for $(i,j) = (1,3)$ or $(2,3)$.

It is shown in \cite{MelUhl} that the model space
$I^m_\upc(\RR^n; \tilde \bmLambda)$ is invariant under FIOs that map each
$\tilde \Lambda_i$ to itself. As a consequence, we can define intersecting
Lagrangian distributions associated to a general intersecting system
$\bmLambda = (\Lambda_0, \Lambda_1, \Lambda_2, \Lambda_3)$.


\begin{definition}[{\cite{MelUhl}*{Definition 8.7}}]
  \label{def:Leg-dist-sys}
  Let $\bmLambda$ be an intersecting system of homogeneous Lagrangian
  submanifolds of $T^*X$. The space $I^m(X; \bmLambda)$ consists of those
  distributional half-densities $u$ that can be written as a locally finite sum
  $$
  u = u_{01} + u_{02} + u_{13} + u_{23} + \sum_i F_i(v_i),
  $$
  where $u_{ij} \in I^{m-\frac{1}{2}}(X; \Lambda_i, \Lambda_j)$ for
  $(i,j) = (0,1)$ or $(0,2)$, $u_{ij} \in I^{m}(X; \Lambda_i, \Lambda_j)$ for
  $(i,j) = (1,3)$ or $(2,3)$, $F_i$ are FIOs mapping the model intersecting
  system $\tilde \bmLambda$ to $\bmLambda$, and
  $v_i \in I^m(\RR^n; \tilde \bmLambda)$.
\end{definition}

As before, we will often omit the space `$X$' from the notation for these spaces
of distributions.

We will find it useful to have a definition of $I^m(\bmLambda)$ defined directly
in terms of phase functions. To this end we give an analogue of
Proposition~\ref{prop:int-pair-phasefunction} in the setting of intersecting
systems. We first need a definition of a phase function parametrizing an
intersecting system $\bmLambda$, locally near a point $q \in \bmLambda$. Notice
that either $q$ is in only one of the $\Lambda_i$; or in one of the four-fold intersections
$\Lambda_0 \cap \Lambda_1$, $\Lambda_0 \cap \Lambda_2$, $\Lambda_1 \cap \Lambda_3$, or $\Lambda_2 \cap \Lambda_3$, and disjoint from the other two;    or in the 4-fold
intersection $\bigcap_{i=0}^3 \Lambda_i$. Since these four pairs $(\Lambda_i, \Lambda_j)$
 form intersecting pairs of Lagrangian submanifolds in
the sense of the previous subsection, the only case in which we have not already
defined a local parametrization is in the case that
$q \in \bigcap_{i=0}^3 \Lambda_i$.

\begin{definition}
  \label{def:Leg-sys-param}
  Let $\bmLambda = (\Lambda_0, \Lambda_1, \Lambda_2, \Lambda_3)$ be a system of
  intersecting Lagrangian submanifolds, and choose a point
  $q \in \bigcap_{i=0}^3 \Lambda_i$ in their intersection.  We say that $\phi$
  is a local parametrization of $\bmLambda$ near $q$ if it is a function
  $\phi(x, \theta, s_1, s_2)$, defined in a neighbourhood of
  $(x_0, \theta_0, 0,0) \subseteq M \times (\RR^k \setminus \{ 0 \}) \times
  \bbR_{\geqslant 0} \times \bbR_{\geqslant 0}$
  and homogeneous of degree 1 in $\theta$ such that
  \begin{itemize}
  \item $d_{\theta, s_1, s_2} \phi(x_0, \theta_0, 0,0) = 0$, and
    $q = (x_0, d_x \phi(x_0, \theta_0, 0,0))$;
  \item the differentials
    \begin{equation}\label{difflinind}
      d_{x, \theta} \bigg( \frac{\partial \phi}{\partial \theta_i} \bigg), \quad
      d_{x, \theta} \bigg( \frac{\partial \phi}{\partial s_1} \bigg), \quad
      \text{and} \quad d_{x, \theta} \bigg( \frac{\partial \phi}{\partial s_2} \bigg)
    \end{equation}
    in the $(x, \theta)$ directions are linearly independent at
    $(x_0, \theta_0, 0,0)$;
  \item the map
    \begin{equation}
      C_0 \defeq \{ (x, \theta) : d_\theta \phi(x, \theta, 0,0) = 0 \} \mapsto
      \{ (x, d_x \phi(x, \theta, 0,0)) \} \subseteq T^* X
      \label{C04bij}\end{equation}
    is a local diffeomorphism from $C_0$ onto a neighbourhood of $q$ in
    $\Lambda_0$;
  \item the map
    $$
    C_1 \defeq \{ (x, \theta, s_1) : d_{\theta,s_1} \phi(x, \theta, s_1, 0) = 0,
    \ s_1 \geq 0 \} \mapsto \{ (x, d_x \phi(x, \theta, s_1, 0)) \} \subseteq T^*
    X
    $$
    is a local diffeomorphism from $C_1$ onto a neighbourhood of $q$ in
    $\Lambda_1$;
  \item the map
    $$
    C_2 \defeq \{ (x, \theta, s_2) : d_{\theta,s_2} \phi(x, \theta, 0, s_2) = 0,
    \ s_2 \geq 0 \} \mapsto \{ (x, d_x \phi(x, \theta, 0, s_2)) \} \subseteq T^*
    X
    $$
    is a local diffeomorphism from $C_2$ onto a neighbourhood of $q$ in
    $\Lambda_2$;
  \item the map
    \begin{multline*}
      C_3 \defeq \{ (x, \theta, s_1, s_2) : d_{\theta, s_1, s_2} \phi(x, \theta,
      s_1, s_2) = 0, \ s_1 \geq 0, s_2 \geq 0 \} \\
      \mbox{} \mapsto \{ (x, d_x \phi(x, \theta, s_1, s_2)) \}
    \end{multline*}
    is a local diffeomorphism from $C_3$ onto a neighbourhood of $q$ in
    $\Lambda_3$.
  \end{itemize}
\end{definition}

\begin{proposition}
  \label{prop:int-sys-phasefunction}
  (i) Let
  $\bmLambda = (\Lambda_0, \Lambda_1, \Lambda_2, \Lambda_3) \subseteq T^* X$ be
  a system of intersecting Lagrangian submanifolds, and let $q$ be a point in
  the intersection. Then there exists a local parametrization of $\bmLambda$
  near $q$.

  (ii) Let $\phi$, defined in a neighbourhood $U$ of $(x_0, \theta_0, 0, 0)$ be
  a local parametrization of $\bmLambda$ near $q$. Let $a(x, \theta, s_1, s_2)$
  be a classical symbol of order $m - \frac{k}{2} + 1 + \frac{n}{4}$ in the
  $\theta$-variables which is compactly supported in $U$.  Then the oscillatory
  integral
  \begin{equation}
    (2\pi)^{-\frac{k}{2}-\frac{n}{4}-\frac{1}{2}} \int \int_0^\infty
    \int_0^\infty e^{i\phi(x, \theta,
      s)} a(x, \theta,
    s) \, ds_1 \, ds_2  \, d\theta \, |dx|^{\frac{1}{2}}
    \label{intleg4m}\end{equation}
  is in $I^m(\bmLambda)$.
\end{proposition}

\begin{proof} The Proposition is proved in the same way as
  Proposition~\ref{prop:int-pair-phasefunction}.
\end{proof}

\begin{remark}\label{rem:4ILS}
  For a given phase function $\phi(x, \theta, s_1, s_2)$ to parametrize
  \emph{some} system of four intersecting Lagrangian submanifolds, locally near
  $(x_0, \theta_0, 0, 0)$, it is sufficient that it satisfies
  $d_{\theta, s_1, s_2}\phi(x_0, \theta_0, 0, 0) = 0$ and condition
  \eqref{difflinind}. Then the sets $\Lambda_i$, $i = 0 \dots 3$, defined as the
  image of $C_i$ in Definition~\ref{def:Leg-sys-param}, are automatically
  Lagrangian submanifolds satisfying the geometric conditions to form a system
  in the sense of Definition~\ref{def:4ILS}.
\end{remark}

We next write down an expression for the symbol of the oscillatory integral
\eqref{intleg4m} at $\Lambda_0$. As in the previous section, we restrict to
conormal bundles, in which case the Maslov bundle is canonically trivial. We
write $\lambda = (\lambda_1, \dots, \lambda_n)$ for coordinates on $C_0$ which
we identify with $\Lambda_0$ via \eqref{C04bij}.

\begin{proposition}\label{prop:symbol-4}
  Using the notation of \eqref{intleg4m}, suppose now that $\Lambda_0$ is the
  conormal bundle of a codimension one submanifold $M_0$ and $M_1$. Then the
  symbol of \eqref{intleg4m} at $\Lambda_0$ is given by
  \begin{equation}
    - (2\pi)^{-1} e^{\frac{i\pi \sigma}{4}} \left. \left[ \frac{  a(x, \theta,
          0,0)}{\phi_{s_1}(x, \theta, 0, 0)\phi_{s_2}(x, \theta, 0, 0)} \right]
    \right|_{C_0}
    |d\lambda|^{\frac{1}{2}} \left| \frac{ \partial (\lambda,
        \phi_\theta)}{\partial (x, \theta)}\right|^{-\frac{1}{2}}
    \label{symbol-Lambda04} \end{equation}
  where $\sigma$ is the signature of the Hessian $\phi''_{\theta\theta}$ at $s_1 = s_2=0$, and $(\lambda', \phi_{s_1}, \phi_{s_2})$ are local coordinates on $C_0$.
\end{proposition}


\section{The microlocal structure of the wave propagator on $C_{4\pi}$}
\label{sec:micro-structure}

We now specialize to the cone $C_{4\pi}$, where we will carry out the actual
analysis of the sine propagator near the singular set.  Let us first pause for a
moment to highlight some features of the cone $C_{4\pi}$.  First, and perhaps
most important, the interior $C_{4\pi}^\circ$ is equivalent to the double
cover of the punctured plane $\bbR^2 \setminus \{ (0,0) \}$.  As a result, the
Schwartz kernel
$\bmE \defeq \calK\!\left[ \frac{\sin(t \sqrt\Delta)}{\sqrt\Delta} \right]$ has
a particularly simple description in this setting (cf.~
\cite{CheTay2}*{p.~448-9}):
\begin{subequations}
  \begin{equation}
    \label{eq:4pi-sine-SK-region-1}
    \bmE(t,r_1,\theta_1;r_2,\theta_2) \equiv 0
  \end{equation}
  when $0 < t < \dist(r_1,\theta_1;r_2,\theta_2)$;
  \begin{equation}
    \label{eq:4pi-sine-SK-region-2}
    \bmE(t,r_1,\theta_1;r_2,\theta_2) = \frac{1}{2\pi} \left[ t^2 - \left( r_1^2
        + r_2^2 - 2 r_1 r_2 \cos(\theta_1 - \theta_2) \right)
    \right]^{-\frac{1}{2}}
  \end{equation}
  when $\dist(r_1,\theta_1;r_2,\theta_2) < t < r_1 + r_2$; and
  \begin{equation}
    \label{eq:4pi-sine-SK-region-3}
    \bmE(t,r_1,\theta_1;r_2,\theta_2) = \frac{1}{4\pi} \left[ t^2 - \left( r_1^2
        + r_2^2 - 2 r_1 r_2 \cos(\theta_1 -
        \theta_2) \right) \right]^{-\frac{1}{2}},
  \end{equation}
\end{subequations}
when $t > r_1 + r_2$.  In particular, the jump discontinuity across the
diffractive front $\{t = r_1 + r_2\}$ is readily apparent on
$C_{4\pi}$.\footnote{Note that \cite{CheTay2}*{eq.~(4.7)} contains a sign error
  that we have corrected here.}  Second, a seemingly incidental fact that will
be important as we continue is that constant vector fields are well-defined on
$C_{4\pi}$ (and indeed any cone with cone angle an integral multiple of $2\pi$,
i.e., the finite-sheeted covering spaces of the punctured plane).

\subsection{The `moving conical point' method}
\label{sec:inhomog-4pi}

Our technique for determining the structure of the wave kernel is the `moving
conical point' method. Given two points $\bmq_1^*$ and $\bmq_2^*$ in
$C^\circ_{4\pi}$, and a positive time $t^*$, we want to determine
$\bmE(t, \bmq_1, \bmq_2)$ for $(t, \bmq_1, \bmq_2)$ in a neighbourhood of
$(t^*, \bmq_1^*, \bmq_2^*)$. To do this, we imagine that we can move the conical
point (that is, the place where the two copies of $\RR^2$ are ramified) along a
straight line, in a direction such that moves it `in between' $\bmq_1^*$ and
$\bmq_2^*$, and then far away (i.e., at a distance $S$ much larger than
$t^*$). This means that the angle between $\bmq_1^*$ and $\bmq_2^*$ tends to
$2\pi$, so the distance between them will be $2S + O(1) \gg t^*$. Then, by
finite propagation speed, after the cone point is so shifted, the wave kernel at
$(t^*, \bmq_1^*, \bmq_2^*)$ will vanish. We then express the kernel using the
fundamental theorem of calculus:
$$
\bmE(t, \bmq_1, \bmq_2) = -\int_0^S \frac{d}{ds} \bmE^s(t, \bmq_1, \bmq_2) \,
ds,
$$
where $\bmE^s(t, \bmq_1, \bmq_2)$ is the wave kernel where the cone point has
been shifted a distance $s$ in our chosen direction. Thus, if we can understand
the derivative of $\bmE^s$ with respect to $s$, then we can compute
$\bmE = \bmE^0$. The reason we can expect the derivative $\frac{d}{ds} \bmE^s$
to be simpler than $\bmE^s$ itself is that the singularity at the \emph{direct}
front is independent of $s$, so $\frac{d}{ds} \bmE^s$ should be associated
purely to \emph{diffractive} behaviour. The rest of this section is devoted to
implementing this method.

To do this in a rigorous manner, rather than moving the cone point, we instead
translate the points $\bmq_1$ and $\bmq_2$ on the cone (in the opposite direction --- see Figure~\ref{fig:moving-conical-point}) using the flow of a
constant vector field $\vec{X} \in \mathcal{V}(C_{4\pi}^\circ)$, which we choose
in a direction such that the two half-lines parallel to $X$ through $\bmq_1$ and
$\bmq_2$ pass on different sides of the cone point; in particular, neither meets
the cone point.

\begin{figure}
  \centering

 \includegraphics{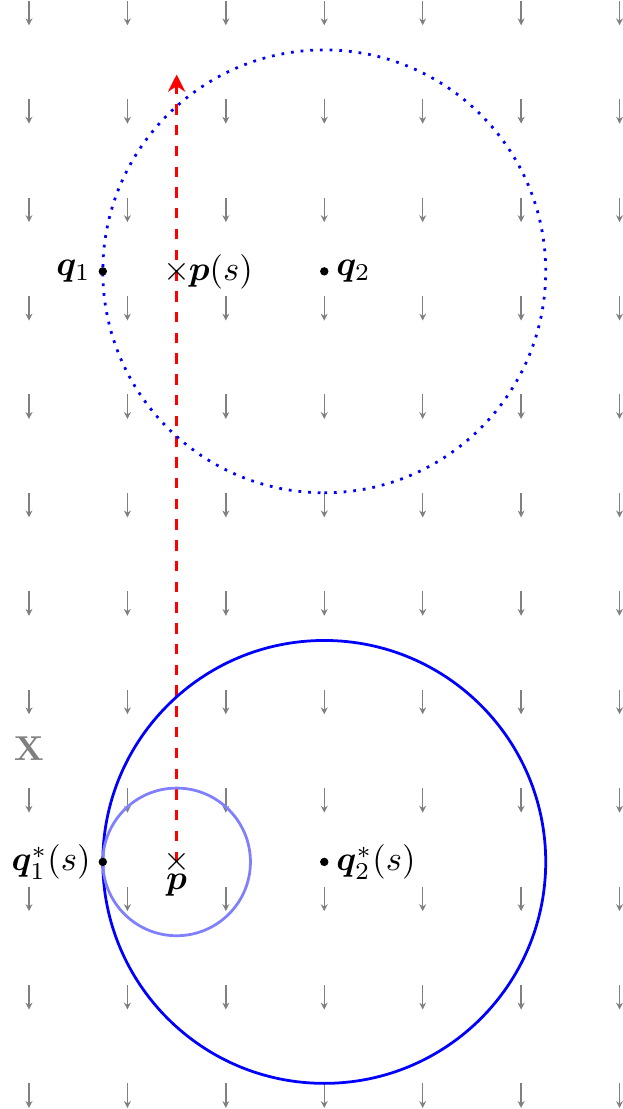}

 \caption{Moving the conical point.  Shown are the singular support of $\bmE^s(t,\bmq_1;\bmq_2) = \bmE(t,\bmq_1(s); \bmq_2(s))$ (solid circles) and the singular support of $\bmE(t,\bmq_1;\bmq_2)$ (dotted circles) as the moving conical point $\bmp(s)$ travels along the flow of $-\vec{X}$; the branch cut is depicted as the red dashed line.}
 \label{fig:moving-conical-point}
\end{figure}

We set $\varphi^s = \varphi^s_{\vec{X}}$ to be the associated flow, the group of
\emph{local}\footnote{The time interval for which $\varphi^s$ is defined depends
  on the starting point; in particular, the points along the reverse flowout of
  $\bmp$ can only be evolved forward for finite time---until they reach $\bmp$.}
diffeomorphisms given by time-$s$ translation along $\vec{X}$.  Using
$\varphi^s$ we assemble the kernel spacetime flow for $\vec{X}$, which is the
group of locally-defined diffeomorphisms $\Phi^s$ on
$\bbR_t \times C_{4\pi}^\circ \times C_{4\pi}^\circ$ given by
\begin{equation*}
  \Phi^s(t,\bmq_1;\bmq_2) \defeq \big( t, \varphi^s(\bmq_1) ; \varphi^s(\bmq_2) \big) =
  \big(t, \bmq_1 + s \, \vec{X} ; \bmq_2 + s \, \vec{X} \big).
\end{equation*}
Consider the distribution
\begin{equation}
  \label{eq:moved-SK}
  \Xi_s\defeq \chi \del_s \! \left[ \big( \Phi^s
    \big){}^* \, \bmE \right] ,
\end{equation}
where
$\chi = \chi(\bmq_1;\bmq_2) \in \calC^\infty(C_{4\pi}^\circ \times
C_{4\pi}^\circ)$
is a smooth function that vanishes near the cone point. Its role is to ensure
that $(\Phi^s)^* \bmE$ is well defined on the support of $\chi$; that is, $\chi$
must be chosen so that it vanishes in the set obtained by translating a small
ball centered at $\bmp$ in the $\vec{X}$ direction, and is identically $1$ in a
neighbourhood of the set
$\{ (\bmq_1 + s \vec{X}, \bmq_2 + s \vec{X}) : s \in \RR, (x_1, x_2) \in U \}$,
where $U$ is a suitably small neighbourhood of $(\bmq_1^*, \bmq_2^*)$. Then for
$(\bmq_1, \bmq_2) \in U$ we have
\begin{equation}
  \label{eq:Xi_s-expansion}
  \Xi_s  =  \chi \cdot
  \del_s \! \left[ \big( \Phi^s \big){}^* \bmE \right] = \del_s \! \left[ \big( \Phi^s \big){}^* \bmE \right] .
\end{equation}

Set
\begin{equation}
  \Upsilon_s(t,\bmq_1;\bmq_2) \defeq \chi(\bmq_1; \bmq_2) \cdot \del_s \! \left[ \big( \Phi^s \big){}^* \bmE \right] \! (t, \bmq_1; \bmq_2);
  \label{eq:Upsilon-defn}\end{equation}
this is the precise version of the quantity $\frac{d}{ds} \bmE^s$ in the
heuristic discussion above.  Thus, we have
\begin{equation}
  \label{eq:W(t)-moving-conical-point-formula}
  \bmE(t,\bmq_1; \bmq_2) = - \int_0^S \Upsilon_s(t, \bmq_1; \bmq_2) \, ds, \quad
   (\bmq_1, \bmq_2 ) \in U,
\end{equation}
provided that $(\Phi^S)^* \bmE(t, \bmq_1; \bmq_2) = 0$ as discussed above.

When $s = 0$, we calculate that
\begin{equation}
  \label{eq:Upsilon_s-pullback-formula}
  \Upsilon_0(t,\bmq_1;\bmq_2) = \vec{X}_1 \bmE(t,\bmq_1;\bmq_2) + \vec{X}_2 \bmE(t,\bmq_1;\bmq_2)
\end{equation}
with $\vec{X}_j$ denoting $\vec{X}$ acting in the $\bmq_j$-variable, and for
general $s$ we have
\begin{equation}
  \label{eq:moved-SK-pullback}
  \Upsilon_s(t,\bmq_1;\bmq_2) = \big( \Phi^s \big){}^*
  \Upsilon_0(t,\bmq_1;\bmq_2) , \quad (\bmq_1, \bmq_2 ) \in U,
\end{equation}
since the vector field $\vec{X}$ is constant.  Pairing $\Upsilon_0$ with a test
function $\psi \in \calC^\infty_\upc(C_{4\pi}^\circ)$ in the $\bmq_2$-variable,
we then integrate by parts to obtain
\begin{equation*}
  \begin{aligned}
    \left< \Upsilon_0, \psi \right>_{\bmq_2} &= \left< \vec{X}_1 \bmE, \psi
    \right>_{\bmq_2} + \left< \vec{X}_2 \bmE, \psi \right>_{\bmq_2}
    \\
    &= \left< \vec{X}_1 \bmE, \psi \right>_{\bmq_2} - \left<
      \bmE, \vec{X} \psi \right>_{\bmq_2} \\
    &= \big(\vec{X} \circ \bfW(t)\big) \psi - \big(\bfW(t) \circ \vec{X}\big)
    \psi \\
    &= \left[ \vec{X} , \bfW(t) \right] \psi .
  \end{aligned}
\end{equation*}
Thus, $\Upsilon_0$ is the Schwartz kernel of the commutator
$\left[ \vec{X}, \bfW(t) \right]$ of the constant vector field with the sine
propagator.  Note this distribution is everywhere well-defined.

A quick computation now yields the operator identity
\begin{equation*}
  \boxop \circ \left[ \vec{X}, \bfW(t) \right] = \left[ \vec{X}, \Delta
  \right] \circ \bfW(t) ,
\end{equation*}
and hence Duhamel's principle implies\footnote{There is a minus sign in the
  formula because our operator $\boxop$ is
  $D_t^2 - \Delta = - \partial_t^2 - \Delta$, while the usual Duhamel formula is
  written for an operator with a positive sign in front of $\partial_t^2$.}
\begin{equation}
  \label{eq:commutator-Duhamel-identity}
  \left[ \vec{X}, \bfW(t) \right] = -\int_{s = 0}^t \bfW(t-s) \circ \left[
    \vec{X}, \Delta \right] \circ \bfW(s) \, ds ,
\end{equation}
where we recall the Schwartz kernel of these operators is $\Upsilon_0$. Using
\eqref{eq:commutator-Duhamel-identity}, we will show that $\Upsilon_0$ is a
multiple of $\delta(t - r_1 - r_2)$, hence a purely diffractive Lagrangian
distribution.  First, we must understand better the commutator
$\left[ \vec{X}, \Delta \right]$.  This is the aim of the next subsection.

\subsection{Distributions supported at the cone point and commutators}
\label{sec:distns-commutators}

To make full use of the expression \eqref{eq:commutator-Duhamel-identity}, we
need to understand explicitly the Schwartz kernel of the commutator
$\left[ \vec{X}, \Delta \right]$.  This requires a brief detour through the
spectral theory of the Laplacian on $C_{4\pi}$, and, in particular, a discussion
of the failure of essential self-adjointness of the Laplace-Beltrami operator
$\Delta_g$ on $C^\circ_{4\pi}$.

Let $H^s(C_{4\pi})$ denote the usual Sobolev spaces on $C_{4\pi}$, defined as
\begin{equation}
  \label{eq:Sobolev-spaces}
  H^k(C_{4\pi}) \defeq \left\{ u \in L^2(C_{4\pi}) : \Diff^k(C_{4\pi}) \cdot u
    \in L^2(C_{4\pi}) \right\}
\end{equation}
for integers $k \in \bbZ_{\geqslant 0}$ and extended to all real orders by
duality and interpolation.  An exercise (essentially the same as a more standard
calculation on $\RR^2$, where the same result holds; cf.~Chapter I.5 of
\cite{AGHKH}) shows that the closure of $\calC^\infty_\upc(C_{4\pi}^\circ)$ in
the graph norm for $\Delta_g$,
\begin{equation*}
  \|u\|_{\Delta_g} \defeq \|u\|_{L^2} + \|\Delta_g u\|_{L^2} ,
\end{equation*}
i.e., the domain of the closure $\overline{\Delta_g}$ of $\Delta_g$, is
\begin{equation}
  \label{eq:closure-domain}
  \overline{\mathfrak{D}} \defeq \Dom\!\left( \overline{\Delta_g} \right) =
  \left\{ u \in H^2(C_{4\pi}) : u(\bmp) = 0 \right\} .
\end{equation}
Thus, if $\rho \in \calC^\infty_\upc(C_{4\pi})$ is any bump function satisfying
$\rho \equiv 1$ for $r \leqslant 1$ and $\rho \equiv 0$ for $r \geqslant 2$,
then this shows
\begin{equation*}
  H^2(C_{4\pi}) = \overline{\mathfrak{D}} \oplus \Span_\bbC\{\rho\}.
\end{equation*}
We show in Lemma~\ref{thm:domain-adjoint-Lap} that the domain of the adjoint of
this operator is
\begin{equation*}
  \overline{\mathfrak{D}}{}^* \defeq \Dom\!\left( \overline{\Delta_g}{}^*
  \right) = \overline{\mathfrak{D}} \oplus
  \Span_\bbC \! \left\{ \rho, \rho \log(r), \rho \, r^\frac{1}{2} \exp\!\left[
      \pm \frac{i}{2} \theta \right], \rho \, r^{-\frac{1}{2}} \exp\!\left[ \pm
      \frac{i}{2} \theta \right]
  \right\} .
\end{equation*}
The choice of a self-adjoint extension of $\overline{\Delta_g}$ is then the
suitable choice of a half-dimensional subspace of
$\overline{\mathfrak{D}}{}^* \big/ \overline{\mathfrak{D}}$ (cf.~\cite{ReeSim2}
for more details on self-adjoint extensions).

In our analysis, we have elected to work with the Friedrichs extension
$\Delta \defeq \Delta_g^{\mathrm{Fr}}$ of the Laplacian, the unique self-adjoint
extension whose domain contains the form domain (which in our setting is
$H^1(C_{4\pi})$).  We define the spaces $\calD_s$ to be the domains of real
powers of this operator:
\begin{equation}
  \label{eq:Lapn-domains}
  \calD_s \defeq \Dom\!\left( \Delta^{\frac{s}{2}} \right) .
\end{equation}
For $s > 1$, these spaces are strictly larger than the Sobolev spaces
$H^s(C_{4\pi})$.  In particular, $\calD_2$ is the Friedrichs domain itself.

To distinguish the elements of $\calD_2$ from those of
$\overline{\mathfrak{D}}$, we must examine their behavior at $\bmp$.  We do so
in the following lemma, which we prove in
Appendix~\ref{sec:proof-asymptotics-cone-point}.

\begin{lemma}
  \label{thm:distr-at-cone-exp}
  Fix a compactly supported, smooth, and radial cutoff
  $\rho \in \calC^\infty_\upc(C_{4\pi})$ which is identically $1$ near $\bmp$.
  For any function $u \in \calD_2$, there exist constants $a_{-1}$, $a_0$, and
  $a_1$ in $\bbC$ and a distribution $v \in \overline{\mathfrak{D}}$ such that
  \begin{equation}
    \label{eq:cone-point-distn-exp}
    u = \left( a_0 + a_{-1} \sqrt{r} \, \exp\!\left[ - \frac{i}{2} \, \theta
      \right] + a_1 \sqrt{r} \, \exp\!\left[ \frac{i}{2} \, \theta \right]
    \right) \rho(r) + v .
  \end{equation}
  In particular, the function
  $u - a_0 - a_{-1} \sqrt{r} \, \exp\!\left[ - \frac{i}{2} \, \theta \right] -
  a_1 \sqrt{r} \, \exp\!\left[ \frac{i}{2} \, \theta \right]$
  vanishes at $\bmp$.
\end{lemma}

\begin{remark}
  We see from Lemma~\ref{thm:distr-at-cone-exp} the system of strict inclusions
  \begin{equation*}
    \overline{\mathfrak{D}} \subsetneq H^2(C_{4\pi}) \subsetneq \calD_2 \subsetneq
    \overline{\mathfrak{D}}{}^* .
  \end{equation*}
\end{remark}

Using this lemma, we see that the Friedrichs extension exactly corresponds to
the choice of the functions
\begin{equation*}
  \varphi_0(r,\theta) \defeq 1, \quad \varphi_{-1}(r,\theta) \defeq \sqrt{r}
  \exp\!\left[ - \frac{i}{2} \, \theta \right], \quad \text{and} \quad
  \varphi_{+1}(r,\theta) \defeq \sqrt{r} \exp\!\left[ \frac{i}{2} \theta \right]
\end{equation*}
as the models for the admissible asymptotics at $\bmp$.  Given a function $u$ in
$\calD_2$, we define the distributions $L_j$ for $j = -1$, $0$, or $+1$ as
\begin{equation}
  \label{eq:cone-point-distn-basis-elts}
  L_j(u) \defeq a_j
\end{equation}
in terms of  the expansion \eqref{eq:cone-point-distn-exp}.  Note that the expansion
\eqref{eq:cone-point-distn-exp} is independent of the choice of the cutoff
$\rho$, for the difference of any two such cutoffs is compactly supported in
$C^\circ_{4\pi}$ and is thus in $\overline{\mathfrak{D}}$.  Hence, the
distributions $L_j$ are well-defined elements of $\overline{\mathfrak{D}}_{-2}$.
Equivalently, we may define the $L_j$'s
using the angular spectral projectors
\begin{equation}
  \label{eq:angular-spectral-projectors}
  \left[\Pi_j u\right]\!(r) \defeq \frac{1}{\sqrt{4\pi}} \int_{\bbR / 4\pi \bbZ}
  u(r,\theta) \, \exp\!\left[ - \frac{i j}{2} \, \theta \right] d\theta ,
\end{equation}
and a straightforward computation shows that
\begin{equation}
  \label{eq:L0pm1}
  L_0(u) = \frac{1}{\sqrt{4\pi}} \lim_{r \downarrow 0} \big[\Pi_0 u\big]
  (r) \quad \text{and} \quad
  L_{\pm 1}(u) = \frac{1}{\sqrt{4\pi}} \lim_{r \downarrow 0} \frac{\big[ \Pi_{\pm 1}
    u \big](r)}{\sqrt{r}} .
\end{equation}
Directly from the definition or from the above, we observe that
$L_{\pm 1}(u) = \overline{L_{\mp 1}(\overline{u})}$.

\begin{corollary}
  \label{thm:distn-cone-point-basis}
  Suppose $L$ is a distribution in $\calD_{-2}$ which is supported only at the
  cone point $\bmp$.  Then $L$ is a linear combination of $L_{-1}$, $L_0$, and
  $L_1$.
\end{corollary}

\begin{proof}
  Suppose $u$ is an element of $\calD_2$.  By \eqref{eq:cone-point-distn-exp} we
  have
  \begin{equation*}
    L(u) = a_0 \, L\big(\rho(r)\big) + a_{-1} \, L\big(
    \varphi_{-1}(r,\theta) \cdot \rho(r) \big) + a_{1} \, L\big(
    \varphi_{1}(r,\theta) \cdot \rho(r) \big)
  \end{equation*}
  since $v$ being an element of $\overline{\mathfrak{D}}$ implies $L(v)$
  vanishes.
  Therefore,
  \begin{equation*}
    L = L\big(\rho(r)\big) \cdot L_0 + L\big(
    \varphi_{-1}(r,\theta) \cdot \rho(r) \big) \cdot L_{-1} + L\big(
    \varphi_{1}(r,\theta) \cdot \rho(r) \big) \cdot L_{1},
  \end{equation*}
  showing that $u$ is a linear combination of $L_0$, $L_{-1}$, and $L_1$ as
  claimed.
\end{proof}

Returning to the commutator $\left[\vec{X},\Delta\right]$, let us observe
\begin{equation*}
  \vec{X} : \calD_2 \centernot\To \calD_1 = H^1(C_{4\pi})
\end{equation*}
since $\calD_2$ is not contained in $H^2(C_{4\pi})$. On the other hand, since
$H^1(C_{4\pi})$ is contained in $\calD_2$, we certainly have
$\vec{X} : \calD_2 \to L^2(C_{4\pi})$, and hence, by duality, also
$\vec{X} : L^2(C_{4\pi}) \to \calD_{-2}$. Therefore, for any $u \in \calD_2$,
the commutator $\left[ \vec{X}, \Delta \right]$ is in $\calD_{-2}$.
On the other hand, if $u$ is compactly supported in $C_{4\pi}^\circ$, then the
action of $\Delta$ on $u$ is the same as the Euclidean Laplacian acting on $u$.
Since the Euclidean Laplacian commutes with constant vector fields, this implies
$\left[ \vec{X}, \Delta \right] u = 0$.
Therefore, the distributional support of $\left[ \vec{X}, \Delta\right] u$ for
any $u \in \calD_2$ is at most the cone point $\bmp$, and thus it fits into the
framework of Corollary~\ref{thm:distn-cone-point-basis}.

\begin{proposition}
  \label{thm:commutator-SK}
  Let $\vec{X} = X_w \, \del_w + X_\wbar \, \del_\wbar$ be a constant vector
  field on $C_{4\pi}$, written in terms of the complex coordinate
  $w = x + iy = re^{i\theta}$.  Then for any distribution $u \in \calD_k$ for
  $k \geqslant 2$, we have
  \begin{equation}
    \label{eq:commutator-SK}
    \left[ \vec{X}, \Delta \right] u = -2\pi \big( X_w \, L_1(u) \cdot
    L_{1} + X_\wbar \, L_{-1}(u) \cdot L_{-1} \big)
  \end{equation}
\end{proposition}

\begin{proof}
  Consider the bilinear pairing
$$
\big\langle [\vec{X}, \Delta] u, v \big\rangle, \quad u, v \in \calD_2.
$$
The discussion above shows that this is well defined for all $u, v \in \calD_2$.
It is clear that this pairing vanishes if either $u$ or $v$ lie in
$\overline{\mathfrak{D}}$.  So to compute the pairing, it suffices to consider
$u$ and $v$ to be linear combinations of the functions $\zeta_0 = \rho(|w|^2)$,
$\zeta_{-1}(w) \defeq \wbar^\frac{1}{2} \, \rho(|w|^2)$ and
$\zeta_1(w) \defeq w^\frac{1}{2} \, \rho(|w|^2)$. In fact, the pairing also
vanishes if either $u$ or $v$ are $\zeta_0$ since this is equal to a constant in
a neighbourhood of the cone point, hence vanishes near the cone point after the
application of either $\vec{X}$ or $\Delta$. So we need only consider $u$ and
$v$ equal to a combination of $\zeta_{\pm 1}$.

First, let $\vec{X} = \del_w$. For this $\vec{X}$, consider the action of the
operator $[\vec{X}, \Delta]$ on a Fourier mode $e^{ij\theta}$, for a
half-integer $j$.  Since the Fourier modes are eigenfunctions of
$\Delta_{\bbS^1_{4\pi}}$, and since $\del_w$ maps $e^{ij\theta}$ to a multiple
of $e^{i(j-1)\theta}$, the same property is true of $[\vec{X}, \Delta]$. It
follows that the only nonzero combination with $\vec{X} = \del_w$ is
$$
\big\langle [\del_w, \Delta] \zeta_1, \zeta_1 \big\rangle.
$$
Similarly, when $\vec{X} = \del_\wbar$, the only nonzero combination occurs when
$u=v=\zeta_{-1}$. In view of these considerations, to establish
\eqref{eq:commutator-SK}, it suffices to show that
\begin{equation}\begin{aligned}
    \big\langle [\del_w, \Delta] \zeta_1, \zeta_1 \big\rangle &= -2\pi , \\
    \big\langle [\del_\wbar, \Delta] \zeta_{-1}, \zeta_{-1} \big\rangle &=
    -2\pi.
  \end{aligned}\label{eq:commutator-zetapm1}\end{equation}
In fact, as the calculations are similar, we only prove the first.

Since we are using the bilinear pairing we have
\[
\big\langle [\del_w, \Delta] \zeta_1, \zeta_1 \big\rangle=-2 \big\langle
\Delta \zeta_1, \del_w\zeta_1 \big\rangle= -2\int_{C_{4\pi}} \Delta \zeta_1
\del_w\zeta_1\, dS
\]
where $dS$ denotes the Euclidean area element.  Using Stokes formula we have:
\[
\int_{r=\varepsilon} (\del_w\zeta_1)^2\, dw =i \int_{r\geq\varepsilon}
\Delta \zeta_1 \del_w\zeta_1\, dS.
\]
For $\varepsilon$ small enough we thus obtain
\[
\int_{r\geq\varepsilon} \Delta \zeta_1 \del_w\zeta_1\, dS= \pi.
\]
The claim follows.
\end{proof}

\subsection{The differentiated wave propagator on $C_{4\pi}$}
\label{sec:diff-wave-prop}
We now apply the formula \eqref{eq:commutator-SK} for the Schwartz kernel of the
commutator $\left[ \vec{X}, \Delta \right]$ to the Duhamel formula
\eqref{eq:commutator-Duhamel-identity} to compute the distribution $\Upsilon_0$.
Writing $\vec{X}$ in complex coordinates, i.e.,
$\vec{X} = X_w \, \del_w + X_{\wbar} \, \del_{\wbar}$, this yields
\begin{multline}
  \label{eq:diff-SK-1}
  \Upsilon_0(t,\bmq_1;\bmq_2) = 2\pi \int_{s = 0}^t \Big\{ X_w \big[ \bfW(t - s)
  L_1 \big](\bmq_1) \cdot \big[ L_{1} \circ \bfW(s) \big](\bmq_2) \\
  \mbox{} + X_{\wbar} \big[ \bfW(t - s) L_{-1} \big](\bmq_1) \cdot \big[ L_{-1}
  \circ \bfW(s) \big](\bmq_2) \Big\} \, ds .
\end{multline}
In particular, this shows $\Upsilon_0$ is an integral superposition of tensor
products of the distributions
\begin{equation}
  \label{eq:fundamental-spherical-waves}
  \ell_j(t) \defeq \bfW(t) \, L_j
\end{equation}
obtained from evolving the distributions $L_j$ under the sine flow $\bfW(t)$.
(Note that the self-adjointness of $\bfW(t)$, and the fact that its kernel is
real, implies $\bfW(t) \, L_j = L_j \circ \bfW(t)$, so we only need to work with
the evolved distributions $\ell_j(t)$.)  Since the $L_j$'s are supported only at
the cone point $\bmp$, we should expect the propagated distributions $\ell_j(t)$
to be spherical waves emanating out from $\bmp$, i.e., they should be
\emph{diffractive}-type waves.  As the next lemma shows, this is indeed the
case.

\begin{lemma}
  \label{thm:basic-diff-waves-structure}
  Let $t > 0$.  The distributions $\ell_1(t)$ and $\ell_{-1}(t)$ on $C_{4\pi}$
  are given explicitly by
  \begin{equation}
    \ell_{\pm 1}(t) =  \frac1{4\pi \sqrt{r} } \delta(t - r) \exp\!\left[
      \mp \frac{i}{2} \, \theta \right].
    \label{l1l-1}\end{equation}
\end{lemma}

\begin{proof}[Proof of Lemma~\ref{thm:basic-diff-waves-structure}]
  It suffices to prove the lemma for $\ell_1(t)$; the statement for
  $\ell_{-1}(t)$ is similar and follows by complex conjugation.

  Recall the spectral projector form of the definition of $L_1$ (see
  \eqref{eq:L0pm1} and \eqref{eq:angular-spectral-projectors}) and
  \begin{equation*}
    L_1(u) = \lim_{r \downarrow 0}  \frac{1}{4\pi \sqrt{r}} \int_{\bbR / 4 \pi
      \bbZ} u(r,\theta) \, \exp\!\left[ - \frac{i}{2} \theta \right] d\theta .
  \end{equation*}
  To compute the action of $\bfW(t)$ on $L_1$, we use Cheeger's functional
  calculus on metric cones \cite{CheTay1}; this expresses $\bmE$ as the sum
  \begin{equation*}
    \bmE(t,r_1,\theta_1;r_2,\theta_2) = \frac{1}{4\pi} \sum_{j \in \bbZ}
    \exp\!\left[ \frac{ij}{2} \,
      (\theta_1 - \theta_2) \right] \int_{\lambda = 0}^\infty \frac{\sin(\lambda
      t)}{\lambda} \, J_{\frac{|j|}{2}}(\lambda r_1) \,
    J_{\frac{|j|}{2}}(\lambda r_2) \, \lambda \, d\lambda
  \end{equation*}
  over the angular modes of $\Delta$.  Since $L_1$ vanishes except at the
  $j = 1$ mode in this sum, we have the following simple formula for the action
  of $\ell_1(t)$.
  \begin{multline*}
    \big[\ell_1(t)\big](u) \\
    \mbox{} = \lim_{r_1 \downarrow 0} \frac{1}{(4\pi)^2 \sqrt{r_1}}
    \int_{\theta_1 = 0}^{4\pi} \int_{r_2 = 0}^\infty \int_{\theta_2 = 0}^{4\pi}
    \left\{ \int_{\lambda = 0}^\infty \frac{\sin(\lambda t)}{\lambda} \,
      J_{\frac{1}{2}}(\lambda r_1) \,
      J_{\frac{1}{2}}(\lambda r_2) \, \lambda \, d\lambda \right\} \\
    \mbox{} \times \exp\!\left[ - \frac{i}{2} \, \theta_1 \right] \exp\!\left[
      \frac{i}{2} (\theta_1 - \theta_2) \right] u(r_2,\theta_2) \, d\theta_1 \,
    r_2 \, dr_2 d\theta_2
  \end{multline*}
  Performing the $\theta_1$-integral, this simplifies to
  \begin{multline*}
    \big[\ell_1(t)\big](u) = \lim_{r_1 \downarrow 0} \frac{1}{4\pi \sqrt{r_1}}
    \int_{r_2 = 0}^\infty \int_{\theta_2 = 0}^{4\pi} \left\{ \int_{\lambda =
        0}^\infty \frac{\sin(\lambda t)}{\lambda} \, J_{\frac{1}{2}}(\lambda
      r_1) \,
      J_{\frac{1}{2}}(\lambda r_2) \, \lambda \, d\lambda \right\} \\
    \mbox{} \times \exp\!\left[ - \frac{i}{2} \, \theta_2 \right]
    u(r_2,\theta_2) \, r_2 \, dr_2 d\theta_2 .
  \end{multline*}
  We now substitute the explicit formula
  $J_{\frac{1}{2}}(z) = \left[\frac{2}{\pi z}\right]^\frac{1}{2} \sin(z)$ into
  the above, giving
  \begin{equation*}
    \ell_1(t) = \lim_{r_1 \downarrow 0} \frac{1}{2 \pi^2 \sqrt{r_2}}
    \int_{\lambda = 0}^\infty \sin(\lambda t) \, \frac{\sin(\lambda
      r_1)}{\lambda r_1} \, \sin(\lambda r_2) \, \exp\!\left[ - \frac{i}{2}
      \theta_2 \right] \, d\lambda .
  \end{equation*}
  By pairing with a test function and using dominated convergence, we see that
  this is equivalent (in the sense of distributions) to the expression
  \begin{equation*}
    \ell_1(t) = \frac{1}{2 \pi^2 \sqrt{r_2}}
    \int_{\lambda = 0}^\infty \sin(\lambda t) \sin(\lambda r_2)
    \exp\!\left[ - \frac{i}{2} \theta_2 \right] d\lambda .
  \end{equation*}

  To conclude the proof, we observe
  \begin{equation*}
    \int_{\lambda = 0}^\infty e^{- i
      \lambda (t + r)} \, d\lambda = \int_{\lambda = -\infty}^0 e^{i \lambda(t +
      r)} \, d\lambda .
  \end{equation*}
  This implies, dropping the subscripts from the base variables and replacing
  the sine functions by their complex exponential definitions, that
  \begin{equation}
    \label{eq:spherical-wave-distn-1}
    \begin{gathered}
      \ell_1(t) = - \frac{1}{8 \pi^2 \sqrt{r}} \int_{\lambda = -\infty}^\infty
      \left\{ e^{i \lambda (t + r)} - e^{i \lambda (t - r)} \right\}
      \exp\!\left[
        - \frac{i}{2} \, \theta \right] d\lambda  \\
      = \frac1{4\pi \sqrt{r} } \delta(t - r) \exp\!\left[ - \frac{i}{2} \,
        \theta \right], \text{ for } t > 0.
    \end{gathered}
  \end{equation}
\end{proof}

\begin{remark} It is remarkable that, on the cone of angle $4\pi$, there are
  solutions to the wave equation, namely
  $r^{-\frac{1}{2}} \delta(t - r) e^{\pm i\theta/2}$ obeying the \emph{sharp}
  Huygen's principle, that is, supported on the light cone itself. This can be
  confirmed by direct calculation, applying the wave operator to these
  distributions.

  We also remark that one can prove Lemma~\ref{thm:basic-diff-waves-structure} without appealing to the
  Cheeger functional calculus: after verifying that the $\ell_\pm(t)$ satisfy the
  wave equation, it only remains to check that $\lim_{t \to 0} \ell_{\pm 1}(t) = 0$
  and $\lim_{t \to 0} (d/dt) \ell_{\pm 1}(t) = L_{\pm 1}$.
\end{remark}

We conclude this subsection with the proof that $\Upsilon_0$ is a Lagrangian
distribution associated to the diffractive Lagrangian relation $\Lambda^\diff$.

\begin{proposition}
  \label{thm:diff-wave-prop-structure}
  Let $t > 0$, and suppose, as in the discussion in
  Section~\ref{sec:inhomog-4pi}, that $\vec{X}$ points in the direction
  $\theta$.  Then the distribution
  $\Upsilon_0 = \calK \big[ \left[ \vec{X}, \bfW(t) \right] \big]$ is given
  explicitly in polar coordinates by
  \begin{equation}
    \Upsilon_0(t,\bmq_1;\bmq_2) = \frac1{4\pi \sqrt{r_1 r_2}}
    \delta(t-r_1-r_2) \cos\! \left( \frac{\theta_1 + \theta_2}{2} -\theta \right).
    \label{eq:Y0}\end{equation}
\end{proposition}

\begin{proof}
  We begin by rewriting the equation \eqref{eq:diff-SK-1} using the
  distributions $\ell_j(t)$:
  \begin{multline*}
    \Upsilon_0(t,\bmq_1;\bmq_2) = 2\pi \int_{s = 0}^t \Big\{ X_w \,
    \big[\ell_1(t -
    s)\big](\bmq_1) \cdot \big[ \ell_{1}(s) \big](\bmq_2) \\
    \mbox{} + X_{\wbar} \, \big[ \ell_{-1}(t - s) \big](\bmq_1) \cdot \big[
    \ell_{-1}(s) \big](\bmq_2) \Big\} \, ds .
  \end{multline*}
  We break up the integral across the sum and consider the first summand:
  \begin{equation*}
    \Upsilon_0^{w}(t,\bmq_1;\bmq_2) \defeq \int_{s = 0}^t \big[ \ell_1(t-s) \big](\bmq_1)
    \cdot \big[ \ell_{1}(s) \big](\bmq_2) \, ds .
  \end{equation*}
  Substituting our expression \eqref{eq:spherical-wave-distn-1} in for
  $\ell_1(t)$ and its conjugate for $\ell_{-1}(t)$, the above becomes
  \begin{equation}
    \label{eq:Y0w}
    \begin{gathered}
      \frac{1}{16 \pi^2 \sqrt{r_1 r_2}} \int_{s = 0}^t \delta((t-s) - r_1)
      \delta(s - r_2) \exp\!\left[ - \frac{i}{2} \, (\theta_1 +\theta_2)
      \right]  \, ds \\
      = \frac{1}{16 \pi^2 \sqrt{r_1 r_2}} \delta(t-r_1-r_2) \exp\!\left[ -
        \frac{i}{2} \, (\theta_1 +\theta_2) \right] .
    \end{gathered}
  \end{equation}
  Similarly, we have
  \begin{equation}
    \Upsilon_0^{\wbar}(t,\bmq_1;\bmq_2) = \frac{1}{16 \pi^2 \sqrt{r_1 r_2}} \delta(t-r_1-r_2)  \exp\!\left[  \frac{i}{2} \, (\theta_1 +\theta_2)
    \right] .
    \label{eq:Y0wbar}  \end{equation}
  For the vector field with direction $\theta $, we have $X_w = \exp(i\theta)=\overline{X_{\wbar}}$. Adding  \eqref{eq:Y0w} times $X_w$ to \eqref{eq:Y0wbar} times $X_{\wbar}$, and multiplying by $2\pi$, we obtain \eqref{eq:Y0}.
\end{proof}

\subsection{The full wave propagator on $C_{4\pi}$}
\label{sec:full-wave-propagator}

Having computed $\Upsilon_0(t, \bmq_1; \bmq_2)$, we return to
\eqref{eq:W(t)-moving-conical-point-formula} and compute the sine wave kernel
$\bmE(t,\bmq_1,\bmq_2)$ on $C_{4\pi}$.  Since our primary interest is in the
behaviour near a geometric diffractive geodesic, let us assume for a while that
$\theta_1$ is close to $0$ and $\theta_2$ is close to $\pi$ (so that the
diffraction angle $\theta_1-\theta_2$ is close to $-\pi$). We then choose to
move the conical point in the direction $\frac{\pi}{2}.$ This amounts to putting
$\theta=-\frac{\pi}{2}$ in the previous formulas.

Let $r_j(s), \theta_j(s)$ be the distance and angle from the point $\bmq_j$ to
the cone point shifted by a distance $s$ in the $\theta = \frac{\pi}{2}$
direction, or equivalently, from the point $\bmq_j(s)$, obtained from $\bmq_j$
by shifting a distance $s$ in the $\theta = -\frac{\pi}{2}$ direction, to the
(fixed) cone point. Notice that, in the limit $s \to \infty$, the angle between
$\bmq_1(s)$ and $\bmq_2(s)$ approaches $2\pi$. In particular, the points will be
distance $r_1(s) + r_2(s) = 2s + O(1)$ apart, in this limit. Thus, the condition
that $(\Phi^s)^* \bmE(t, \bmq_1; \bmq_2) = 0$ is valid for large $s$. Hence, we
can write, using \eqref{eq:Y0w} and
\eqref{eq:W(t)-moving-conical-point-formula},
\begin{equation}
  \label{eq:W(t)delta}
  \bmE(t,\bmq_1,\bmq_2) = \frac1{4\pi} \int_0^\infty (r_1(s) r_2(s))^{-\frac{1}{2}} \delta(t - r_1(s) - r_2(s)) \sin\! \left( \frac{\theta_1(s) + \theta_2(s)}{2} \right) ds.
\end{equation}
This can be written
\begin{equation}\label{eq:EasMU}
  \bmE(t,\bmq_1;\bmq_2)=\int_{s\geq 0}\int_{-\infty}^\infty  e^{i \phi(t,\bmq_1,\bmq_2,s,\omega)} a(t, \bmq_1,\bmq_2,s,\omega)   \, ds  \, d\omega
\end{equation}
with the following phase function and amplitude:
\begin{equation}\label{Phi-phase}
  \begin{aligned}
    \phi(\bmq_1,\bmq_2,s, \omega) & = \Big(\sqrt{x_1^2 + (y_1-s)^2} +
    \sqrt{x_2^2 + (y_2-s)^2} - t\Big)\omega,\\
    a(t,\bmq_1,\bmq_2,s,\omega) & =
    \frac{1}{8\pi^2}\cdot\big(r_1(s)r_2(s)\big)^{-\frac{1}{2}}\cdot
    \sin\!\left(\frac{\theta_1(s)+\theta_2(s)}{2}\right).
  \end{aligned}
\end{equation}
Since the phase function is a nondegenerate phase function in the sense of
Definition~\ref{def:intLeg-param}, we find that the propagator is in the
Melrose-Uhlmann class.

This construction can actually be carried out as long as
$\theta_1\in (-\frac{\pi}{2},\frac{\pi}{2})$ and
$\theta_2 \in (\frac{\pi}{2},\frac{3\pi}{2}).$ When $\theta_1$ is in the same
interval but $\theta_2$ now belongs to $(-\frac{3\pi}{2},-\frac{\pi}{2})$ (thus
containing the diffraction angle of $+\pi$), the conical point have to be moved
in the opposite direction $\theta =-\frac{\pi}{2}.$ This leads to a similar
expression. Observe however that in that case the phase is now
\begin{equation*}
  \phi(\bmq_1,\bmq_2,s, \omega) = \Big(\sqrt{x_1^2 + (y_1+s)^2} + \sqrt{x_2^2 + (y_2+s)^2} - t\Big)\omega,
\end{equation*}

In the remaining cases for which $\theta_2$ belongs respectively to
$(-\frac{\pi}{2},\frac{\pi}{2})$ and $(\frac{3\pi}{2},\frac{5\pi}{2})$ the
conical point can be moved in the $\theta=\pi$ direction. It should be noted
however that in this case the limit $s\rightarrow \infty$ of
$\bmE^s(t,\bmq_1,\bmq_2)$ is not $0$ but the free solution.

In any case, it follows that $\bmE$ is an intersecting Lagrangian distribution
in a neighbourhood of $(t, \bmq_1; \bmq_2)$. Close to the diffraction angle
$-\pi$, we use the form of the phase \eqref{Phi-phase} to determine the two
Lagrangian submanifolds.  First, when $s=0$, it is clear that $\phi |_{s=0}$
parametrizes the Lagrangian $N^* \{ t = r_1 + r_2 \} = \Lambda^\diff$. Second,
when $s=0$, $\phi$ is stationary with respect to $s$ when the cone point lies on
the straight line between $(x_1,y_1+s)$ and $(x_2, y_2+s)$. In this case, the
second derivative $\partial^2_{ss} \phi$ is nonzero, and we can eliminate the
variable $s$ by replacing it with its stationary value. In this case, the sum of
distances $\sqrt{x_1^2 + (y_1+s)^2} + \sqrt{x_2^2 + (y_2+s)^2}$ is equal to the
distance between $(x_1, y_1+s)$ and $(x_2, y_2+s)$, which is the same as the
distance between $\bmq_1$ and $\bmq_2$. So an equivalent phase function is
$(|\bmq_1-\bmq_2|-t)\omega$, and this parametrizes the conormal bundle of the
direct front, $\Lambda^\geom$.

This essentially proves
\begin{proposition}
  \label{thm:4pi-propagator-structure}
  For each \emph{fixed} $t > 0$, the sine propagator kernel $\bmE$ on $C_{4\pi}$
  is an intersecting Lagrangian distribution on
  $C_{4\pi}^\circ \times C_{4\pi}^\circ$ of order $-1$:
  \begin{equation*}
    \bmE(t) \in   r_1^{-\frac{1}{2}} \, r_2^{-\frac{1}{2}} \cdot I^{-1}
    \! \left( C_{4\pi}^\circ \times C_{4\pi}^\circ; \Lambda^\diff, \Lambda^\geom
    \right) ;
  \end{equation*}
  in particular, it has Lagrangian order $-1$ on
  $\Lambda^\geom \setminus \Lambda^\diff$ and order $- \frac{3}{2}$ on
  $\Lambda^\diff \setminus \Lambda^\geom$.
\end{proposition}

\subsection{The Cheeger-Taylor formula}
It is instructive to compute the integral \eqref{eq:W(t)delta} explicitly, and
confirm that we obtain the Cheeger-Taylor formulae for the wave kernel from
Section~\ref{sec:inhomog-4pi}. Let us consider the case in which
$\theta_1\in (-\frac{\pi}{2},\frac{\pi}{2})$ and
$\theta_2\in (\frac{\pi}{2},\frac{3\pi}{2})$

Since the functions $r_i(s)$ take the form $\sqrt{r_0^2 + (s-s_0)^2}$, they are
convex functions of $s$. Therefore, as $s$ ranges from $0$ to $\infty$, the
delta function $\delta(t - r_1(s) - r_2(s))$ can be nonzero for at most two
values of $s$. More precisely, if $t < r_1 + r_2$ and the angle between $\bmq_1$
and $\bmq_2$ is greater than $\pi$, then there are no values of $s$ for which
$t = r_1(s) + r_2(s)$, since in this case, both $r_1(s)$ and $r_2(s)$ are
increasing in $s$. On the other hand, suppose that $t < r_1 + r_2$ and the angle
between $x_1$ and $x_2$ is \emph{less} than $\pi$. We might as well assume that
$t > d(x_1, x_2)$, since otherwise the wave kernel is zero due to finite speed
of propagation. In this case, $r_1(s) + r_2(s)$ decreases until the cone point
lies directly between $x_1$ and $x_2$, when we have
$r_1(s) + r_2(s) = d(x_1, x_2) < t$, and then increases to infinity. It follows
that in this case there are two values of $s$ for which $t = r_1(s) +
r_2(s)$.
The final case is $t > r_1 + r_2$. In this case, regardless of whether
$r_1(s) + r_2(s)$ initially increases or decreases, there is always one value of
$s$ for which $t = r_1(s) + r_2(s)$.

For each value of $s$ satisfying $t = r_1(s) + r_2(s)$, we calculate the
contribution to the integral \eqref{eq:W(t)delta}. This is given by
\begin{multline}
  \frac1{4\pi} (r_1(s) r_2(s))^{-\frac{1}{2}} \big| r_1'(s) + r'_2(s)
  \big|^{-1} \sin \!\left( \frac{\theta_1(s) + \theta_2(s)}{2} \right) \\
  = \frac1{4\pi} (r_1(s) r_2(s))^{-\frac{1}{2}} \left| \sin\!\left(\theta_1(s)\right) + \sin \! \left(\theta_2(s) \right) \right|^{-1} \sin \!\left( \frac{\theta_1(s) + \theta_2(s)}{2} \right)\\
  = \frac1{4\pi} \left[(r_1(s) r_2(s)\right]^{-\frac{1}{2}} \left|
  \frac{\sin\!\left(\frac{\theta_1(s)+\theta_2(s)}{2}\right)}{\sin\!\left(
    \theta_1(s) \right) + \sin\!\left( \theta_2(s) \right)} \right|,
\end{multline}
since by choice $\frac{\theta_1(s)+\theta_2(s)}{2}\in (0,\pi).$ Using the
addition formula for $\sin \theta_1 +\sin \theta_2$ we obtain that the
contribution can be written
\[
\frac{1}{8\pi} (r_1(s)r_2(s))^{-\frac{1}{2}} \left| \cos\!\left(
\frac{\theta_1-\theta_2}{2} \right) \right|^{-1}
\]
and we want to prove that this coincides with
\[
\frac{1}{4\pi} \Big(
t^2-\big(r_1^2+r_2^2-2r_1r_2\cos(\theta_1-\theta_2)\big)\Big)^{-\frac{1}{2}},
\]
whenever the moved conical point $\bmp(s)$ lies in between $\bmq_1$ and
$\bmq_2.$ This implies that $t=r_1(s)+r_2(s)$ so that we have (we omit the
dependence on $s$)
\begin{multline}
   t^2-\big(r_1^2+r_2^2-2r_1r_2\cos(\theta_1-\theta_2)\big)\\ \mbox{} =
   2r_1r_2\big( 1+\cos(\theta_1-\theta_2)\big) = 4 r_1r_2 \cos^2\!\left(\frac{\theta_1-\theta_2}{2}\right) .
\end{multline}
The claim thus follows.

The wave kernel on $C_{4\pi}$ is therefore given by $0$, $1$ or $2$ times this
quantity, according as there are $0$, $1$ or $2$ values of $s > 0$ satisfying
$t = r_1(s) + r_2(s)$, as discussed above. This agrees with the expression
\eqref{eq:4pi-sine-SK-region-1}--\eqref{eq:4pi-sine-SK-region-3} obtained by
Cheeger-Taylor.


\section{The microlocal structure of the wave propagator on $C_\alpha$}
\label{sec:micro-structure-alpha}

We now analyze the structure of the Schwartz kernel $\bmE$ of the sine
propagator on the cone $C_\alpha$ of generic cone angle $\alpha$.  First let us
recall the definitions of the geometric and diffractive Lagrangians and their
intersection:  the geometric (or ``main'') Lagrangian is
\begin{subequations}
  \begin{equation}
    \label{eq:geom-Lagrangian-relation}
    \Lambda^\geom \defeq N^* \! \left\{ t^2 = r_1^2 + r_2^2 - 2 r_1 r_2
      \cos(\theta_1 - \theta_2) \text{ and } |\theta_1 - \theta_2| \leqslant \pi
    \right\},
  \end{equation}
  the diffractive Lagrangian is
  \begin{equation}
    \label{eq:diff-Lagrangian-relation}
    \Lambda^\diff \defeq N^* \! \left\{ t^2 = \left(r_1 + r_2\right)^2 \right\} ,
  \end{equation}
  and their intersection is the singular set
  \begin{equation}
    \label{eq:singular-set}
    \Sigma \defeq \Lambda^\geom \cap \Lambda^\diff
  \end{equation}
\end{subequations}
In particular, we note that $\pr(\Sigma) = \left\{ \text{$t^2 = (r_1 + r_2)^2$
    and $\theta_1 - \theta_2 = \pm \pi$}\right\}$.

To do this, we use Friedlander's representation of the sine wave kernel
$\bmE(t)$ on the cone of angle $\alpha$, which expresses, in effect, this wave
kernel as the $\alpha$-periodized sine wave kernel on the cone of angle
$\infty$. Because of this, the wave kernels on two different cones
$C_{\alpha_1}$ and $C_{\alpha_2}$ are closely related. We use this fact,
together with our complete understanding of the case $\alpha = 4\pi$ from
Section~\ref{sec:micro-structure}, to prove the following theorem for any cone.
\begin{theorem}
  \label{thm:structure-sine-prop}
  The Schwartz kernel $\bmE$ of the sine propagator $\bfW(t)$ on the Euclidean cone $C_\alpha$ is an intersecting Lagrangian distribution of class
  \begin{equation*}
    (r_1 r_2 )^{-\frac{1}{2}} \cdot I^{-\frac{5}{4},-\frac{1}{2}} \left(\bbR \times C^\circ_\alpha \times  C^\circ_\alpha; \Lambda^\diff, \Lambda^\geom \right) .
  \end{equation*}
\end{theorem}

\subsection{Friedlander's construction of the wave propagator}
\label{sec:friedlanders-construction}

To start our study of the sine propagator near the singular set $\Sigma$, we
recall Friedlander's construction of the Schwartz kernel of $\bfW(t)$ from
\cite{Fri}.

Let $G(y,z)$ be the $L^1_\loc$-function on $\bbR^2_{(y,z)}$ given by
\begin{equation}
  \label{eq:Fried1}
  G(y,z) \defeq
  \begin{cases}
    H(y + \cos(z)) \, H(\pi - |z|), & y < 1 \\[1em]
    - \dfrac{1}{\pi} \left\{ \arctan\!\left[ \dfrac{\pi - z}{\arccosh(y)}
      \right] + \arctan\!\left[ \dfrac{\pi + z}{\arccosh(y)} \right] \right\}, &
    y > 1 .
  \end{cases}
\end{equation}
Form its periodization with respect to the map
\begin{equation*}
  \bbR^2 \ni (y,z) \longmapsto (y,z+\alpha) \in \bbR^2,
\end{equation*}
and denote the resulting function by $G_\alpha(y,z)$; concretely,
\begin{equation*}
  G_\alpha(y,z) = \sum_{k \in \bbZ} G(y,z + \alpha \cdot k) .
\end{equation*}
We may thus view $G_\alpha(y,z)$ as a function on
$\bbR \times \left( \bbR \big/ \alpha \bbZ \right)$.  Now, define the operator
$\bm{A} : \bbR \times \left( \bbR \big/ \alpha \bbZ \right) \To \bbR \times
C_\alpha^\circ \times C_\alpha^\circ$
as the composite $\bm{A} = A_3 A_2 A_1$, where
\begin{itemize}
\item $A_1 = \big[ \del_y \big]^\frac{1}{2}$ is half-derivation in the
  $y$-variable, that is, the composition of differentiation in $y$ with the
  fractional integral operator with kernel given by $|y-y'|^{-\frac{1}{2}}$;
\item $A_2 = F^*$ is pullback by the map
  \begin{equation*}
    F(t,r_1,\theta_1;r_2,\theta_2) = \left( y = \frac{t^2 - r_1^2 - r_2^2}{2 r_1
        r_2}, z = \theta_1 - \theta_2 \right) ;
  \end{equation*}
\item and $A_3$ is multiplication by the factor
  $\frac{1}{2\pi \sqrt{ \smash[b]{2 r_1 r_2} }}$.
\end{itemize}

\begin{proposition}[\cites{Fri,Hil}]
  \label{thm:Friedlander-construction}
  The operator $\bm{A}$ is a Fourier integral operator associated to the
  Lagrangian relation
  \begin{equation}
    \label{eq:Lagn-relation-Fried-A}
    \Lambda_F \defeq N^* \! \left\{ y = \frac{t^2 - r_1^2 - r_2^2}{2 r_1 r_2}
      \text{ and } z = \theta_1 - \theta_2 \right\} ,
  \end{equation}
  and the Friedlander distribution $\bm{A} G_\alpha$ on
  $\bbR_t \times C_\alpha^\circ \times C_\alpha^\circ$ is well-defined and equal
  to the Schwartz kernel of the sine propagator, $\bmE$.
\end{proposition}

The important feature of Friedlander's construction for us is the ease with
which it decomposes $\bmE$ into pieces which are either associated to the
geometric wave, the diffracted wave, or their intersection $\Sigma$.  We use
this to show the structure of $\bmE$ near $\Sigma$ is the same (up to a purely
diffractive term) for all cone angles $\alpha$.

\begin{proposition}
  \label{thm:ml-structure-sing-set-ind}
  Let $C_{\alpha_1}$ and $C_{\alpha_2}$ be two Euclidean cones.  There are
  isometric neighborhoods
  $V_1^\pm \subseteq \bbR_t \times C_{\alpha_1}^\circ \times C_{\alpha_1}^\circ$
  and
  $V_2^\pm \subseteq \bbR_t \times C_{\alpha_2}^\circ \times C_{\alpha_2}^\circ$
  of the set
  \begin{equation*}
    \left\{ \text{$t^2 = (r_1 + r_2)^2$ and $\theta_1 - \theta_2 = \pm \pi$}
    \right\} = \pr(\Sigma)
  \end{equation*}
  on which
  \begin{equation}
    \bmE_{\alpha_1} - \bmE_{\alpha_2} \in I^{-\frac{7}{4}} \! \left(V_j^\pm,
      \Lambda^\diff \right) ,
    \label{sine-wave-difference}  \end{equation}
  where $\bmE_\alpha$ is the sine propagator kernel on $\bbR_t \times
  C_\alpha^\circ \times C_\alpha^\circ$. The key point is that \eqref{sine-wave-difference} is purely diffractive.
\end{proposition}

\begin{proof}
  Let us start with the case where $\alpha_1$ and $\alpha_2$ are both greater
  than $2\pi$.  From Proposition~\ref{thm:Friedlander-construction} we know that
  $\bmE_\alpha = \bm{A} G_\alpha$, so we may prove this proposition by showing
  an analogous statement for the periodized function $G_\alpha$.  We note that
  the projection of $\Sigma$ to the base
  $\bbR_t \times C_\alpha^\circ \times C_\alpha^\circ$ corresponds to
  \begin{equation*}
    y = \frac{t^2 - r_1^2 - r_2^2}{2 r_1 r_2} = 1 \quad \text{and} \quad z =
    \theta_1 - \theta_2 = \pm \pi
  \end{equation*}
  in the original $(y,z)$-coordinates used to define $G$.

  Let $\alpha_* \defeq \min(\alpha_1,\alpha_2)$, and set
  $\varepsilon \defeq \frac{1}{8} \left( \alpha_* - 2\pi \right)$.  Choose a
  smooth bump function $\rho \in \calC^\infty_\upc(\bbR_z)$ satisfying
  $\rho(z) \equiv 1$ when $|z| < \frac{\varepsilon}{2}$ and $\rho \equiv 0$ when
  $|z| > \varepsilon$.  From $\rho$ and $G$ we define the following:
  \begin{equation*}
    \begin{alignedat}{3}
      \rho^\pm(z) &\defeq \rho(z \mp \pi) & \qquad G^\pm(y,z) &\defeq G(y,z) \,
      \rho^\pm(z) \\
      \rho^0(z) &\defeq \left(1 - \rho^+(z) - \rho^-(z) \right) \cdot
      \mathbf{1}_{\{|z| \leqslant \pi\}}(z) & G^0(y,z) &\defeq G(y,z) \,
      \rho^0(z) \\
      \rho^\infty(z) &\defeq \left(1 - \rho^+(z) - \rho^-(z) \right) \cdot
      \mathbf{1}_{\{|z| \geqslant \pi\}}(z) & G^\infty(y,z) &\defeq G(y,z) \,
      \rho^\infty(z) .
    \end{alignedat}
  \end{equation*}
  Thus, $G = G^0 + G^\infty + G^+ + G^-$. Using
  \eqref{eq:Lagn-relation-Fried-A}, and the calculus of wavefront sets, we see
  that these pieces of $G$ correspond to the geometric wavefront, the diffracted
  wavefront, and a small neighbourhood of $\Sigma$, respectively, after
  periodization and the application of $\bmA$.

  Now, consider the $\alpha_j$-periodizations
  $ G^\bullet_{\alpha_j}(y,z) \defeq \sum_{k \in \bbZ} G^\bullet(y,z + \alpha_j
  \cdot k)$
  of these distributions for $j = 1$ and $2$.  By choosing $\varepsilon$ as
  above, so that $\alpha_j > 2\pi + 2\varepsilon$, we have
  on the set
  $\big\{ z \in \left( - \pi - \varepsilon, \pi + \varepsilon \right) \big\}$
  \begin{equation*}
    G_{\alpha_1}(y,z) - G_{\alpha_2}(y,z) = G^\infty_{\alpha_1}(y,z) - G^\infty_{\alpha_2}(y,z)
  \end{equation*}
  since $G^\pm_{\alpha_1} = G^\pm_{\alpha_2}$ and
  $G^0_{\alpha_1} = G^0_{\alpha_2}$ here.  Therefore, if we view $\bmE_\alpha$
  as the restriction of $\bmA G_\alpha$ to the fundamental domain
  $\left[ - \frac{\alpha}{2}, \frac{\alpha}{2} \right)$ for the periodization,
  we may set
  $V_j^\pm = V^\pm \defeq F^{-1} ( \bbR \times (\pm \pi - \varepsilon, \pm \pi +
  \varepsilon))$ and conclude
  \begin{equation*}
    \bmE_{\alpha_1} - \bmE_{\alpha_2} \Big\vert_{V^\pm} = \bmA \left[ G^\infty_{\alpha_1} - G^\infty_{\alpha_2} \right]
    \Big\vert_{V^\pm} \in I^{-\frac{7}{4}} \! \left( V^\pm, \Lambda^\diff \right),
  \end{equation*}
  since $A G^\infty$ is purely diffractive. This establishes the result in the
  case $\alpha_1, \alpha_2 > 2\pi$.

  Finally, to extend to general cone angles $\alpha$ we use the method of
  images:  distributions on $C_\alpha$ may be represented as $\alpha$-periodic
  distributions on its $N$-fold cover $C_{N\alpha}$, where $N$ is any positive
  integer.  The result holds using $\bmE_{N \alpha}$ in place of $\bmE_\alpha$
  by the above, and we may recover the result for $\bmE_\alpha$ by restricting
  to a single period of length $\alpha$ in the angular variables.
\end{proof}

\begin{proof}[Proof of Theorem~\ref{thm:structure-sine-prop}]
  Theorem~\ref{thm:structure-sine-prop} follows immediately from
  Proposition~\ref{thm:4pi-propagator-structure} and
  Proposition~\ref{thm:ml-structure-sing-set-ind}.
\end{proof}

We conclude the microlocal structure of the half-wave kernel
$\bmU \defeq \calK \! \left[ e^{- it \sqrt\Delta} \right]$ as a corollary of
this result.

\begin{corollary}\label{cor:UisMU}
  The Schwartz kernel $\bmU$ of the half-wave group
  $\calU(t) \defeq e^{ -it\sqrt\Delta}$ on
  $\bbR \times C_\alpha^\circ \times C_\alpha^\circ$ is an intersecting
  Lagrangian distribution in the class
  \begin{equation*}
    r_1^{-\frac{1}{2}} r_2^{-\frac{1}{2}} I^{-\frac{1}{4}}\!\left(
      \bbR \times C_\alpha^\circ \times C_\alpha^\circ ;
      \Lambda^\diff_+, \Lambda^\geom_+ \right),
  \end{equation*}
  where $( \Lambda^\diff_\pm, \Lambda^\geom_\pm )$ is the forward/backward part
  of the intersecting pair $( \Lambda^\diff , \Lambda^\geom)$, i.e., the pair
  given by intersecting $( \Lambda^\diff, \Lambda^\geom )$ with
  $\left\{ (t,\tau) \times T^* C_\alpha^\circ \times T^* C_\alpha^\circ : \mp
    \tau > 0 \right\}$.
\end{corollary}

\begin{proof}
  We know from Theorem~\ref{thm:structure-sine-prop} that $\bfW$, the sine
  kernel, is in the class
  $I^{-5/4}(\RR \times C^\circ_\alpha \times C^\circ_\alpha; \Lambda^\diff,
  \Lambda^\geom)$.
  By taking a derivative in $t$, we find that $\cos t \sqrt{\Delta}$ is in the
  class
  $I^{-1/4}(\RR \times C_\alpha \times C_\alpha; \Lambda^\diff,
  \Lambda^\geom)$. We can write
$$
\cos t \sqrt{\Delta} = \frac1{2} \Big( e^{-it \sqrt{\Delta}} +
e^{it\sqrt{\Delta}} \Big).
$$
Since $e^{ \mp it \sqrt{\Delta}}$ is annihilated by the operator
$(D_t \pm \sqrt{\Delta})$, which has symbol $\tau \pm |\xi|$, we see that its
wavefront set is contained in $\{ \mp \tau > 0 \}$.  Therefore,
$e^{\mp i t \sqrt{\Delta}}$ is microlocally identical to $2\cos t \sqrt{\Delta}$
on $\Lambda^\diff_{\pm}$, and microlocally trivial on $\Lambda^\diff_{\mp}$.
\end{proof}

\begin{remark}
  In \cite{Hil}, a similar argument is used to pass from the sine kernel to the
  half-wave kernel but the factor of $2$ has been incorrectly omitted.
\end{remark}

\begin{example}\label{ex:U4piasMU}
  Starting from expression \eqref{eq:EasMU}, this procedure yields the following
  expression. On the cone of angle $4\pi$, for $\theta_1$ close to $\pi$ and
  $\theta_2$ close to $0$ we have:
  \[\displaystyle
  \bmU_{4\pi}(t,\bmq_1,\bmq_2)m \sim \frac{-i}{4\pi^2} \int_{
    0}^\infty\int_{0}^\infty e^{i\phi(t,\bmq_1,\bmq_2,s,\omega)}
  \frac{\sin\big(\frac{\theta_1(s)+\theta_2(s)}{2}\big)}{\big(r_1(s)r_2(s)\big)^{\frac{1}{2}}}\cdot
  \omega \, dsd\omega \big|d\bmq_1d\bmq_2\big|^{\frac{1}{2}},
  \]
 (where $\sim$ means equal modulo $C^\infty$) in which $\phi,\,r_j(s),\,\theta_j(s)$ are defined as in \eqref{eq:EasMU}.
\end{example}

\begin{remark}
  We emphasize that the novelty in Theorem~\ref{thm:structure-sine-prop} is the
  precise determination of the structure of the wave kernel near the singular
  set $\Sigma$, the intersection between $\Lambda^\geom$ and
  $\Lambda^\diff$. Indeed, the Lagrangian structure of the wave kernel near
  $\Lambda^\geom \setminus \Lambda^\diff$ (where the cone point plays no role,
  due to finite speed of propagation) follows from classical work of H\"ormander
  \cite{Hor0} (also together with Duistermaat \cite{DuiHor}). On the other hand,
  on metric cones (of any dimension), Cheeger and Taylor showed that the
  wavefront set of the wave kernel is contained in
  $\Lambda^\geom \cup \Lambda^\diff$ and showed the Lagrangian structure of the
  wave kernel near $\Lambda^\diff \setminus \Lambda^\geom$
  \cite{CheTay1}*{Section 2}, \cite{CheTay2}*{Section 5}.  More generally, on
  spaces with cone-like singularities, Melrose and Wunsch \cite{MelWun} proved
  that the wavefront set of the wave kernel is contained in
  $\Lambda^\geom \cup \Lambda^\diff$; morover, they also showed that the
  diffractive singularity is $(n-1)/2$-order more regular than the geometric
  singularity. Notice that this difference in order agrees with our results,
  since for an intersecting Lagrangian distribution, the order on $\Lambda_0$
  (here, the diffractive Lagrangian) is always smaller than the order on
  $\Lambda_1$ (here, the geometric Lagrangian) by $\frac{1}{2}$. This also shows
  that our result is restricted to dimension 2: in higher dimensions, it
  \emph{cannot} be true that the wave kernel on a cone is in the Melrose-Uhlmann
  calculus. In the latter case it would be interesting to know whether the
  kernel lies in the class of distributions that are constructed in
  \cite{GuiUhl} and that generalize the Melrose-Uhlmann construction.
\end{remark}

\subsection{Proof of Theorem \ref{thm:intro-kernel}}
\label{sec:proof}

Let us first consider the case of a diffractive geodesic of length $t^*$ joining
$\bmq_2^*$ to $\bmq_1^*$ with a diffraction angle of $\pi.$
\begin{remark}
  It may seem peculiar to use $\bmq_2$ as the starting point and $\bmq_1$ as the
  final point of the geodesic, but this is coherent with searching for an
  expression for $\bmU(t,\bmq_1,\bmq_2).$
\end{remark}

We can use a Euclidean system of coordinates such that
\begin{itemize}
\item $\bmq_2^*$ corresponds to $(-r_2^*,0),$
\item the geodesic corresponds to the horizontal line starting from
  $(-r_2^*,0).$
\end{itemize}
This Euclidean coordinate system can be uniquely extended to a local
isometry\footnote{If $\alpha>2\pi$ this isometry is actually one-to-one onto its
  range.} from $\bbR^2\setminus \{ (0,y),~y>0\}$ into $C_\alpha.$ We will freely
use this local isometry to identify points $(\bmq_1,\bmq_2)$ in a neighbourhood
of $(\bmq_1^*,\bmq_2^*)$ with their preimages in $\bbR^2.$

The point $\bmq_1^*$ corresponds to $(r_1^*,0)$ in this system of Euclidean
coordinates. The geodesic between $\bmq_2$ and $\bmq_1$ is horizontal and it can
be seen that it is geometrically diffractive with angle $+\pi$ since it is the
limit of horizontal geodesics approaching from below.  For $s\geq 0$, we denote
by $\bmp_{+}(s)$ the point with coordinates $(0,-s)$ in this Euclidean system
and we set
\[
\phi_+(t,\bmq_1,\bmq_2,s,\omega)\,\defeq\, \Big[
\big|\bmq_2-\bmp_+(s)\big|+\big|\bmq_1-\bmp_+(s)\big|-t\Big]\cdot \omega,
\]
where $\big|\bmq -\bmq'\big|$ denotes the Euclidean distance in $\bbR^2.$

When the angle of diffraction is $-\pi$ we can proceed similarly. The
diffractive geodesic is now the limit of horizontal geodesics from above and the
cut is now $\{ (0,y),~y< 0\}.$ We then define $\bmp_-(s)\defeq (0,s)$ and
\[
\phi_-(t,\bmq_1,\bmq_2)\defeq \Big[
\big|\bmq_2-\bmp_-(s)\big|+\big|\bmq_1-\bmp_-(s)\big|-t\Big]\cdot \omega
\]

\begin{lemma}
  In either situation, locally near
  $(t^*,\bmq_1^*,\bmq_2^*)\in \bbR\times C^\circ_\alpha\times C^\circ_\alpha,$
  $\phi_{\pm}$ is a phase function for the intersecting pair
  $(\Lambda_+^\geom,\Lambda_+^\diff).$
\end{lemma}


According to corollary \ref{cor:UisMU} and to section \ref{sec:int-Lagn-distns}
there exists a symbol $a_\alpha$ such that, locally near
$(t^*,\bmq^*_1,\bmq_2^*),$ we have the expression
\[
\bmU_\alpha(t,\bmq_1,\bmq_2) =(2\pi)^{-2} \int_{s\geq 0}\int_{\omega>0}
e^{i\phi_\pm(t,\bmq_1,\bmq_2,s,\omega)}
a_{\alpha,\pm}(t,\bmq_1,\bmq_2,s,\omega)\,dsd\omega \big|
d\bmq_1d\bmq_2\big|^{\frac{1}{2}}.
\]
Moreover, $a_{\alpha,\pm}$ has an asymptotic expansion of the form
\begin{equation}\label{eq:expaalpha}
  a_\alpha\sim \sum_{k\geq 0} a_{\alpha,\pm,1-k}(\bmq_1,\bmq_2,s) \, \omega^{1-k}.
\end{equation}

The only thing left to prove is the relation with the geometric theory of
diffraction. This is done by computing the leading amplitude of $\bmU_\alpha$ near
the diffracted front and away from $\Sigma$, and comparing it with Proposition
\ref{prop:hwkgtd} in the Appendix.

Starting from the preceding expression and using the methods and results of
section \ref{sec:int-Lagn-distns}, the leading term on the diffracted front is
given by
\[
\bmU_\alpha(t,\bmq_1,\bmq_2) \sim -(2\pi)^{-2} \int_{\omega>0}
e^{i\phi_\pm(t,\bmq_1,\bmq_2,0,\omega)}
\frac{a_{\alpha,\pm,1}(\bmq_1,\bmq_2,0)\omega} {i\partial_s
  \phi_{\pm}(t,\bmq_1,\bmq_2,0,\omega)} d\omega \big|
d\bmq_1d\bmq_2\big|^{\frac{1}{2}}.
\]

We compute
$\partial_s \phi_{\pm}(t,\bmq_1,\bmq_2,s=0,\omega) =\pm \big( \sin
\theta_1+\sin \theta_2 \big)\omega,$
and compare with equation \ref{eq:hwkgtd}. We obtain

\[
- \frac{a_{\alpha,\pm,1}(\bmq_1,\bmq_2,s=0)}{\pm i \big( \sin \theta_1+\sin
  \theta_2\big)}= 2\pi(r_1r_2)^{-\frac{1}{2}} \,
S_\alpha(\theta_1-\theta_2)
\]
so that finally
\begin{equation}
\label{eq:as=0}
  a_{\alpha,\pm, 1}(t,\bmq_1,\bmq_2,s=0,\omega) \sim \mp 2\pi  i
  \cdot \frac{S_\alpha(\theta_1-\theta_2)}{\big(
    r_1r_2)^{\frac{1}{2}}}\cdot \big[ \sin \theta_1+\sin \theta_2\big ]\cdot \omega.
\end{equation}
This is the last statement in Theorem \ref{thm:intro-kernel}.

\begin{remark}
  This formula actually gives a way of computing $S_\alpha$ if we know the
  symbol in the Melrose-Uhlmann representation. For instance, starting from the
  formula in example \ref{ex:U4piasMU} for the propagator near a diffractive
  geodesic with an angle $-\pi$ on a cone of angle $4\pi$ we derive
  \[
  a_{4\pi,-,1}(\bmq_1,\bmq_2,s)=-i \left(
  r_1(s)r_2(s)\right)^{-\frac{1}{2}}\sin \! \left(
  \frac{\theta_1(s)+\theta_2(s)}{2}\! \right)
  \]
  The preceding formula thus yields

  \begin{equation*}
    \begin{aligned}
      S_{4\pi}(\theta_1-\theta_2)& = \frac{-2\pi}{-i\left(\sin \theta_1 +\sin
        \theta_2\right)} \cdot \frac{-i}{4\pi^2} \cdot \sin \left(
      \frac{\theta_1(0)+\theta_2(0)}{2}\right) \\
      & =\frac{-1}{2\pi}\frac{\sin \left( \frac1{2} (\theta_1+\theta_2) \right)}{\left(\sin \theta_1 +\sin
        \theta_2\right)}\\
      & =\frac{-1}{4\pi} \Big(\cos
        \big(\frac{\theta_1-\theta_2}{2}\big) \Big)^{-1}.
    \end{aligned}
  \end{equation*}
  This agrees with the formula \eqref{eq:S4pi} in Appendix \ref{sec:GTD}.
\end{remark}

\begin{remark}
  It is interesting to note that $a_{\alpha,\pm,1}(\bmq_1,\bmq_2,s=0,\omega)$ is
  actually a regularization of the symbol on the diffracted front. The latter
  blows up when approaching the intersection and this formula gives
  an effective way of regularizing the contribution of a diffractive
  geodesic when the diffraction angle approaches $\pm\pi$ (compare with the
  approach of \cite{BogoPavSch}).
\end{remark}


\section{The wave kernel after two geometric diffractions}
\label{sec:wave-kernel-two-diffractions}

Theorem \ref{thm:intro-kernel} can be used to understand the half-wave
propagator on a ESCS after microlocalization along a particular diffractive
geodesic. We will now present two applications of this method. A systematic
study leading to a better knowledge of wave-invariants of a ESCS will be done
elsewhere.

Now that we have the basic structure of the half-wave kernel on the cone
$C_\alpha$, we next determine the structure of the kernel after two diffractions
on a Euclidean surface with conic singularities (ESCS).  While one could
continue to calculate the structure for an arbitrary number of diffractions and
any kind of diffraction, we will focus on two geometric diffractions since this
is the first case for which our approach yields a significant improvment on the
existing literature.

Let $X$ be an ESCS as described in the Section \ref{sec:introduction}, and let
$\bmq_1^*$ and $\bmq_2^*$ be two points in $X$ with a geodesic $\gamma$ of
length $t^* > 0$ between them. Denote also by $\xi_i^*$ the covector in
$T_{\bmq_i^*}^*X$ of the bicharacteristics that projects onto $\gamma.$

Our aim is to find an oscillatory integral representation of the Schwartz kernel
of the operator
\begin{equation*}
  A_1 \, \calU(t) \, A_2 ,
\end{equation*}
where $A_i \in \Psi^0(\Sigma^\circ)$ is microlocalizing near
$(\bmq_i^*, \xi_i^*)$ and $\calU(t) \defeq e^{-it\sqrt{\Delta}}$ is the
half-wave kernel at time $t$ with $t$ close to $t^*$.

In order to fix notations we assume the following.  The geodesic starts at
$\bmq_2^*$ then hits a cone point $\bmp_2$ then a cone point $\bmp_1$ and
finally ends at $\bmq_1^*.$ We denote by $a$ the distance (along this geodesic)
from $\bmq_2^*$ to $\bmp_2^*,$ by $b$ the distance between $\bmp_2$ and $\bmp_1$
and by $c=t^*-(a+b)$ the distance from $\bmp_1$ to $\bmq_1^*.$ Moreover, we
suppose that this geodesic passes \emph{geometrically} through both two cone
points $\bmp_2$ and $\bmp_1$; i.e., $\gamma$ is locally a limit of
non-diffractive geodesics.

Every geodesic with only one diffraction, which is geometric, is a limit of nondiffractive
geodesics.
For a general diffractive geodesic with several geometric diffractions, it may
happen that, locally,  the geodesic is a limit of non-diffractive
geodesics, but not globally. However, in our case, since $\gamma$ has only two
geometric diffractions, it is always such a limit of non-diffractive geodesics
(see \cite{Hil2}).  We show this by generalizing the
construction we did for the geometric diffractive geodesic on a
cone\footnote{This construction is the same as the rectangles with slits that
  are used in \cite{Hil2}}.

We start with a Euclidean coordinate system at $\bmq_2^*$ such that $\bmq_2^*$
corresponds to $(-a,0)$ and the geodesic is horizontal and we try to extend this
coordinate system. In the extended system the geodesic will correspond to the
horizontal segment that joins $(-a,0)$ to $(b+c,0)$ so that $\bmp_2$ will
correspond to $(0,0)$ and $\bmp_1$ to $(b,0).$ We remove from $\bbR^2$ the cuts
\begin{equation*}
  \begin{aligned}
    \mbox{cut}_2&\defeq \big \{ (0,\epsilon_2 s),~s>0\big \},\\
    \mbox{cut}_1& \defeq \big \{ (b,\epsilon_1 s),~s>0\big \},
  \end{aligned}
\end{equation*}
in which $\epsilon_i,\,i=1,2$ is such that the angle of diffraction at $\bmp_i$
is $\epsilon_i\pi.$ Exploiting the flatness of $X$, the original coordinate
system can be extended to a local isometry from an open set
$V\subset \bbR^2\setminus \big( \mbox{cut}_1\cup \mbox{cut}_2\big)$ that
contains the horizontal segment. If both $\epsilon_i$ have the same sign then
$\gamma$ is the limit of non-diffractive geodesics that pass above the two cone
points (or below the two cone points). If the $\epsilon_i$ have opposite signs, $\gamma$ is a limit
of non-diffractive geodesics that cross the horizontal line between the two cone
points. This case is illustrated in Figure~\ref{fig:geom-diffraction}.

\begin{figure}
  \centering







 \includegraphics{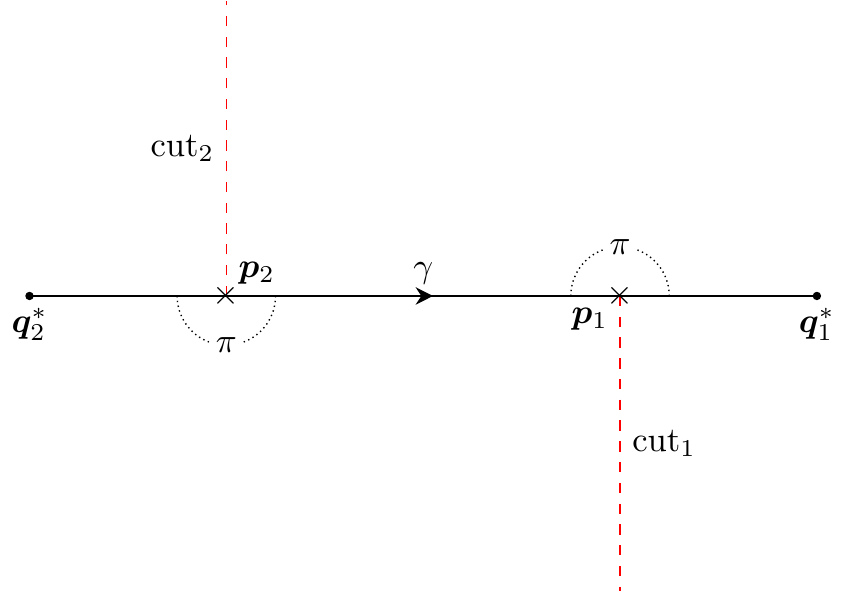}

 \caption{The geodesic $\gamma$ passing through two cone points, with a geometric diffraction in both cases}
 \label{fig:geom-diffraction}
\end{figure}

Using this local isometry we can define, for $(\bmq_1,\bmq_2)$ near
$(\bmq_1^*,\bmq_2^*),$ the functions $|\bmq_i-\bmp_i|$ and $|\bmq_1-\bmq_2|$ to
be the Euclidean distance in $\bbR^2$ of the corresponding preimages.

For $t$ close to $t^*$ we can then define the following Lagrangian submanifold
in $T^*\big (X^\circ\times X^\circ\big)$
\begin{equation*}
  \begin{aligned}
    \Lambda^0 & \,\defeq\, N^*\big\{
    |\bmq_2-\bmp_2|+b+|\bmp_1-\bmq_1|=t\,\big\},\\
    \Lambda^1& \,\defeq\, N^*\big\{ |\bmq_2-\bmp_1|\,+|\bmp_1-\bmq_1|=t\,\big\},\\
    \Lambda^2& \,\defeq\, N^*\big\{ |\bmq_2-\bmp_2|\,+|\bmp_2-\bmq_1|=t\,\big\},\\
    \Lambda^3& \,\defeq\, N^*\big\{ |\bmq_2-\bmq_1|=t\,\big\},\\
  \end{aligned}
\end{equation*}
It is straightforward to check that $\Lambda^3$ corresponds to direct
propagation, $\Lambda^1$ corresponds to one diffraction at $\bmp_1$, $\Lambda^2$
to one diffraction at $\bmp_2$ and $\Lambda^0$ to two diffractions in a row at
$\bmp_2$ and $\bmp_1.$

The aim of this section is the following
\begin{proposition}\label{prop:2dif}
  For $t > 0$ fixed near $t^*$, the Schwartz kernel of $A_1 \, \calU(t) \, A_2$
  is an intersecting Lagrangian distribution of order $0$ associated to the four
  Lagrangian submanifolds $(\Lambda_0, \Lambda_1, \Lambda_2, \Lambda_3)$.
\end{proposition}

\begin{proof}
  We begin with a decomposition of $A_1 \, \calU(t) \, A_2$ in which only one
  conic point plays a role in each factor.  This is straightforward: we choose a
  time $t_0\in (a, a+b)$, say $t_0 = a + \frac{b}{2}$, and we write
  $A_1 \, \calU(t) \, A_2 = \left( A_1 \, \calU(t-t_0) \right) \left( \calU(t_0)
    \, A_2\right)$.  In terms of their Schwartz kernels, this is
  \begin{equation}
    \calK\!\left[ A_1\, \calU(t) \, A_2 \right] \! (\bmq_1, \bmq_2) = \int_X
    \calK\!\left[ A_1 \, \calU(t - t_0) \right](\bmq_1,\bmq) \cdot \calK\!\left[
      \calU(t_0) \, A_2 \right] \! ( \bmq, \bmq_2) \, d\bmq.
    \label{A1UA2}
  \end{equation}
  Due to the assumptions on the microlocalizers $A_i$, the only points $\bmq$
  that contribute to the singularities of \eqref{A1UA2} are points near
  $(b/2, 0)$ in our coordinate system. In each factor of the composition above,
  the singularities of the half-wave kernel only meet one cone point. Thus,
  modulo smooth errors, we may replace the half-wave kernel by the half-wave
  kernel on an exact cone in each factor, allowing us to use the results of
  Section~\ref{sec:micro-structure}.

  More precisely, in \eqref{A1UA2}, to obtain a singularity
  $(\bmq_1, \bm\xi_1; \bmq_2, \bm\xi_2)$ in the canonical relation of
  $A_1 \, \calU(t) \, A_2$, we must have $(\bmq, \bm\xi; \bmq_2, \bm\xi_2)$ in
  the canonical relation of $\calU(t_0) \, A_2$ and
  $(\bmq_1, \bm\xi_1; \bmq, \bm\xi)$ in the canonical relation of
  $A_1 \, \calU(t-t_0)$. For $t$ sufficiently close to $t^*$, this implies that
  $\bmq$ is close to the point $(b/2, 0)$. That is, up to a $\calC^\infty$
  error, we may insert a cutoff function $\chi^2(\bmq)$ into \eqref{A1UA2},
  where $\chi$ is supported close to $(0, b/2)$:
  \begin{equation}
    \calK\!\left[ A_1 \, \calU(t) \, A_2 \right] \! (\bmq_1, \bmq_2) \equiv
    \int_X  \calK\!\left[ A_1 \, \calU(t - t_0) \right](\bmq_1,\bmq) \cdot
    \chi^2(\bmq) \cdot \calK\!\left[ \calU(t_0) \, A_1 \right] \! ( \bmq, \bmq_2)
    \, d\bmq
    \label{A1UA2-2}
  \end{equation}
  modulo $\calC^\infty$ errors.  Moreover, restricting the microlocal supports
  of $A_1$ and $A_2$ if needed, we may assume that the support of $\chi$ is
  contained in a ball that is isometric to the corresponding ball in $\bbR^2.$
  By the above assumptions on the geodesic $\gamma$, the half-wave operators
  $\calU(t_0)$ and $\calU(t- t_0)$ in the compositions
  $\chi(\bmq) \, \calU( t_0) \, A_2$ and $A_1 \, \calU(t-t_0) \, \chi(\bmq)$ can
  be replaced (up to a smooth error) by the corresponding wave kernels on the
  exact cones with cone points $\bmp_1$, resp.~ $\bmp_2$, which we know from
  Section~\ref{sec:micro-structure} are intersecting Lagrangian distributions
  associated to the diffractive and main fronts. That is, we can express the
  Schwartz kernel of $\chi(\bmq) \, \calU( t_0) \, A_2$ in the oscillatory
  integral form
  \begin{equation}
   (2\pi)^{-2}  \int_{0}^\infty \int_{0}^\infty   e^{i\phi_2(\bmq, \bmq_2, t_0, s_2, \omega_2)}
    a_2(\bmq,\bmq_2, t_0, s_2, \omega_2) \,  d\omega_2 \, ds_2
    \label{wave-kernel-Phi2}
  \end{equation}
  where $\phi_2$ is the phase function
  \begin{equation}
    \label{Phi2}
    \phi_2 \defeq \big[ |\bmq-\bmp_2(s_2)| + |\bmp_2(s_2)-\bmq_2| -
    t_0 \big ] \cdot \omega_2 ,
  \end{equation}
  where $\bmp_2(s_2)$ has coordinates $(0,-\epsilon_2 s_2)$ and $|\cdot|$
  denotes the Euclidean distance in $\bbR^2.$

  Similarly, the Schwartz kernel of $A_1 \, \calU(t-t_0) \, \chi(\bmq)$ has the
  oscillatory integral representation
  \begin{equation}
    (2\pi)^{-2} \int_{0}^\infty  \int_{0}^\infty e^{i\phi_1(\bmq_1, \bmq, t-t_0, s_1, \omega_1)}
    a_1(\bmq_1,\bmq, t-t_0, s_1, \omega_1)  \, d\omega_1 \, ds_1
    \label{wave-kernel-Phi1}
  \end{equation}
  where $\phi_1$ is the phase function
  \begin{equation}
    \label{Phi1}
    \phi_1 \defeq \big[ |\bmq_1-\bmp_1(s_1)|+|\bmp_1(s_1)-\bmq| -
    (t-t_0) \big] \cdot \omega_1,
  \end{equation}
  where now $\bmp_1(s_1)$ has coordinates $(b,-\epsilon_1 s_1).$ Here, $a_i$ is
  smooth, supported in $\omega_i \geqslant 1$, and is a symbol of order $1$ in
  $\omega_i$.

  Therefore, \eqref{A1UA2-2} is given by an oscillatory integral (up to smooth
  errors) of the form
  \begin{multline}
  (2\pi)^{-4}   \int_{X_{\bmq}} \int_{\bbR^2_{\bm\omega}} \int_{s_1 = 0}^\infty \int_{s_2 =
      0}^\infty e^{i \phi_1 + i \phi_2} \,  a_1(\bmq_1,\bmq, t-t_0, s_1, \omega_1) \\
    \times a_2(\bmq,\bmq_2, t_0, s_2, \omega_2) \, ds_1 \, ds_2 \, d\omega_1 \,
    d\omega_2 \, d\bmq
    \label{wave-kernel-comp}
  \end{multline}
  We now show that in the overall phase function $\Phi \defeq \phi_1 + \phi_2$
  we can eliminate the variables $(\bmq, \omega_2)$. This is possible if the
  following non-degeneracy condition is satisfied:
  \begin{equation}
    d_{(\bmq, \omega_2)} \Phi = 0 \implies \det  d^2_{(\bmq, \omega_2), (\bmq,
      \omega_2)} \Phi \neq 0 .
    \label{ztheta-cond}
  \end{equation}
  The condition $ d_{\bmq} \Phi = 0$ implies that $\bmq$ is on
  the segment $[p_2(s_2),p_1(s_1)]$ and that $\omega_1 = \omega_2$. The condition $d_{\omega_2} \Phi = 0$ implies that $\bmq$ is at distance $t_0-|\bmp_2(s_2)-\bmq_2|$
  from $\bmp_2(s_2).$ Since the non-degeneracy condition is coordinate free and
  has to be verified with fixed $s_1,s_2,\omega_1,\bmq_1,\bmq_2$ we can choose
  for $\bmq$ cartesian coordinates $(x,y)$ in a rotated and translated coordinate frame,
   such that the origin corresponds to
  the critical point and the conical points $\bmp_i(s_i)$ have the following
  coordinates:  $\bmp_1(s_1)=(B,0),~\bmp_2(s_2)=(-A,0)$ with positive $A,~B.$ We
  observe that $A$ and $B$ depend on all remaining variables.

  In these coordinates we have (we only keep $(x,y,\omega_2)$ as variables since
  the other ones are fixed)
  \begin{multline*}
      \Phi(x,y,\omega_2) = \omega_1\cdot \left[
      |\bmq_1-\bmp_1(s_1)|+\sqrt{(B-x)^2+y^2}-(t-t_0)\right] \\
        \mbox{} +\omega_2\cdot \left[
      \sqrt{(A+x)^2+y^2}+|\bmq_2-\bmp_2(s_2)|-t_0\right]
  \end{multline*}

  We compute that
  \begin{equation}
    \begin{aligned}
      d_{x} \Phi &=  \frac{(x-B)\omega_1}{\sqrt{(B-x)^2+y^2}} + \frac{(A+x) \omega_2}{\sqrt{(A+x)^2+y^2}}  , \\
      d_{y} \Phi &= \frac{y \omega_1}{\sqrt{(B-x)^2+y^2}} +
      \frac{y \omega_2}{\sqrt{(A+x)^2+y^2}} , \\
      d_{\omega_2} \Phi &= \sqrt{(A+x)^2+y^2}+|\bmq_2-\bmp_2(s_2)|-t_0
    \end{aligned}
  \end{equation}
  The critical point is easily seen to be $(x=0,y=0,\omega_2=\omega_1).$ We can
  then compute the Hessian of $\Phi$ in the $(x,y, \omega_2)$-variables and
  evaluate it at the critical point:
  \begin{equation}
    \begin{bmatrix}
      \partial_{xx}\Phi & 0  & 1\\
      0 & \omega_1C & 0 \\
      1 & 0 & 0
    \end{bmatrix}, \quad C \defeq \frac1{A} + \frac1{B}.
    \label{Hessian}\end{equation}

  The determinant is $-C \omega_1 < 0$ so that the non-degeneracy condition is
  satisfied. It is straightforward to check that this matrix has two positive eigenvalues and
  one negative eigenvalue. The signature thus is $1.$

  Hence, using the argument of H\"ormander \cite{HorFIO}, we can write the
  oscillatory integral where we replace $(x,y, \omega_2)$ by their values at the
  stationary point that we denote by $\bmq_c$. We obtain the oscillatory
  integral (writing $\omega$ for $\omega_1$)
  \begin{equation}
   (2\pi)^{-5/2}  \int_{-\infty}^\infty \int_{s_1 = 0}^\infty \int_{s_2
      = 0}^\infty e^{i\Psi(t,\bmq_1,\bmq_2, s_1, s_2,
      \omega)} \tilde a(t,\bmq_1, \bmq_2, s_1, s_2, \omega)   \, ds_1 \,
    ds_2 \,  d\omega,
    \label{wk-result}
  \end{equation}
  where the phase function $\Psi(t, \bmq_1, \bmq_2, s_1, s_2, \omega_1)$ is seen to be
  \begin{equation}\label{Psi}
    \begin{aligned}
      \Psi & \defeq \big[
      |\bmq_2-\bmp_2(s_2)|+|\bmp_2(s_2)-\bmq_c|+|\bmq_c-\bmp_1(s_1)|+|\bmp_1(s_1)-\bmq_1|-t
      \big]
      \omega_1  \\
      & = \big[
      |\bmq_2-\bmp_2(s_2)|+|\bmp_2(s_2)-\bmp_1(s_1)|+|\bmp_1(s_1)-\bmq_1|-t
      \big] \omega_1,
    \end{aligned}
  \end{equation}
  and the amplitude is given by
  \begin{equation}
  \tilde a(t,\bmq_1, \bmq_2, s_1, s_2, \omega) = e^{i\pi/4} (\omega C)^{-1/2} a_1(\bmq_1, \bmq_c, t-t_0, s_1, \omega) a_2(\bmq_c, \bmq_2, t_0, s_2, \omega).
 \label{tildea} \end{equation}

  We can now verify easily, using Definition~\ref{def:Leg-sys-param} and
  Remark~\ref{rem:4ILS}, that $\Psi$ pa\-ra\-met\-rizes the given system of four
  Lagrangian submanifolds. Indeed, a simple computation shows that at
  $(t^*,\bmq_1^*,\bmq_2^*,\omega^*) = (c,(c-a, 0), (-a, 0), 1)$ we have
  $d_{\omega, s_1, s_2} \Psi(t,\bmq_1^*, \bmq_2^*, \omega^*, 0, 0)=0$.
  Moreover, explicit computation shows that at this point the differential
  $d( \partial\Psi/\partial s_1)$ is a nonzero multiple of $dy_1$, the
  differential $d( \partial\Psi/\partial s_2)$ is a nonzero multiple of $dy_2$,
  and $d( \partial\Psi/\partial \omega_1)$ has a nonzero $dt$ component.  Thus
  these differentials are linearly independent, implying that the localized
  propagator is an intersecting Lagrangian distribution associated to the above
  system. It is not hard to check that the four Lagrangians correspond to no
  diffractions ($\Lambda_3$), one diffraction ($\Lambda_1$, $\Lambda_2$),
  arising from interaction with $\bmp_1$ or $\bmp_2$ respectively, and two
  diffractions ($\Lambda_0$).  Finally, as $\tilde a$ in \eqref{tildea} is a symbol in $\omega$ of
  order $3/2$, we see directly from
  Proposition~\ref{prop:int-sys-phasefunction} that the order of the
  distribution (that is, the order on $\Lambda_3$, the direct front) is $0$ (where $t$ is treated as a parameter).
\end{proof}

To conclude this section, we compute the principal symbol of the wave kernel
$\calU(t) \defeq e^{-it \sqrt{\Delta}}$ at the twice-diffracted Lagrangian
$\Lambda_0$ using Proposition~\ref{prop:symbol-4}. This amounts to computing $\tilde a$, to leading order in $\omega$, at $s_1 = s_2 = 0$. We will do the computation in the case $\epsilon_1=-1$ and $\epsilon_2=+1,$ as in Figure~\ref{fig:geom-diffraction}. The other cases are similar.

Clearly, from \eqref{tildea}, we need the leading order behaviour of $a_i$ at $s_i = 0$.
This is given by \eqref{eq:as=0}. Substituting into \eqref{tildea}, we find that  when $s_1 = s_2 = 0$,
\begin{multline}
 \tilde a(t,\bmq_1, \bmq_2, 0, 0, \omega) =  e^{i\pi/4} C^{-1/2}
 (2\pi)^2 S_{\alpha_1} (-\pi - \theta_1)  \sin \theta_1 \big( |\bmq_1 - \bmp_1| \, | \bmq_c - \bmp_1|\big)^{-1/2} \\ \times S_{\alpha_2}(\theta_2) \sin \theta_2 \big( |\bmq_2 - \bmp_2| \, | \bmq_c - \bmp_2| \big)^{-1/2}
 \omega^{3/2} \text{ mod } S^{1/2}.
\label{atilde3} \end{multline}
 When $s_1 = s_2 = 0$ we have
 $$
 C = \frac{A + B}{AB} = \frac{ |\bmq_c - \bmp_1| +  |\bmq_c - \bmp_2|}{ |\bmq_c - \bmp_1| \, |\bmq_c - \bmp_2|} = \frac{ b}{ |\bmq_c - \bmp_1| \, |\bmq_c - \bmp_2|} .
 $$
 Using coordinates where $(r_i,\theta_i)$ are polar coordinates for $\bmq_i$ centered at $\bmp_i$, we can simplify \eqref{atilde3} to
\begin{multline}
 \tilde a(t,\bmq_1, \bmq_2, 0, 0, \omega) =  e^{i\pi/4}
 (2\pi)^2 S_{\alpha_1} (-\pi - \theta_1)   S_{\alpha_2}(\theta_2) \sin \theta_1 \sin \theta_2  \\
 \times \Big(\frac{r_1r_2}{b} \Big)^{-1/2}
 \omega^{3/2} \text{ mod } S^{1/2}.
 \label{atilde4}\end{multline}

 Using Proposition $\ref{prop:symbol-4}$, and the identities
  \begin{equation*}
    \Psi_{s_1} = \frac{\omega \, y_2}{|\bmy|} = \omega \sin(\theta_1) \quad
    \text{and} \quad \Psi_{s_2} = \frac{\omega x_2}{|\bmx|} = \omega
    \sin(\theta_2)
  \end{equation*}
valid  when $s_1 = s_2 = 0$, we find that the principal symbol at the twice-diffracted Lagrangian $\Lambda_0$ is
  \begin{multline}
    \left. \frac{1}{2\pi} \left[ \frac{\tilde a(\bmx, \bmy, t_0, t, 0, 0,
          \omega)}{\Psi_{s_1} \Psi_{s_2}} \right] \right|_{C_0} \left|
      \frac{\partial (r_1, \theta_1, \theta_2, \omega, r_1 + b + r_2 -
        t)}{\partial (x, y, \omega)} \right|^{-\frac{1}{2}} |dr_1
    d\theta_1 d\theta_2 d\omega|^{\frac{1}{2}}  \\
    =  \frac{2\pi \, e^{i\pi/4}}{\omega^{\frac{1}{2}} b^{\frac{1}{2}}}
    S_{\alpha_2}(\theta_2) \, S_{\alpha_1}(-\pi -  \theta_1) \,  |dr_1
    d\theta_1 d\theta_2 d\omega|^{\frac{1}{2}} . \qedhere
\label{2Diff-prsymb}  \end{multline}


\section{Contributions to the wave trace of an isolated orbit with two geometric
  diffractions}
\label{sec:wave-invariants}

As a byproduct of our approach we can compute in a rather straightforward way
the leading contribution to the wave trace of any kind of periodic orbit, thus
generalizing \cite{Hil}. We present here the case of an isolated periodic
geodesic that has two geometric diffractions (and no other diffractions).

More precisely, we assume that the orbit $\gamma$ diffracts at $\bmp_1$ and
$\bmp_2$ (not necessarily distinct) and that the angles of diffraction are
$-\pi$ and $+\pi.$ We construct as before the rectangle with cuts that is
associated with this periodic geodesic. We see that near $\bmp_1$ the geodesic
is locally the limit of non-diffractives geodesics that pass above $\bmp_1.$
Near $\bmp_2$ it is locally the limit of non-diffractive geodesics that pass
below.  It follows that one cannot translate the orbit to a nearby periodic
orbit, so that the orbit is isolated as a periodic orbit.

\begin{remark}
  If instead of translating the orbit we rotate it then we do obtain
  non-diffractive geodesics that converge to $\gamma$ on any interval $[0,T]$
  but these won't be periodic.
\end{remark}

Let $\bmq$ be a point on $\gamma,$ we intend to compute
\[
\sigma_\rho(t) \defeq \Tr(A_1\calU(t)A_2\rho),
\]
for $t$ close to the period $L$, $A_i$ is a microlocal projector near
$(\bmq,\xi^*)$ and $\rho$ a bump function near $\bmq$ such that on the support
of $\rho$ the principal symbols of $A_1$ and $A_2$ are identically $1$ on the
lift of the geodesic. More precisely, we first choose $A_1$ and $A_2$ such that
for $t$ close to $L,$ any geodesic of length $t$ whose starting point is in the
microsupport of $A_2$ and whose endpoint is in the microsupport of $A_1$ stays
close to $\gamma.$ The bump function is chosen afterwards.

We construct the Euclidean system of coordinates as before: the periodic orbit
lies along the $x$-axis, with cone points $\bmp_2$ located at $(0,0)$ and
$\bmp_1$ at $(b, 0)$, and we identify $(x,y)$ with $(x + L, y)$, where $L$ is
the period.

According to Section~\ref{sec:wave-kernel-two-diffractions}, the Schwartz kernel
of the half-wave operator $A_1\calU(t)A_2$ after two diffractions has the
following oscillatory integral representation
\begin{equation}
 (2\pi)^{-5/2}  \int_0^\infty \int_{0}^\infty
  \int_{0}^\infty e^{i\psi(\bmx, \bmy, t, s_1, s_2,
    \omega)} \, \tilde a(t,\bmq_1,\bmq_2, s_1, s_2, \omega) \, ds_1 \, ds_2 \,
  d\omega,
  \label{Psioscint}
\end{equation}
where $\psi$ is given by
\[
\psi(t,\bmq_1,\bmq_2,s_1,s_2,\omega)= \big[ |\bmq_2-\bmp_2(s_2)|+
|\bmp_2(s_2)-\bmp_1(s_1)|+ |\bmp_1(s_1)-\bmq_1|-t\big]\cdot \omega
\]
and $\tilde a$ is given by \eqref{atilde4}:
\[
\tilde{a}(t,\bmq_1,\bmq_2,0,0,\omega)\sim (2\pi)^2 \cdot  e^{i\frac{\pi}{4}} \cdot
\frac{\sin \theta_1 \cdot S_{\alpha_1}(-\pi-\theta_1)\cdot \sin \theta_2\cdot
  S_{\alpha_2}(\theta_2)}{(r_1 b r_2)^{\frac{1}{2}}}\cdot \omega^{\frac{3}{2}}.
\]

We are thus lead to compute
\begin{equation}
  \sigma_{\rho}(t)\defeq
 (2\pi)^{-5/2} \int_{X}\int_0^\infty \int_0^\infty \int_0^\infty
  e^{i\psi(t,\bmq + (L,0),\bmq,s_1,s_2,\omega)}\tilde{a} \, \rho(\bmq) \,
  ds_1ds_2d\bmq d\omega,
\end{equation}
where we have set $\bmq_2 = \bmq$ and $\bmq_1 = \bmq + (L,0)$.

We choose to parametrize $\bmq$ by $(x,y)$: the Euclidean coordinates near
$\bmq_2^*.$ In this oscillatory integral, we first perform a stationary phase
with respect to $y.$ We denote by $y_c$ the stationary (critical) point. We
observe geometrically that $(x,y_c)$ is on the segment
$[\bmp_1(s_1),\bmp_2(s_2)].$ Moreover, we compute
\[
|\partial^2_y \psi(t, (x,y_c),0,0,\omega)|=\frac{|L-b|}{|x||L-b-x|}\cdot
\omega.
\]
It follows that the critical point remains non-degenerate for small $(s_1,s_2).$
Since geometrically, it is obvious that the critical point is a minimum, it also
follows that the signature is $+1.$

We now observe that the phase, when evaluated at the critical point becomes
independent of the remaining $x.$ More precisely it is given by $\tilde{\psi}$
where we have set
\begin{equation}
\tilde{\psi}(t,s_1,s_2,\omega)=\Big[
\sqrt{b^2+(s_1+s_2)^2}+\sqrt{(L-b)^2+(s_1+s_2)^2}-t\Big]\cdot \omega.
\label{psitilde}\end{equation}
We thus obtain after applying the stationary phase:
\[
\sigma_{\rho}(t) = \int_0^L\int_0^\infty \int_0^\infty \int_0^\infty e^{i\tilde{\psi}}
A(t,x,s_1,s_2,\omega)\rho((x,y_c))\,ds_1 \, ds_2 \, d\omega \, dx,
\]
where $A$ is a symbol that, at leading order and for $s_1=s_2=0$, reads
\begin{equation*}
  \begin{aligned}
    A(t,x,0,0,\omega) &\sim (2\pi)^{1/2 - 5/2} e^{i\frac{\pi}{4}}\tilde{a}(t,(x+L,y_c),(x,y_c),0,0,
    \omega)\\
    &\hspace*{1in} \mbox{} \times
    |\partial^2_y\psi(t,(x,y_c),0,0,\omega)|^{-\frac{1}{2}} \rho(x,y_c) \\
    &\sim i \cdot \frac{\sin(\theta_1) \cdot S_{\alpha_1}(-\pi-\theta_1)
    \cdot \sin(\theta_2) \cdot S_{\alpha_2}(\theta_2)}{ \sqrt{b \cdot |x||L-b-x|}
    } \cdot \frac{\sqrt{|x||L-b-x|}}{\sqrt{L-b}} \cdot \omega \\
    & \sim i\cdot \frac{\sin \theta_1\cdot S_{\alpha_1}(-\pi-\theta_1)\cdot
      \sin \theta_2 \cdot S_{\alpha_2}(\theta_2)}{\sqrt{b(L-b)}}\cdot \omega.
  \end{aligned}
\end{equation*}
It remains to evaluate an oscillatory integral of the form
\[
I(t)\defeq \int_0^L\int_0^\infty \int_0^\infty \int_0^\infty e^{i\tilde{\psi}(t, s_1,s_2,\omega)}
\tilde{A}(t,x,s_1,s_2,\omega)\cdot \omega \, ds_1\, ds_2 \, d\omega \, dx
\]
in which $\tilde{A}$ is a symbol in $\omega$.

If we forget the restriction on the domain for $(s_1,s_2)$, this is a standard
oscillatory integral and the phase has a smooth submanifold of fixed point. The
restriction on the domain makes it a little less standard. Although we could
perform a general treatment for this kind of oscillatory integrals, in our case,
the nature of the phase allows for a more direct computation.

We first make the change of variables $u=s_1+s_2,~v=s_1-s_2.$ In these coordinates, $\tilde \psi$ is independent of $v$; we write $\tilde \psi(t, u, \omega)$ for the phase expressed in these coordinates. Notice that, by \eqref{psitilde}, it is a smooth function of $u^2$, and is stationary in $u$ only at $u=0$.
The domain of
integration becomes $u\geq 0$ and $-u\leqslant v\leqslant u.$ We obtain the
integral
\[
I(t) = \int_0^\infty \int_0^\infty e^{i\tilde{\psi}(t, u,\omega)}
\left[\frac{1}{2}\int_{-u}^u \tilde{A}(u,v,\omega) dv\right] du \, d\omega.
\]
Since the factor in square brackets vanishes at $u=0$, the leading contribution of this integral is obtained by performing an integration by parts in $u$. To do this we write
$$
\tilde \psi(t, u, \omega) = (t-L)\omega + \tilde {\tilde \psi}(t, u, \omega), \quad
\tilde {\tilde \psi}(t, u, \omega) = O(u^2), \ u \to 0.
$$
Then we have $\lim_{u\rightarrow 0} \frac{\partial_u
  \tilde {\tilde{\psi}}(u,\omega)}{u}=\partial^2\tilde{\psi}(0,\omega) \neq 0.$
We obtain
\[
I(t) \sim i\int_0^\infty e^{i(t-L)\omega}
\tilde{A}_0(0,0,\omega)(\partial_u^2 \tilde{\psi}(0,\omega))^{-1} d\omega
\]
where the index $0$ means that we have taken the principal part of $\tilde{A}.$

It remains to evaluate all the quantities in our case observing that when
$s_1,s_2$ go to $0,$ $\theta_1$ go to $0$ and $\theta_2$ go to $\pi.$ Using \eqref{Salpha-expr} and \eqref{psitilde}, we have
\begin{equation*}
  \begin{aligned}
    \lim_{\theta_1 \rightarrow 0} \sin \theta_1
    S_{\alpha_1}(-\pi-\theta_1)& =\, \frac{1}{2\pi}\\
    \lim_{\theta_2 \rightarrow \pi} \sin \theta_2
    S_{\alpha_2}(\theta_2)& =\, -\frac{1}{2\pi}\\
    \partial^2_u \tilde{\psi}(0,\omega) & = \omega\cdot \frac{L}{b(L-b)}.
  \end{aligned}
\end{equation*}
Putting everything together we obtain:
\begin{equation*}
  \begin{aligned}
    \sigma_{\rho}(t) & \,\sim \int_0^\infty e^{i\omega (L-t)}
    \frac{\sqrt{b(L-b)}}{4\pi^2 L} d\omega \cdot \int_0^L
    \rho(x,0) dx\\
    & \,\sim \frac1{i} \frac{\sqrt{b(L-b)}}{4\pi^2 L} \cdot (t-L-i0)^{-1} \cdot \int_0^L
    \rho(x,0) dx
  \end{aligned}
\end{equation*}

The contribution of the whole periodic orbit is obtained by using a covering
argument (i.e. choosing carefully near each point $A_1,A_2$ and $\rho$ so that
in the end $\sum \rho$ is identically $1$ in a neighbourhood of the
geodesic). In the process, we have to be careful near the cone point. The
contribution of a (small) neighbourhood of the cone point can be computed using
the following trick (that is already used in \cite{Hil} and \cite{Wun}).

Suppose $\rho_c$ is a function that is identically $1$ near $\bmp_1.$ We want to
compute
\[
\sigma_{\rho_C}(t)\,\defeq\,\Tr(\calU(t)\rho_c)
\]
We insert microlocal cutoffs so that, up to a smooth remainder we have
\[
\sigma_{\rho_C}(t)=\Tr(A_1\calU(t-t_0)A_2U(t_0)\rho_c).
\]
Using the cyclicity of the trace we need to calculate
\[
\sigma_{\rho_C}(t)=\Tr(U(t_0)\rho_c A_1\calU(t-t_0)A_2).
\]
In the latter expression, thanks to the cutoffs, all the operations (composition
and taking the trace) take place away of the conical point. So we can proceed as
before.

In the end, if we sum all the contributions, it will amount to sum all the
contributions $\int \rho(x,0) dx$ and this will give the length of the geodesic.

We obtain the following proposition.

\begin{proposition}
  On a ESCS, the leading contribution to the wave trace of an isolated periodic
  diffractive orbit with two geometric diffraction is
  \[
 \frac{1}{4i\pi^2} \cdot \sqrt{b(L-b)} \cdot (t-L-i0)^{-1}.
  \]
\end{proposition}

\begin{remark}
As a point of comparison, we recall the analogous leading-order contribution of
a nondegenerate closed orbit $\gamma$ on a compact, smooth manifold in the trace
theorem of Duistermaat and Guillemin \cite{DuiGui}:
\begin{equation*}
  (2\pi)^{-1} L \, i^{\sigma_{\gamma}} \, \big| \Id - P_\gamma
  \big|^{-\frac{1}{2}}(t-L-i0)^{-1} .
\end{equation*}
Here, $P_\gamma$ is the Poincar\'e return map in the directions transverse to
the level set of the symbol and to the flow direction, and $i^{\sigma_{\gamma}}$
is a Maslov factor (with $\sigma_\gamma$ the Morse index of the geodesic). The
singularity we obtain here from an isolated periodic orbit with two geometric diffractions is thus of the same order.

We can also compare this with the singularity contributed by a non-geometric diffractive periodic orbit with one diffraction, as computed in \cite[Theorem 2]{Hil}. This has leading singularity $(t-L-i0)^{-1/2}$ and is hence one half order more regular. On the other hand, the singularity contributed by a cylinder of periodic geodesics is to leading order $(t-L-i0)^{-3/2}$, from op. cit. which is half an order more singular.

Notice that a cylinder of periodic geodesics necessarily has geometrically diff\-racted geodesics at its boundary.
In the second article in this series, we intend to use the analysis of the present paper to compute higher order terms in the wave trace singularity arising from such a cylinder.
\end{remark}

\appendix


\section{Domains of operators and admissible asymptotics at the cone point}
\label{sec:proof-asymptotics-cone-point}

In the course of our construction of the wave propagators on $C_{4\pi}$, we
needed information about the domains of operators related to the
Laplace-Beltrami operator $\Delta_g$.  The first such result was a description
of the domain of the adjoint operator $\overline{\Delta_g}{}^*$.

\begin{lemma}
  \label{thm:domain-adjoint-Lap}
  Let $\rho \in \calC^\infty((0,\infty)_r)$ be a smooth cutoff satisfying
  $\rho \equiv 1$ for $r \leqslant 1$ and $\rho \equiv 0$ for $r \geqslant 2$.
  Then the domain of $\left(\overline{\Delta_g}\right)^*$ as an unbounded
  operator on $L^2(C_{4\pi})$ is
  \begin{equation}
    \label{eq:domain-adjoint-Lap}
    \overline{\mathfrak{D}}{}^* = \overline{\mathfrak{D}} \oplus \Span_\bbC \!
    \left\{ \rho, \rho \log(r), \rho \, r^\frac{1}{2} \exp\!\left[
        \pm \frac{i\theta}{2} \right], \rho \, r^{-\frac{1}{2}} \exp\!\left[ \pm
        \frac{i\theta}{2} \right]
    \right\} .
  \end{equation}
\end{lemma}

\begin{proof}
  Using the symmetry of $\Delta_g$, we may decompose
  $\overline{\mathfrak{D}}{}^*$ as
  \begin{equation*}
    \overline{\mathfrak{D}}{}^* = \overline{\mathfrak{D}} \oplus \Null(\Delta_g -
    \alpha_1) \oplus \Null(\Delta_g - \alpha_2)
  \end{equation*}
  for any distinct $\alpha_1$ and $\alpha_2$ lying outside the spectrum of
  $\Delta_g$ (cf.~\cite{ReeSim2}).  Moreover, nonnegativity of $\Delta_g$
  implies that it is sufficient to let $\alpha_j = - \beta_j^2$ for distinct
  choices of $\beta_j$.  Thus, let us suppose that $u$ is an element of
  $\Null(\Delta_g + \beta^2)$, i.e.,
  \begin{equation}
    \label{eq:adjlemma-1}
    \left( \Delta_g + \beta^2 \right) u(r,\theta) = - \frac{1}{r^2} \left[ r^2
      \del_r^2 + r \del_r - \left( \beta^2 r^2 - \del_\theta^2 \right) \right]
    u(r,\theta) = 0 .
  \end{equation}
  By separating variables using the spectral projectors
  \eqref{eq:angular-spectral-projectors}, we may rewrite $u$ as a Fourier series
  of the form
  \begin{equation*}
    u(r,\theta) = \frac{1}{\sqrt{4\pi}} \sum_{j \in \bbZ}
    \hat{u}_{j}(r) \exp\!\left[ i \, \frac{j}{2} \, \theta \right] .
  \end{equation*}
  Then the quality \eqref{eq:adjlemma-1} implies the corresponding equality
  \begin{equation}
    \label{eq:adjlemma-2}
    \left( L_j + \beta^2 \right) \hat{u}_j(r) \defeq - \frac{1}{r^2} \left[ r^2
      \del_r^2 + r \del_r - \left( \beta^2
        r^2 + \frac{j^2}{4} \right) \right] \hat{u}_j(r) = 0 .
  \end{equation}
  Introducing the change of variables $s = \beta r$ into \eqref{eq:adjlemma-2},
  this differential equation becomes
  \begin{equation}
    \label{eq:adjlemma-3}
    - \frac{\beta^2}{s^2} \left[ s^2 \del_s^2 + s \del_s - \left( s^2 +
        \frac{j^2}{4} \right) \right] \hat{u}_j(\beta^{-1} s) = 0 ,
  \end{equation}
  which is the modified Bessel equation, up to the overall factor of
  $- \frac{\beta^2}{s^2}$.  Thus, the Fourier coefficients $\hat{u}_j$ must be
  linear combinations of the modified bessel functions $I_{\frac{j}{2}}(s)$ and
  $K_{\frac{j}{2}}(s)$.

  Now, observe that the condition that our original function $u$ is an element
  of $L^2(C_{4\pi})$ forces each of the Fourier coefficients
  $\hat{u}_j(\beta^{-1}s)$ to be elements of the function space $L^2( (0,\infty)_s , s ds)$.
  Indeed, the Fourier decomposition in $\theta$ induces a factoring
  \begin{equation*}
    L^2(C_{4\pi}) = \ell^2\!\left( \bbZ ; L^2((0,\infty)_r , r dr) \right) ,
  \end{equation*}
  and our change of variables identifies $L^2((0,\infty)_r, r dr)$ with
  $L^2((0,\infty)_s, s ds)$.  This implies that the only admissible solutions to
  \eqref{eq:adjlemma-3} are
  \begin{equation*}
    \hat{u}_0( \beta^{-1} s) = K_0(s), \quad \hat{u}_{\pm 1}(\beta^{-1} s) =
    K_{\frac{1}{2}}(s), \quad \text{and} \quad \hat{u}_j(\beta^{-1} s) = 0
    \text{ for $|j| \geqslant 2$} .
  \end{equation*}
  These are the only modified Bessel functions which are globally in
  $L^2((0,\infty)_s,s ds)$, as may be easily gleaned from their asymptotics as
  $s \to 0$ and $s \to \infty$ in \cite{AbrSte}.  Hence,
  \begin{equation}
    \label{eq:adjlemma-4}
    \Null(\Delta_g + \beta^2) = \Span_\bbC \! \left\{ K_0(\beta r),
      K_{\frac{1}{2}}(\beta r) \exp\!\left[ \frac{i\theta}{2} \right],
      K_{\frac{1}{2}}(\beta r) \exp\!\left[ - \frac{i\theta}{2} \right] \right\} .
  \end{equation}

  Let $\rho \in \calC^\infty((0,\infty)_r)$ be a cutoff as in the statement of
  the lemma, and observe that
  \begin{equation*}
    \left[ 1 - \rho(r) \right] K_0(\beta r) \quad \text{and} \quad \left[ 1 -
      \rho(r) \right] K_{\frac{1}{2}}(\beta r)
  \end{equation*}
  are both Schwartz in $r$ and vanish at the cone point.  This shows they are
  elements of $\overline{\mathfrak{D}}$, which in turn implies that
  $\overline{\mathfrak{D}}{}^*$ is equal to
  \begin{equation*}
    \overline{\mathfrak{D}} \oplus \Span_\bbC
    \left\{ \rho \, K_0(\beta_1 r), \rho \, K_{\frac{1}{2}}(\beta_1 r)
      \exp\!\left[ \pm \frac{i\theta}{2} \right], \rho \, K_0(\beta_2 r), \rho
      \, K_{\frac{1}{2}}(\beta_2 r)
      \exp\!\left[ \pm \frac{i\theta}{2} \right] \right\}
  \end{equation*}
  for any two distinct choices of $\beta_j > 0$.  Similarly, since
  \begin{equation*}
    K_0(x) = - \log\!\left( \frac{x}{2} \right) - \gamma + \mathrm{O}(x) \text{
      as $x \to 0$}
    \quad \text{and}
    \quad K_{\frac{1}{2}}(x) = \left( \frac{\pi}{2 x} \right)^{\frac{1}{2}} e^{-x} ,
  \end{equation*}
  where $\gamma$ is the Euler-Mascheroni constant and $\Gamma(z)$ is the
  $\Gamma$-function, we have that
  \begin{multline*}
    \Span_\bbC \left\{ \rho \, K_0(\beta_1 r), \rho \, K_{\frac{1}{2}}(\beta_1
      r) \exp\!\left[ \pm \frac{i\theta}{2} \right], \rho \, K_0(\beta_2 r),
      \rho \, K_{\frac{1}{2}}(\beta_2 r)
      \exp\!\left[ \pm \frac{i\theta}{2} \right] \right\} \\
    \mbox{} \equiv \Span_\bbC \!  \left\{ \rho, \rho \log(r), \rho \,
      r^\frac{1}{2} \exp\!\left[ \pm \frac{i\theta}{2} \right], \rho \,
      r^{-\frac{1}{2}} \exp\!\left[ \pm \frac{i\theta}{2} \right] \right\}
    \mod{\overline{\mathfrak{D}}} .
  \end{multline*}
  This concludes the proof.
\end{proof}

The other piece of information about domains we needed was the expansion of
elements of the Friedrichs domain $\calD_2$ at the cone point given in
Lemma~\ref{thm:distr-at-cone-exp}.  We now prove this lemma.

\begin{proof}[Proof of Lemma~\ref{thm:distr-at-cone-exp}]
  The Friedrichs domain $\calD_2$ is characterized as the subspace of
  $\overline{\mathfrak{D}}{}^*$ which is included in the Dirichlet form domain
  associated to $\Delta_g$, i.e., those distributions $u$ which are bounded in
  \begin{equation*}
    Q_{\Delta_g}(u) = \left<u,u\right>_{L^2} +  \left< u , \Delta_g u
    \right>_{L^2} .
  \end{equation*}
  As the Dirichlet form domain is precisely $H^1(C_{4\pi})$, we may conclude
  from the description \eqref{eq:domain-adjoint-Lap} of
  $\overline{\mathfrak{D}}{}^*$ that
  \begin{equation*}
    \calD_2 = \overline{\mathfrak{D}} \oplus \Span_\bbC \! \left\{ \rho,  \rho \,
      r^\frac{1}{2} \exp\!\left[ \frac{i\theta}{2} \right],  \rho \,
      r^\frac{1}{2} \exp\!\left[- \frac{i\theta}{2} \right] \right\}
  \end{equation*}
  since these are the only elements of
  $\overline{\mathfrak{D}}{}^* \big/ \overline{\mathfrak{D}}$ which are elements
  of $H^1(C_{4\pi})$.  The lemma follows.
\end{proof}

\section{Geometric theory of diffraction}\label{sec:GTD}
In this appendix we proceed with a construction of the kernel of the wave
propagator that allows to compute explicitly the symbol on both Lagrangians
$\Lambda^\geom$ and $\Lambda^\diff$ away of their intersection. For the
diffracted part, this is known in the literature as the {\em geometric theory of
  diffraction} \cite{Keller} and we provide an interpretation of this construction based on
scattering of waves on the cone $C_\alpha$.

At the direct front, the symbol is just as it is on $\RR^2$.  Recall that, on
$\RR^2$, the half-wave kernel as a \emph{distributional half-density} is
\begin{equation}
  (2\pi)^{-2} \int e^{i((x-y) \cdot \xi - t |\xi|)} \, d\xi \big| dx
  dy\big|^{\frac{1}{2}}.
\end{equation}
Let $\vec{e}_1$ be a unit vector in the plane pointing from $x$ to $y$, and let
$(\vec{e}_1, \vec{e}_2)$ be an oriented orthonormal basis. We write
$\xi = \omega \vec{e}_1 + \rho \vec{e}_2$. Then the integral can be written
\begin{equation}
  (2\pi)^{-2} \int e^{i(|x-y| \omega - t \sqrt{\omega^2 + \rho^2})} \, d\rho \, d\omega \big| dx dy\big|^{\frac{1}{2}}.
\end{equation}
Assume $t > 0$.  Then there are stationary points on the line
$\left\{ \rho = 0, \omega > 0 \right\}$. We can integrate out $\rho$, and to
leading order (that is, replacing the expression with the leading term in the
stationary phase expansion at $\rho = 0$) we get
\begin{equation}
  (2\pi)^{-\frac{3}{2}} \int e^{i(|x-y| -t) \omega} \chi(\omega)
  e^{-\frac{i\pi}{4}} \left( \frac{\omega}{t} \right)^\frac{1}{2} \, d\omega
  \big| dx dy\big|^{\frac{1}{2}}
  \label{2dwavekernel}
\end{equation}
where $\chi \in \calC^\infty(\bbR)$ is zero for $\omega < 1$ and $1$ for
$\omega \geq 2$. Thus, the principal symbol of this distribution at
$N^* \{ |x-y| = t \}$, for $t$ fixed, is
\begin{equation}
  e^{-\frac{i\pi}{4}} \chi(\omega) \left( \frac{\omega}{t} \right)^\frac{1}{2}
  \big| dy ds d\omega \big|^{\frac{1}{2}}
\end{equation}
for $s$ the arc length along the circle $\{ |x-y| = t \}$ and $\omega$ the
cotangent variable dual to $|x-y| - t$.

We now return to $C_\alpha$, the cone of angle $\alpha$, and we restrict our
attention to $t > 0$.  On the diffracted front $\Lambda^\diff$ and away from the
direct front, the half-wave kernel $\bmU(t)$ takes the oscillatory integral form
\begin{equation}
  (2\pi)^{-\frac{3}{2}} \int e^{i(r+r'-t) \cdot \omega} \, K(r, \theta; r',
  \theta'; \omega) \left|r dr d\theta \, r' dr' d\theta' \right|^\frac{1}{2} .
  \label{diff-front}
\end{equation}

The amplitude $K(r, \theta; r', \theta'; \omega)$ is a symbol of order 0 in
$\omega$, as follows from the kernel $\bmU(t)$ being of order $-\frac{1}{2}$
(for each fixed $t$) at $\Lambda^\diff$.  We consider how this part of the
propagator acts on a particular initial condition. Consider the exact
solution\footnote{This solution is an example of the ``plane waves'' arising out
  of Cheeger's functional calculus.} to the half-wave equation given by
\begin{equation}
  \label{eq:plane-wave-solution}
  \sqrt{\frac{\pi}{2}} \Big( \int e^{-i\lambda t^*} J_{|\nu|}(\lambda r)
  \tilde\chi(\lambda) \, d\lambda  \Big) e^{i\nu \theta},
\end{equation}
where $\tilde\chi \in \calC^\infty(\bbR)$ is supported in $[2, \infty)$ and
identically $1$ for $\lambda \geq 4$, say, and where
$\nu = \frac{2\pi\ell}{\alpha}$ for some integer $\ell$. This distribution
\eqref{eq:plane-wave-solution} is conormal to $\{ r = -t \}$ for $t < 0$ and to
$\{ r = t \}$ for $t > 0$, as can be seen by using the expansion for the Bessel
function as its argument gets large:
\begin{equation}
  J_{|\nu|}(z) = \sqrt{\frac{2}{\pi z}} \cos \! \left(z - \frac{|\nu| \pi}{2} -
    \frac{\pi}{4} \right) \sum_{j=0}^\infty a_j(z)
  \label{Bessel-expansion}
\end{equation}
where $a_j \in S^{-j}(\bbR_z)$ are the homogeneous terms in the expansion in $z$
with $a_0(z) \equiv 1$.  Therefore, up to a smooth error, the solution
\eqref{eq:plane-wave-solution} has the form
\begin{equation*}
  e^{i\nu \theta} \int e^{-i \lambda t} \Big( e^{i(\lambda r - |\nu|\pi/2 -
    \pi/4)} + e^{-i(\lambda r - \nu\pi/2 - \pi/4)} \Big) a(\lambda r) (\lambda
  r)^{-\frac{1}{2}} \, d\lambda
\end{equation*}
for $a \in S^0(\bbR)$.  The singularities for $t < 0$, say $t = -t_*$, take the
form
\begin{equation}
  (2\pi)^{-\frac{1}{2}}  e^{i\nu \theta} \int e^{-i (\lambda (r - t_*) -
    |\nu|\pi/2 - \pi/4)}
  a(\lambda r) (\lambda r)^{-\frac{1}{2}} \, d\lambda,
  \label{t=-t*}\end{equation}
and for $t > 0$, say $t = +t_*$,
\begin{equation}
  (2\pi)^{-\frac{1}{2}} e^{i\nu \theta} \int e^{i (\lambda (r - t_*) -
    |\nu|\pi/2 - \pi/4)}  a(\lambda r) (\lambda r)^{-\frac{1}{2}} \, d\lambda.
  \label{t=t*}\end{equation}
On the other hand, if we apply the wave kernel $e^{-2it_*\sqrt\Delta}$ to the
initial condition \eqref{t=-t*} we obtain \eqref{t=t*} up to smooth terms. The
direct front, away from the diffracted wave, does not contribute for $t > t_*$,
and the singularities come purely from the diffracted front (except at the
intersection, where $\theta - \theta' \equiv \pm \pi \mod{\alpha}$). Away from
the direct front, applying \eqref{diff-front} to \eqref{t=-t*} gives us
\begin{multline*}
  (2\pi)^{-\frac{3}{2}} \int e^{i(r+r' - 2t_*)\omega} K(r,\theta; r', \theta';
  \omega) e^{i(-\lambda (r'-t_*) + |\nu| \pi/2 + \pi/4)} (\lambda
  r')^{-\frac{1}{2}} a(\lambda
  r')e^{i\nu \theta'} \\
  \mbox{} \times r' \, dr' \, d\theta' \, d\omega
\end{multline*}
and after applying stationary phase in the $(r', \omega)$-variables we obtain
\begin{equation*}
  (2\pi)^{-\frac{1}{2}} \int e^{i(r-t_*)\lambda} (\lambda r)^{-\frac{1}{2}}
  \chi(\lambda) \left\{
    e^{i|\nu| \pi/2} e^{i\pi/4} \int (rr')^{\frac{1}{2}} K(r, \theta; r', \theta';
    \omega) e^{i \nu \theta'} \, d\theta' \right\} d\lambda .
\end{equation*}
This must yield \eqref{t=t*} up to smooth terms. Therefore the leading-order
part of the amplitude $S$, viewed as a Schwartz kernel in the
$(\theta,\theta')$-variables, maps $e^{i\nu \theta}$ to the quantity
\begin{equation*}
  (2\pi)^{\frac{1}{2}}e^{-i\pi/2} e^{-i |\nu| \pi} (rr')^{-\frac{1}{2}} e^{i\nu
    \theta} .
\end{equation*}
Hence, the principal part of $K$ corresponds to the operator
\begin{equation*}
  -i (2\pi)^{\frac{1}{2}} (rr')^{-\frac{1}{2}} e^{-i\pi
    \sqrt{\smash[b]{\Delta_{\bbS^1_\alpha}}}},
\end{equation*}
since $e^{i\nu \theta}$ is an eigenfunction of
$\sqrt{\smash[b]{\Delta_{\bbS^1_\alpha}}}$ with eigenvalue $|\nu|$.

We can compare this principal part to the absolute scattering matrix
$S(\lambda)$ for the cone $C_\alpha$. This is, by definition, the map from the
``incoming boundary data'' of generalized eigenfunctions of
$\sqrt{\smash[b]{\Delta_{\bbS^1_\alpha}}}$ with eigenvalue $\lambda$, to the
``outgoing boundary data''. These are the coefficients of $e^{-i\lambda r}$,
respectively $e^{+i\lambda r}$, in the expansions of the generalized
eigenfunction as $r \to \infty$. By inspection of the generalized eigenfunctions
$$
J_\nu(\lambda r) e^{i \nu \theta},
$$
and \eqref{Bessel-expansion}, we see that this operator is
$-i e^{-i\pi \sqrt{\smash[b]{\Delta_{\bbS^1_\alpha}}}}$. Hence, we obtain at
leading order and provided $\theta - \theta' \not\equiv \pm \pi \mod{\alpha}$
\[
K(r,\theta,r',\theta') \sim
(2\pi)^{\frac{1}{2}}(rr')^{-\frac{1}{2}}S_\alpha(\theta-\theta')
\]
where $S_\alpha(\theta-\theta')$\footnote{We have used the invariance by
  rotation to write this kernel in this form.} is the kernel of the absolute scattering
matrix for the cone of angle $\alpha$ or, equivalently the kernel of
$-ie^{-i\pi\Delta_{\bbS^1_\alpha}}$

The principal symbol of the diffracted wave is therefore the leading-order part
of
\begin{multline*}
  K(r, \theta; r', \theta'; \omega )|dr d\theta d\theta' d\omega|^{\frac{1}{2}}
  \left| \frac{\partial (r, \theta; r', \theta'; \omega)}{\partial (x, y,
      \omega)} \right|^{-\frac{1}{2}} \\
  \equiv -i \sqrt{2\pi} (rr')^{-\frac{1}{2}} \calK\!\left[
    e^{-i\pi\sqrt{\smash[b]{\Delta_{\bbS^1_\alpha}}}} \right] (\theta, \theta')
  \cdot (rr')^{\frac{1}{2}} |dr d\theta d\theta' d\omega|^{\frac{1}{2}}
  \mod{S^{-1}} ,
\end{multline*}
which after simplification is
\begin{equation}
  \sqrt{2\pi} \, S_\alpha(\theta-\theta') |dr d\theta d\theta'
  d\omega|^{\frac{1}{2}} .
  \label{Srr'}
\end{equation}

The distribution $S_\alpha$ can be computed using Fourier series.  Indeed, since
it is the kernel of $-ie^{-i\pi \Delta_{\bbS^1_\alpha}}$ we have
\begin{equation}
  \begin{aligned}
    S_\alpha(\theta) & = \frac{-i}{\alpha}\sum_{k\in \bbZ}
    e^{-i\pi \big|\frac{2k\pi}{\alpha}\big|}e^{-\frac{2ik\pi}{\alpha}\theta}\\
    &= \frac{-i}{\alpha} \Big[ 1+\sum_{k\geq 1}
    e^{-\frac{2ik\pi}{\alpha}(\pi-\theta)}+\sum_{k\geq 1}
    e^{-\frac{2ik\pi}{\alpha}(\pi+\theta)}\Big]\\
    &=\frac{-i}{\alpha}\Big [ 1+
    \frac{e^{-\frac{2i\pi}{\alpha}(\pi-\theta)}}{1-e^{-\frac{2i\pi}{\alpha}(\pi-\theta)}}+
    \frac{e^{-\frac{2i\pi}{\alpha}(\pi+\theta)}}{1-e^{-\frac{2i\pi}{\alpha}(\pi+\theta)}}
    \Big]\\
    &=\frac{-i}{\alpha}\Big [1 +
    \frac{e^{-\frac{i\pi}{\alpha}(\pi-\theta)}}{2i\sin\big(\frac{\pi}{\alpha}(\pi-\theta)
      \big)}+
    \frac{e^{-\frac{i\pi}{\alpha}(\pi+\theta)}}{2i\sin\big(\frac{\pi}{\alpha}(\pi+\theta
      )\big)}\Big]\\
    & = \frac{-1}{2\alpha} \frac{\sin(\frac{2\pi^2}{\alpha}
      )}{\sin\big(\frac{\pi}{\alpha}(\pi-\theta)\big) \sin\big(\frac{\pi}{\alpha}(\pi+\theta)\big)}
  \end{aligned}
  \label{Salpha-expr}
\end{equation}

In the case $\alpha=4\pi$, this simplifies to
\begin{equation}\label{eq:S4pi}
  S_{4\pi}(\theta)=\frac{-1}{8\pi}\cdot \frac{1}{\sin(\frac{\pi-\theta}{4})\sin(\frac{\pi+\theta}{4})}=\frac{-1}{4\pi}\cdot
  \frac{1}{\cos \frac{\theta}{2}}.
\end{equation}

Summarizing this computation we have the following proposition.
\begin{proposition}\label{prop:hwkgtd}
  Microlocally near the diffracted front $\Lambda^\diff$ and away from $\Sigma$, the
  leading part of the half-wave kernel $\bmU_\alpha$ on the cone of angle
  $\alpha$ is given by the following oscillatory integral (using polar
  coordinates)
  \begin{equation}\label{eq:hwkgtd}
    \bmU_\alpha(t,\bmq_1,\bmq_2)\sim \frac{1}{2\pi}\int_{\omega>0}
    e^{i(r_1+r_2-t)}
    (r_1r_2)^{-\frac{1}{2}}S_\alpha(\theta_1-\theta_2)\,d\omega \big| d\bmq_1d\bmq_2\big|^{\frac{1}{2}}
  \end{equation}
  with
  \begin{equation}\label{eq:Salpha}
    S_\alpha(\theta)=\frac{-1}{2\alpha}
    \frac{\sin(\frac{2\pi^2}{\alpha} )}{\sin\big(\frac{\pi}{\alpha}(\pi-\theta)\big)\sin\big(\frac{\pi}{\alpha}(\pi+\theta)\big)}
  \end{equation}
\end{proposition}

\begin{remark}
  This coincides with Theorem 4 in \cite{Hil} up to the factor $2$ that as been
  omitted there.
\end{remark}

\section*{References}
\label{sec:references}

\end{document}